\documentclass[11pt]{article}
\usepackage{amsmath,amssymb,amsthm,bm,graphicx,color,epsfig,enumerate,caption}
\usepackage{subfig}
\usepackage{tabularx}
\usepackage{hyperref}

\usepackage{algorithm}
\usepackage{algorithmic}

\newtheorem{theorem}{Theorem}[section]

\newtheorem{definition}[theorem]{Definition}

\newcommand{\R}{\mathbb{R}}

\newcommand{\grad}{\nabla}

\newcommand{\eps}{\epsilon}

\newcommand{\diag}{\operatorname{diag}}

\newcommand{\SNR}{\text{SNR}}

\newcommand{\VAR}{\text{Var}}

\begin{document}

\title{Statistical Analysis of Synchrosqueezed Transforms}

\author{Haizhao Yang\\\
  \vspace{0.1in}\\
  Department of Mathematics, Duke University
}

\date{October 2014, Revised: August 2016}
\maketitle

\begin{abstract}
Synchrosqueezed transforms are non-linear processes for a sharpened time-frequency representation of wave-like components. They are efficient tools for identifying and analyzing wave-like components from their superposition. This paper is concerned with the statistical properties of compactly supported synchrosqueezed transforms for wave-like components embedded in a generalized Gaussian random process in multidimensional spaces. Guided by the theoretical analysis of these properties, new numerical implementations are proposed to reduce the noise fluctuations of these transforms on noisy data. A MATLAB package SynLab together with several heavily noisy examples is provided to support these theoretical claims.

\end{abstract}



\section{Introduction}
\label{sec:intro}
Non-linear and non-stationary wave-like signals (also termed as chirp signals in $1$D) are ubiquitous in science and engineering, e.g., clinical data \cite{HauBio2,HauBio1}, seismic data \cite{seismic1,seismic2,GeoReview}, climate data \cite{SSRobust,PhysicalAnal}, astronomical data \cite{Chirplet1,Chirplet2}, materials science \cite{2DCrystal,SingerSinger:06,2DGWS}, and art investigation \cite{Canvas}. 
Analyzing instantaneous properties (e.g., instantaneous frequencies, instantaneous amplitudes and instantaneous phases \cite{Boashash92,Picinbono97})  or local properties (concepts for $2$D signals similar to ``instantaneous" in $1$D) of signals has been an important topic for over two decades. In many 
applications \cite{Daubechies2011,Eng2,ENG1,Hau-Tieng2013,SSWPT,SSCT}, these signals can be modeled as a superposition of several wave-like components with slowly varying amplitudes, frequencies or wave vectors, contaminated by noise. For example, a complex signal
\begin{equation}
\label{P1}
f(x)=\sum_{k=1}^K \alpha_k(x) e^{2\pi i N_k \phi_k(x)} + e(x),
\end{equation}
where $\alpha_k(x)$ is the instantaneous (local) amplitude, $N_k \phi_k(x)$ is the instantaneous (local) phase, $N_k\phi_k'(x)$ is the instantaneous frequency (or $N_k\grad\phi_k(x)$ as the local wave vector), and $e(x)$ is a noisy perturbation term. 

A powerful tool for analyzing signal \eqref{P1} is the synchrosqueezed transform ({SST}) consisting of a linear  time-frequency analysis tool and a synchrosqueezing technique. It belongs to more generally time-frequency reassignment techniques \cite{Auger1995,Chassande-Mottin2003,Chassande-Mottin1997,Daubechies1996} (see also \cite{TFReview} for a recent review). The SST was initialized in \cite{Daubechies1996} and further analyzed in \cite{Daubechies2011} for the $1$D wavelet transform. Suppose $W_f(\xi,x)$ is the wavelet transform of a $1$D wave-like component $f(x)=\alpha(x)e^{2\pi i N\phi(x)}$. The wavelet time-frequency representation has a wide support spreading around the instantaneous frequency curve $(N\phi'(x),x)$. It was proved that the instantaneous frequency information function $v_f(\xi,x)=\frac{\partial_x W_f(\xi,x)}{2\pi i W_f(\xi,x)}$ is able to approximate $N\phi'(x)$. Hence, the synchrosqueezing technique shifts the value of $W_f(\xi,x)$ from $(\xi,x)$ to $(v_f(\xi,x),x)$, generating a sharpened time-frequency representation $T_f(v,x)$ with a support concentrating around the curve $(N\phi'(x),x)$. The localization of the new representation not only improves the resolution of the original spectral analysis due to the uncertainty principle but also make it easier to decompose the superposition in \eqref{P1} into individual components.

 A variety of {SST}s have been proposed after the $1$D synchrosqueezed wavelet transform (SSWT) in \cite{Daubechies1996}, e.g., the synchrosqueezed short time Fourier transform (SSSTFT) in \cite{SSSTFT}, the synchrosqueezed wave packet transform (SSWPT) in \cite{1DSSWPT,SSWPT}, the synchrosqueezed curvelet transform (SSCT) in \cite{SSCT} and the $2$D monogenic synchrosqueezed wavelet transform in \cite{2Dwavelet}. Rigorous analysis has proved that these transforms can accurately decompose a class of superpositions of wave-like components and estimate their instantaneous (local) properties if the given signal is noiseless. To improve the synchrosqueezing operator in the presence of strongly non-linear instantaneous frequencies,  some further methods have been proposed in \cite{TFReview,SSGeneral,Oberlin} based on an extra investigation of the higher order derivatives of the phase of a wave-like component. All these {SST}s are compactly supported in the frequency domain to ensure accurate estimations. To better analyze signals with a trend or big data sets that need to be handled in real-time computation, a recent paper \cite{Hau2014} proposes a new synchrosqueezing method based on carefully designed wavelets with sufficiently many vanishing moments and a minimum support in the time domain. Previous synchrosqueezed transforms need to access every data sample of $f(x)$ to compute the synchrosqueezed transform at a specific time or location $x$. However, compactly supported synchrosqueezed transforms only need a small portion of the data samples in a neighborhood of $x$. Hence, compactly supported synchrosqueezed transforms are computationally more efficient and are better tools for modern big data analysis. However, mathematical analysis on the accuracy of this compactly supported SSWT is still under development. This paper addresses this problem in the framework of the SSWPT that includes the SSWT as a special case.
 
Another important topic in the study of {SST}s is the statistical analysis of the synchrosqueezing operator, since noise is ubiquitous in real applications. A pioneer paper \cite{SSRobust3} in this direction studied the statistical properties of the $1$D spectrogram reassignment method by calculating the probability density function of white Gaussian noise after reassignment. A recent paper \cite{SSRobust} focused on the statistical analysis of the $1$D SSWT for white Gaussian noise. It estimated the probability of a good estimation of the instantaneous frequency provided by the instantaneous frequency information function $v_f(\xi,x)$. A following paper \cite{SSRobust2} generalized its results to generalized stationary Gaussian process. To support the application of {SST}s to real-time problems and multidimensional problems, this paper analyzes the statistical properties of multidimensional {SST}s that can be compactly supported in the time domain. 

Turning to the robustness issue in a numerical sense, it is of interest to design an efficient implementation of a sharpened time-frequency representation with reduced noise fluctuations. The idea of multitapering, first proposed in \cite{Taper1} for stationary signals and further extended in \cite{Taper2,Taper3} for non-stationary signals, attempted to improve the statistical stability of spectral analysis by generating multiple windowed trials of the noisy signal and averaging the spectral analysis of these trials. By combining the multitapering and time-frequency reassignment techniques, \cite{Taper4} proposed the $1$D multitapering time-frequency reassignment for a sharpened time-frequency representation with reduced noise fluctuations. Since its implementation is based on Hermite functions, its efficient generalization in multidimensional spaces is not straightforward. Guided by the theoretical analysis of the statistical properties of {SST}s, this paper proposes efficient numerical implementations of multidimensional {SST}s based on highly redundant frames. Since the {SST} of a noiseless signal is frame-independent, averaging the SSTs from multiple time-frequency frames reduces the noise fluctuation while keeping the localization of the synchrosqueezed time-frequency representation.

The rest of this paper is organized as follows. In Section \ref{sec:thm}, the main theorems for the compactly supported {SST}s in multidimensional spaces and their statistical properties are presented. In Section \ref{sec:numeric}, a few algorithms and their implementations are introduced in detail to improve the statistical stability of {SST}s. Several numerical examples with heavy noise are provided to demonstrate the proposed properties. We conclude this paper in Section \ref{sec:conclusion}.

\section{Theory for synchrosqueezed transforms ({SST}s)}
\label{sec:thm}
Let us briefly introduce the basics and assumptions of compactly supported {SST}s in Section \ref{subsec:comp} before discussing their statistical properties in Section \ref{subsec:stat}. While the synchrosqueezing technique can be applied to a wide range of time-frequency transforms, the discussion here is restricted to the framework of multidimensional wave packet transforms to save space. It is easy to extend these results to other transforms (see \cite{Thesis} for the example of the $2$D synchrosqueezed curvelet transform). Due to the space limitation, only the main ideas of the proofs are presented. Readers are referred to \cite{Thesis} for more details. 

\subsection{Compactly supported {SST}s}
\label{subsec:comp}
Previously, the synchrosqueezed wave packet transform (SSWPT) was built using mother wave packets compactly supported in the frequency domain. This paper studies a wider class of mother wave packets defined below.

\begin{definition}
An $n$-dimensional mother wave packet $w(x)\in C^m(\R^n)$ is of type $(\eps,m)$ for some $\eps\geq 0$, and some non-negative integer $m$, if $\widehat{w}(\xi)$ is a real-valued smooth function with a support that covers the ball $B_1(0)$ centered at the origin with a radius $1$ satisfying that:
\[
|\widehat{w}(\xi)|\leq \frac{\eps}{(1+|\xi|)^m},
\]
for $|\xi|>1$
and $|\widehat{w}(\xi)|>0$, for $|\xi|<1$.
\end{definition}
Since $w\in C^m(\R^n)$, the above decaying requirement is easy to satisfy. Actually, we can further assume $\widehat{w}(\xi)$ is essentially supported in a ball $B_d(0)$ with $d\in(0,1]$ to adapt signals with close instantaneous frequencies, i.e., $|\widehat{w}(\xi)|$ is approximately zero outside this support up to an $\epsilon$ truncation error. However, $d$ is just a constant in later asymptotic analysis. Hence, we omit this discussion and consider it as $1$ in the analysis but implement it in the numerical tool. Similarly to the discussion in \cite{1DSSWPT,SSWPT}, we can use this mother wave packet $w(x)$ to define a family of $n$-dimensional compactly supported wave packets 
  \[
  w_{a b}(x)=|a|^{ns/2} w\left(|a|^s(x-b)\right) e^{2\pi i (x-b)\cdot a},
  \]  
  through scaling, modulation, and translation, controlled by a geometric parameter $s$, for $a$, $b\in \R^n$. With this family of wave packets ready, we define the wave packet transform via
  \begin{align*}
    W_f(a,b) 
    &= \langle f,w_{a b}\rangle =  \int f(x) \overline{w_{a b}(x)}dx.
  \end{align*}

Previously in \cite{1DSSWPT,SSWPT}, the synchrosqueezed wave packet transform (SSWPT) was proposed to analyze a class of intrinsic mode type functions as defined below.

\begin{definition}
  \label{mSSWPT:def:IMTF}
  A function $f(x)=\alpha(x)e^{2\pi iN \phi(x)}$ in $\R^n$ is an intrinsic mode type
  function (IMT) of type $\left(M,N\right)$ if $\alpha(x)$ and $\phi(x)$ satisfy
  \begin{align*}
    \alpha(x)\in C^\infty, \quad |\grad \alpha(x)|\leq M, \quad 1/M \leq \alpha(x)\leq M, \\
    \phi(x)\in C^\infty,  \quad  1/M \leq |\grad \phi(x)|\leq M, \quad |\grad^2 \phi(x)|\leq M.
   \end{align*}
\end{definition}

It has been proved that the instantaneous frequency (or local wave vector when $n>1$) information function $v_f(a,b)=\frac{ \nabla_b W_f(a,b) }{ 2\pi i W_f(a,b)}$ of an IMT $f(x)$ can approximate $N\nabla \phi(b)$ if the mother wave packet $w(x)$ is of type $(0,\infty)$ and $N$ is sufficiently large. A careful inspection of previous proofs shows that the approximation $v_f(a,b) \approx N\nabla \phi(b)$ is still valid up to an $\epsilon$ relative error if the mother wave packet is of type $(\epsilon,m)$ for any positive integer. See Theorem 2.2.7 in \cite{Thesis} for a detailed proof. Hence, if we squeeze the coefficients $W_f(a,b)$ together based upon the same information function $v_f(a,b)$, then we would obtain a sharpened time-frequency representation of $f(x)$. This motivates the definition of the synchrosqueezed energy distribution 
  \begin{equation*}
    T_f(v,b) = \int_{\R^n} |W_f(a,b)|^2 \delta\left(\Re{v_f(a,b)}-v\right) da \label{eq:SED}
  \end{equation*}
  for $v,b\in \R^n$. Here $\delta$ denotes the Dirac delta function and $\Re v_f(a,b)$ means the real part of $v_f(a,b)$. 
  
For a multi-component signal $f(x)=\sum_{k=1}^K \alpha_k(x)e^{2\pi iN_k \phi_k(x)}$, the synchrosqueezed energy of each component will also concentrate around each  $N_k\nabla \phi_k(x)$ if these components satisfy the well-separation condition defined below.

\begin{definition}
  \label{mSSWPT:def:SWSIMC}
  A function $f(x)$ is a well-separated superposition of type
  $\left(M,N,K,s\right)$ if
  \[
  f(x)=\sum_{k=1}^K f_k(x),
  \] 
  where each $f_k(x)=\alpha_k(x)e^{2\pi iN_k \phi_k(x)}$ is an IMT of type $\left(M,N_k\right)$ with $N_k\geq N$ and the phase functions satisfy the
  separation condition: for any $(a,b)\in \R^{2n}$, there exists at most one $f_k$ satisfying that 
  \[
\left|a\right|^{-s}\left| a-N_k\grad\phi_k(b) \right|\leq 1.
  \]
  We denote by $F\left(M,N,K,s\right)$ the set of all
  such functions.
\end{definition}

In real applications, this well-separation condition might not be valid for a multi-component signal at every $x$. However, the SST will work wherever the well-separation condition is satisfied locally.

The key analysis of the SSWPT is how well the information function $v_f(a,b)$ approximates the instantaneous frequencies or local wave vectors. If the approximation is accurate enough, the synchrosqueezed energy distribution $T_f(a,b)$ gives a sharpened time-frequency representation of $f(x)$. 
We close this section with the following theorem that summarizes the main analysis of the $n$-dimensional SSWPT for a superposition of IMTs without noise or perturbation. 
In what follows, when we write $O\left(\cdot\right)$, $\lesssim$, or $\gtrsim$, the implicit constants may depend on $M$, $m$ and $K$. Readers are referred to Theorem 2.2.7 in \cite{Thesis} for the proof of Theorem \ref{thm:1d1} here.
\begin{theorem}
  \label{thm:1d1}
  Suppose the n-dimensional mother wave packet is of type $(\epsilon,m)$, for any fixed $\eps\in(0,1)$ and any fixed integer $m\geq 0$.
  For a function $f(x)$, we define
  \[
  R_{\eps} = \{(a,b): |W_f(a,b)|\geq |a|^{-ns/2}\sqrt \eps\},
  \]
    \[
  S_{\eps} = \{(a,b): |W_f(a,b)|\geq \sqrt \eps\},
  \]
  and 
  \[
  Z_{k} = \{(a,b): |a-N_k\grad  \phi_k(b)|\leq |a|^s \}
  \]
  for $1\le k\le K$. For fixed $M$, $m$, $K$, $s$, and $\eps$, there
  exists a constant $N_0\left(M,m,K,s,\eps\right)\simeq \max\left\{ \epsilon^{\frac{-2}{2s-1}},\epsilon^{\frac{-1}{1-s}} \right\}$ such that for any
  $N>N_0$ and $f(x)\in F\left(M,N,K,s\right)$ the following statements
  hold.
  \begin{enumerate}[(i)]
  \item $\{Z_{k}: 1\le k \le K\}$ are disjoint and $S_{\eps}\subset R_{\eps}
    \subset \bigcup_{1\le k \le K} Z_{k}$;
  \item For any $(a,b) \in R_{\eps} \cap Z_{k}$, 
    \[
    \frac{|v_f(a,b)-N_k\grad \phi_k(b)|}{ |N_k \grad \phi_k(b)|}\lesssim\sqrt \eps;
    \]
     \item For any $(a,b) \in S_{\eps} \cap Z_{k}$, 
    \[
    \frac{|v_f(a,b)-N_k\grad \phi_k(b)|}{ |N_k \grad \phi_k(b)|}\lesssim N_k^{-ns/2}\sqrt \eps.
    \]
  \end{enumerate}
  \end{theorem}

\subsection{Statistical Properties of {SST}s}
\label{subsec:stat}
Similarly to the noiseless case, we will analyze how well the information function $v_f(a,b)$ approximates instantaneous frequencies or local wave vectors in the case when a superposition of IMTs is contaminated by random noise. To simplify the discussion, we will sketch out the proofs in the one-dimensional case and refer the readers to \cite{Thesis} for higher dimensional cases.

Let us start with a simple case in which the superposition is perturbed slightly by a contaminant, Theorem \ref{thm:1d2} below shows that the information function $v_f(a,b)$ can approximate instantaneous frequencies with a reasonable error determined by the magnitude of the perturbation. 

\begin{theorem}
  \label{thm:1d2}
Suppose the mother wave packet is of type $(\eps,m)$, for any fixed $\eps\in(0,1)$ and any fixed integer $m\geq 0$.
  Suppose $g(x) = f(x) + e(x)$, where $e(x)\in L^\infty$ is a small error term that satisfies $\|e\|_{L^\infty} \leq \sqrt{\epsilon_1}$ for some $\eps_1>0$. For any $p\in \left(0,\frac{1}{2}\right]$, let $\delta=\sqrt{\epsilon}+\epsilon_1^{\frac{1}{2}-p}$. Define
  \[
  R_{\delta} = \{(a,b): |W_g(a,b)|\geq |a|^{-s/2} \delta\},
  \]
    \[
  S_{\delta} = \{(a,b): |W_g(a,b)|\geq  \delta\},
  \]
  and 
  \[
  Z_{k} = \{(a,b): |a-N_k\phi_k'(b)|\leq |a|^s \}
  \]
  for $1\le k\le K$.  
  For fixed $M$, $m$, $K$, $s$, and $\eps$, there
  exists a constant $N_0\left(M,m,K,s,\eps\right)\simeq \max \left\{  \epsilon^{\frac{-1}{2s-1}},\epsilon^{\frac{-1}{1-s}}\right\}$ such that for any
  $N>N_0\left(M,m,K,s,\eps\right)$ and $f(x)\in F\left(M,N,K,s\right)$ the following statements
  hold.
  \begin{enumerate}[(i)]
  \item $\{Z_k: 1\le k \le K\}$ are disjoint and $S_{\delta}\subset R_{\delta}
    \subset \bigcup_{1\le k \le K} Z_k$;
  \item For any $(a,b) \in R_{\delta} \cap Z_k$, 
    \[
    \frac{|v_g(a,b)-N_k\phi_k'(b)|}{ |N_k \phi_k'(b)|}\lesssim\sqrt \eps + \epsilon_1^p;
    \]
    \item For any $(a,b) \in S_{\delta} \cap Z_k$, 
    \[
    \frac{|v_g(a,b)-N_k\phi_k'(b)|}{ |N_k \phi_k'(b)|}\lesssim\frac{\sqrt \eps + \epsilon_1^p }{N_k^{s/2}}.
    \]
  \end{enumerate}
\end{theorem}
We introduce the parameter $p$ to clarify the relation among the perturbation, the threshold and the accuracy for better understanding the influence of perturbation or noise. For the same purpose, a parameter $q$ will be introduced in the coming theorems. If the threshold $\delta$ is larger, e.g., $\delta\geq \sqrt{\frac{\epsilon_1}{\epsilon}}$, the relative estimate errors in $\left(ii\right)$ and $\left(iii\right)$ are bounded by $\sqrt{\epsilon}$ and $\frac{\sqrt{\epsilon}}{N_k^{s/2}}$, respectively. Similarly, one can show that the information function computed from the wave packet coefficient with a larger magnitude can better approximate the instantaneous frequency.

Below is a sketch of the proof of Theorem \ref{thm:1d2}. See the proof of Theorem 3.2.1 in \cite{Thesis} for a detailed proof.
\begin{proof} We only need to discuss the case when $a>0$. By the definition of the wave packet transform, we have
\begin{equation}
\label{Robust:eqn:1d21}
|W_e(a,b)|\lesssim  \sqrt{\epsilon_1}a^{-s/2}\quad \text{ and } \quad |\partial_b W_e(a,b)|\lesssim\sqrt{\epsilon_1}\left(a^{1-s/2}+a^{s/2}\right).
\end{equation}
If $(a,b)\in R_{\delta}$, then $|W_g(a,b)|\geq a^{-s/2}\delta$. Together with Equation \eqref{Robust:eqn:1d21}, it holds that 
\begin{equation}
\label{Robust:eqn:1d23}
|W_f(a,b)|\geq |W_g(a,b)|-|W_e(a,b)|\geq a^{-s/2}\left(\delta-\sqrt{\epsilon_1}\right)\geq a^{-s/2}\sqrt{\epsilon}.
\end{equation}
Hence, $S_{\delta}\subset R_{\delta}\subset R_{\epsilon}$, where $R_{\epsilon}$ is defined in Theorem \ref{thm:1d1} and is a subset of $\bigcup_{1\le k \le K} Z_k$. So, $\left(i\right)$ is true by Theorem \ref{thm:1d1}.

Since $R_{\delta}\subset R_{\epsilon}$, $(a,b)\in R_{\delta}\cap Z_k$ implies $(a,b)\in R_{\epsilon}\cap Z_k$. Hence, by Theorem \ref{thm:1d1}, it holds that
\begin{equation}
\label{Robust:eqn:1d24}
 \frac{|v_f(a,b)-N_k\phi_k'(b)|}{ |N_k \phi_k'(b)|}\lesssim\sqrt \eps,
\end{equation}
when $N$ is larger than a constant $N_0\left(M,m,K,s,\eps\right)\simeq \max \left\{  \epsilon^{\frac{-1}{2s-1}},\epsilon^{\frac{-1}{1-s}}\right\}$. Notice that $(a,b)\in Z_k$ implies $a\simeq N_k$. Hence, by Equation \eqref{Robust:eqn:1d21} to \eqref{Robust:eqn:1d24},
\begin{equation*}
\frac{|v_g(a,b)-N_k\phi_k'(b)|}{ |N_k \phi_k'(b)|}\leq \frac{|v_f(a,b)-N_k\phi_k'(b)|}{ |N_k \phi_k'(b)|}+\frac{|\frac{\partial_bW_f(a,b)}{2\pi iW_f(a,b)}-\frac{\partial_bW_g(a,b)}{2\pi iW_g(a,b)}|}{ |N_k \phi_k'(b)|}\leq \sqrt{\epsilon}+\epsilon_1^p,
\end{equation*}
when $N>N_0$. Hence, $\left(ii\right)$ is proved. The proof of $\left(iii\right)$ is similar. 
\end{proof}

Next, we will discuss the case when the contamination $e$ is a random perturbation. \cite{Gaussian2,Gaussian1,Gaussian3,Gaussian5,Gaussian4} are referred to for basic facts about generalized random fields and complex Gaussian processes. To warm up, we start with additive white Gaussian process in Theorem \ref{thm:1d3} and extend it to a generalized zero mean stationary Gaussian process in Theorem \ref{thm:1d4}. We assume that $e$ has an explicit power spectral function denoted by $\widehat{e}(\xi)$. $\|\cdot\|$ represents the $L^2$ norm and $\langle \cdot,\cdot\rangle$ is the standard inner product.

\begin{theorem}
  \label{thm:1d3}
  Suppose the mother wave packet is of type $(\eps,m)$, for any fixed $\eps\in(0,1)$ and any fixed integer $m\geq \frac{2}{1-s}+4$.  
  Suppose $g(x) = f(x) + e$, where $e$ is zero mean white Gaussian process with a variance $\epsilon_1^{1+q}$ for some $q>0$ and some $\eps_1>0$. For any $p\in \left(0,\frac{1}{2}\right]$, let $\delta=\sqrt{\epsilon}+\epsilon_1^{\frac{1}{2}-p}$. Define
  \[
  R_{\delta} = \{(a,b): |W_g(a,b)|\geq a^{-s/2} \delta\},
  \]
    \[
  S_{\delta} = \{(a,b): |W_g(a,b)|\geq  \delta\},
  \]
  and 
  \[
  Z_{k} = \{(a,b): |a-N_k\phi_k'(b)|\leq a^s \}
  \]
  for $1\le k\le K$. 
  For fixed $M$, $m$, $K$, $s$, and $\eps$, there
  exists a constant $N_0\left(M,m,K,s,\eps\right)\simeq \max \left\{  \epsilon^{\frac{-1}{2s-1}},\epsilon^{\frac{-1}{1-s}}\right\}$ such that for any
  $N>N_0\left(M,m,K,s,\eps\right)$ and $f(x)\in F\left(M,N,K,s\right)$ the following statements
  hold.
  \begin{enumerate}[(i)]
  \item $\{Z_k: 1\le k \le K\}$ are disjoint. 
  \item If $(a,b)\in R_{\delta}$, then $(a,b)\in \bigcup_{1\le k \le K} Z_k$ with a probability at least \[1-e^{-O\left(N_k^{-s}\epsilon_1^{-q}\right)} + O\left( \frac{\eps}{N_k^{m(1-s)}}\right).\]
  \item If $(a,b)\in S_{\delta}$, then $(a,b)\in \bigcup_{1\le k \le K} Z_k$ with a probability at least \[1-e^{-\epsilon_1^{-q}\|w\|^{-2}} + O\left( \frac{\eps}{N_k^{m(1-s)}}\right).\]
  \item If $(a,b) \in R_{\delta} \cap Z_k$ for some $k$, then 
    \[
    \frac{|v_g(a,b)-N_k\phi_k'(b)|}{ |N_k \phi_k'(b)|}\lesssim\sqrt \eps + \epsilon_1^p
    \]
    is true with a probability at least \[\left(1-e^{-O\left(N_k^{2-3s}\epsilon_1^{-q}\right)}\right)\left(1-e^{-O\left(N_k^{-s-2}\epsilon_1^{-q}\right)}\right) +O\left( \frac{\eps}{N_k^{(m-4)(1-s)-2}}\right).\]
    \item If $(a,b) \in S_{\delta} \cap Z_k$, then
    \[
    \frac{|v_g(a,b)-N_k\phi_k'(b)|}{ |N_k \phi_k'(b)|}\lesssim\frac{\sqrt \eps + \epsilon_1^p }{N_k^{s/2}}
    \]
    is true with a probability at least \[\left(1-e^{-O\left(N_k^{2-2s}\epsilon_1^{-q}\right)}\right)\left(1-e^{-O\left(N_k^{-2}\epsilon_1^{-q}\right)}\right) + O\left( \frac{\eps}{N_k^{(m-4)(1-s)-2}}\right).\] 
  \end{enumerate}
\end{theorem}
We only sketch the proof of the above theorem. See the proof of Theorem 3.2.2 in \cite{Thesis} for a long proof.
\begin{proof}
Step $1$: we prove this theorem when the mother wave packet is of type $(0,m)$ first, i.e., compactly supported in the frequency domain. 

Since $w_{ab}$ and $\partial_b w_{ab}$ are in $L^1\cap C^{m-1}$, $W_e(a,b)$ and $\partial_b W_e(a,b)$ are Gaussian variables. Hence, $W_g(a,b)=W_f(a,b)+W_e(a,b)$ and $\partial_b W_g(a,b)=\partial_b W_f(a,b)+\partial_b W_e(a,b)$ can be understood as Gaussian variables. Furthermore, $W_e(a,b)$ and $\left(W_e(a,b),\partial_b W_e(a,b)\right)$ are circularly symmetric Gaussian variables by checking that their pseudo-covariance matrices are zero. Therefore, the distribution of $W_e(a,b)$ is determined by its variance 
\[
\frac{e^{-\epsilon_1^{-\left(1+q\right)}|z_1|^2\|w\|^{-2}}}{\pi\epsilon_1^{1+q}\|w\|^2}.
\] 
If we define
\begin{equation*}
V=
\begin{pmatrix}
\|\widehat{w}\|^2 & \langle \widehat{w_{ab}},2\pi i\xi\widehat{ w_{ab}}\rangle\\
\langle 2\pi i\xi\widehat{ w_{ab}},\widehat{w_{ab}}\rangle & \langle 2\pi i\xi\widehat{w_{ab}},2\pi i\xi\widehat{ w_{ab}}\rangle
\end{pmatrix},
\end{equation*}
then $\epsilon_1^{1+q}V$ is the covariance matrix of $\left(W_e(a,b),\partial_b W_e(a,b)\right)$ and its distribution is described by the joint probability density
\[
\frac{e^{-\epsilon_1^{-\left(1+q\right)} z^{*} V^{-1}z  }  }{\pi^2\epsilon_1^{2\left(1+q\right)} \det V},
\] 
where $z = \left(z_1,z_2\right)^T$, $T$ and $*$ denote the transpose operator and conjugate transpose operator. $V$ is an invertible and self-adjoint matrix, since $W_e(a,b)$ and $\partial_b W_e(a,b)$ are linearly independent. Hence, there exist a diagonal matrix $D$ and a unitary matrix $U$ such that $V^{-1}=U^*DU$.

Part $\left(i\right)$ is true by previous theorems. Define the following events
\[
G_1=\left\{|W_e(a,b)|<a^{-s/2}\sqrt{\epsilon_1}\right\},
\]
\[
G_2=\left\{|W_e(a,b)|<\sqrt{\epsilon_1}\right\},
\]
\[
G_3 = \left\{|\partial_b W_e(a,b)|<\sqrt{\epsilon_1}\left(a^{1-s/2}+a^{s/2}\right)\right\},
\]
\[
  H_k=\left\{ \frac{|v_g(a,b)-N_k\phi_k'(b)|}{|N_k\phi_k'(b)|}\lesssim \sqrt{\epsilon} + \epsilon_1^p \right\},
  \]
  and
  \[
  J_k=\left\{ \frac{|v_g(a,b)-N_k\phi_k'(b)|}{|N_k\phi_k'(b)|}\lesssim \frac{\sqrt{\epsilon} + \epsilon_1^p }{N_k^{s/2}}\right\},
  \]
  for $1\leq k\leq K$. To conclude Part $\left(ii\right)$ to $\left(v\right)$, we need to estimate the probability $P\left(G_1\right)$, $P\left(G_2\right)$, $P\left(G_1\cap G_3\right)$, $P\left(G_2\cap G_3\right)$, $P\left(H_k\right)$ and $P\left(J_k\right)$.  Algebraic calculations show that 
 \begin{equation*}
 P\left(G_1\right)=1-e^{-a^{-s}\epsilon_1^{-q}\|w\|^{-2}} \quad\text{ and }\quad  P\left(G_2\right)=1-e^{-\epsilon_1^{-q}\|w\|^{-2}}.
 \end{equation*}
 
We are ready to summarize and conclude $\left(ii\right)$ and $\left(iii\right)$. If $(a,b)\in R_\delta$, then 
  \begin{equation}
  \label{Robust:eqn:1d31}
  |W_e(a,b)+W_f(a,b)|\geq a^{-s/2}\left(\epsilon_1^{1/2-p}+\sqrt{\epsilon}\right).
  \end{equation}
  If $(a,b)\notin \bigcup_{1\le k \le K} Z_k$, then 
  \begin{equation}
  \label{Robust:eqn:1d32}
  |W_f(a,b)|\leq a^{-s/2}\epsilon.
  \end{equation}
Equation \eqref{Robust:eqn:1d31} and \eqref{Robust:eqn:1d32} lead to $|W_e(a,b)|\geq a^{-s/2}\sqrt{\epsilon_1}$. Hence,  
\begin{equation*}
P\left( (a,b)\notin\bigcup_{1\le k \le K} Z_k\right)\leq P\left(|W_e(a,b)|\geq a^{-s/2}\sqrt{\epsilon_1} \right)= 1-P\left(G_1\right).
\end{equation*}
This means that if $(a,b)\in R_\delta$, then $(a,b)\in \bigcup_{1\le k \le K}Z_k$ with a probability at least $P\left(G_1\right)=1-e^{-a^{-s}\epsilon_1^{-q}\|w\|^{-2}}=1-e^{-O\left(N_k^{-s}\epsilon_1^{-q}\right)}$, since $a\simeq N_k$ if $(a,b)\in Z_k$. So, $\left(ii\right)$ is true. 
A similar argument applied to $(a,b)\in S_\delta$ shows that $(a,b)\in \bigcup_{1\le k \le K}Z_k$ with a probability at least $P\left(G_2\right)=1-e^{-\epsilon_1^{-q}\|w\|^{-2}}$.
Hence, $\left(iii\right)$ is proved. 
 
Recall that $V^{-1}=U^*DU$. By the change of variables $z'=Uz$, we can show that
  \begin{equation*}
  P\left(G_1\cap G_3\right) \geq \left(1-e^{-\frac{D_{11}d_1^2}{\epsilon_1^{1+q}}}\right)\left(1-e^{-\frac{D_{22}d_1^2}{\epsilon_1^{1+q}}}\right),
  \end{equation*}
  and 
    \begin{eqnarray*}
  P\left(G_2\cap G_3\right)  \geq \left(1-e^{-\frac{D_{11}d_2^2}{\epsilon_1^{1+q}}}\right)\left(1-e^{-\frac{D_{22}d_2^2}{\epsilon_1^{1+q}}}\right).
  \end{eqnarray*}
  We can further estimate that $D_{11}\simeq a^{2\left(1-s\right)}$ and $D_{22}\simeq a^{-2}$. Therefore,
  \[
    P\left(G_1\cap G_3\right) \geq \left(1-e^{-O\left(a^{2-3s}\epsilon_1^{-q}\right)}\right)\left(1-e^{-O\left(a^{-s-2}\epsilon_1^{-q}\right)}\right),
    \]
    and
    \[
     P\left(G_2\cap G_3\right) \geq \left(1-e^{-O\left(a^{2-2s}\epsilon_1^{-q}\right)}\right)\left(1-e^{-O\left(a^{-2}\epsilon_1^{-q}\right)}\right).
  \]
  By Theorem \ref{thm:1d2}, if $(a,b)\in R_\delta\cap Z_k$ for some $k$, then
  \[
  P\left(H_k\right)\geq P\left(H_k|G_1\cap G_3\right)P\left(G_1\cap G_3\right)= P\left(G_1\cap G_3\right) \geq \left(1-e^{-O\left(a^{2-3s}\epsilon_1^{-q}\right)}\right)\left(1-e^{-O\left(a^{-s-2}\epsilon_1^{-q}\right)}\right).
  \]
  Note that $a\simeq N_k$ when $a\in Z_k$, then 
   \[
  P\left(H_k\right)\geq  \left(1-e^{-O\left(N_k^{2-3s}\epsilon_1^{-q}\right)}\right)\left(1-e^{-O\left(N_k^{-s-2}\epsilon_1^{-q}\right)}\right).
  \]
  Similarly, if $(a,b)\in S_\delta\cap Z_k$ for some $k$, then
  \[
    P\left(J_k\right)\geq P\left(J_k|G_2\cap G_3\right)P\left(G_2\cap G_3\right)= P\left(G_2\cap G_3\right) \geq \left(1-e^{-O\left(N_k^{2-2s}\epsilon_1^{-q}\right)}\right)\left(1-e^{-O\left(N_k^{-2}\epsilon_1^{-q}\right)}\right).
  \]
  These arguments prove $\left(iv\right)$ and $\left(v\right)$.
  
  Step $2$: we  go on to prove this theorem when the mother wave packet is of type $(\eps,m)$ with $m\geq \frac{2}{1-s}+4$. We would like to emphasize that the requirement is crucial to the following asymptotic analysis and it keeps the error caused by the non-compact support of $\widehat{w}$ reasonably small.
   
  The sketch of the proof is similar to the first step. $W_e(a,b)$ and $\left(W_e(a,b),\partial_b W_e(a,b)\right)$ are still Gaussian variables but in general not circularly symmetric, because they would not have zero pseudo-covariance matrices. Suppose they have covariance matrices $C_1$ and $C_2$, pseudo-covariance matrices $P_1$ and $P_2$, respectively. We can still check that they have zero mean, $C_1=\eps_1^{1+q}\|w\|^2$ and $C_2=\eps_1^{1+q}V$, where $V$ is defined in the first step. By the definition of the mother wave packet of type $(\eps,m)$, the magnitude of every entry in $P_1$ and $P_2$ is bounded by $O\left(\frac{\epsilon_1^{1+q}\eps}{a^{m(1-s)}} \right)$. Notice that the covariance matrix of $\left(W_e(a,b), W_e^*(a,b)\right)$ is 
\[
V_1=
\begin{pmatrix}
C_1 & P_1\\
P_1^* & C_1^*
\end{pmatrix}.
\]
By Equation $(27)$ in \cite{Gaussian3} and the Taylor expansion, the distribution of $W_e(a,b)$ is described by the following distribution
\[
\frac{e^{-\frac{1}{2} (z_1^*,z_1)V_1^{-1} (z_1,z_1^*)^T   }}{ \pi \sqrt{\det V_1}  }
= \frac{e^{-\epsilon_1^{-\left(1+q\right)}|z_1|^2\|w\|^{-2}}}{\pi\epsilon_1^{1+q}\|w\|^2} \left( 1+ O\left( \frac{\eps |z_1|^2}{\epsilon_1^{1+q}a^{m(1-s)}}  \right)\right).
\] 
By the same argument, the covariance matrix of $\left(W_e(a,b),\partial_b W_e(a,b),W_e^*(a,b),\partial_b W_e^*(a,b)\right)$ is
\[
V_2=
\begin{pmatrix}
C_2 & P_2\\
P_2^* & C_2^*
\end{pmatrix}.
\]
Let $z = \left(z_1,z_2\right)^T$. Then the distribution of $\left(W_e(a,b),\partial_b W_e(a,b) \right)$ is described by the joint probability density
\begin{equation}
\label{Robust:eqn:1d3p}
\frac{e^{-\frac{1}{2}(z_1^*,z_2^*,z_1,z_2) V_2^{-1}(z_1,z_2,z_1^*,z_2^*)^T  }  }{\pi^2\sqrt{\det V_2}}.
\end{equation}
Notice that $C_2=\eps_1^{1+p}V$ and $V$ has eigenvalues of order $a^2$ and $a^{2(s-1)}$ determined by estimating the diagonal entries of the matrix $D$ in the diagonalization $V^{-1}=U^*DU$. Hence, $C_2$ has eigenvalues of order $\eps_1^{1+p}a^2$ and $\eps_1^{1+p}a^{2(s-1)}$. Recall that the magnitude of every entry in $P_2$ is bounded by $O\left(\frac{\epsilon_1^{1+q}\eps}{a^{m(1-s)}} \right)$. This means that $V_2$ is nearly dominated by diagonal blocks $C_2$ and $C_2^*$. Basic spectral theory for linear transforms shows that
\[
V_2^{-1}=
\begin{pmatrix}
C_2^{-1} & \\
 & (C_2^*)^{-1}
\end{pmatrix}
+ P_\eps,
\]
where $P_\eps$ is a matrix with $2$-norm bounded by $O\left(\epsilon_1^{-\left(1+q\right)} \eps a^{(m-4)(s-1)} \right)$.
$\frac{m-6}{m-4}\geq s$ is crucial to the above spectral analysis. Since every entry of $P_2$ is bounded by $O\left(\frac{\epsilon_1^{1+q}\eps}{a^{m(1-s)}} \right)$,
\[
\det V_2 =  (\det C_2)^2 +O\left(\frac{\epsilon_1^{4(1+q)}\eps}{a^{m-2-(m+2)s}} \right),
\]
where the residual comes from the entry bound and the eigenvalues of $C_2$. Hence \eqref{Robust:eqn:1d3p} is actually
\[
\frac{e^{-\epsilon_1^{-\left(1+q\right)} z^{*} V^{-1}z  }e^{-\frac{1}{2}(z_1^*,z_2^*,z_1,z_2) P_\eps (z_1,z_2,z_1^*,z_2^*)^T   }  }{\pi^2 \epsilon_1^{2(1+q)}\sqrt{ (\det V)^2 +O\left(\frac{\eps}{a^{m-2-(m+2)s)}} \right) }}.
\] 
By the same argument as in the first step, we can show that there exist a diagonal matrix $D=\diag\{a^{2(1-s)},a^{-2} \} $ and a unitary matrix $U$ such that $V^{-1}=U^*DU$. Part $\left(i\right)$ is still true by previous theorems. To conclude Part $\left(ii\right)$ to $\left(v\right)$, we still need to estimate the probability of those events defined in the first step, i.e.,  $P\left(G_1\right)$, $P\left(G_2\right)$, $P\left(G_1\cap G_3\right)$, $P\left(G_2\cap G_3\right)$, $P\left(H_k\right)$ and $P\left(J_k\right)$.  By the estimations above, one can show that
 \begin{eqnarray*}
 P\left(G_1\right)=1-e^{-a^{-s}\epsilon_1^{-q}\|w\|^{-2}} + O\left( \frac{\eps}{a^{m(1-s)}}\right) ,
 \end{eqnarray*}
 and
  \begin{eqnarray*}
 P\left(G_2\right)
=1-e^{-\epsilon_1^{-q}\|w\|^{-2}}  + O\left( \frac{\eps}{a^{m(1-s)}}\right) .
 \end{eqnarray*}

Hence, we can conclude $\left(ii\right)$ and $\left(iii\right)$ follows the same proof in the first step. Next, we look at the last two part of this theorem.

Let us introduce notations $\delta_1=a^{-s/2}\sqrt{\epsilon_1}$, $\delta_2= \sqrt{\epsilon_1}$, $\delta_3=\left(a^{1-s/2}+a^{s/2}\right)\sqrt{\epsilon_1}$, $d_1=\min\{\frac{\delta_1}{\sqrt{2}},\frac{\delta_3}{\sqrt{2}}\}$, and $d_2=\min\{\frac{\delta_2}{\sqrt{2}},\frac{\delta_3}{\sqrt{2}}\}$. By previous estimations, we have
  \begin{eqnarray}
  \label{Robust:eqn:1d3g1g3}
 & & P\left(G_1\cap G_3\right) \nonumber \\
&=& \int_{\{|z_1|<\delta_1,|z_2|<\delta_3\}}  \frac{e^{-\frac{1}{2}(z_1^*,z_2^*,z_1,z_2) V_2^{-1}(z_1,z_2,z_1^*,z_2^*)^T  }  }{\pi^2\sqrt{\det V_2}} dz_1dz_2\nonumber \\ 
&=&   \int_{\{|z_1|<\delta_1,|z_2|<\delta_3\}}  \frac{e^{-\epsilon_1^{-\left(1+q\right)} z^{*} V^{-1}z  }e^{-\frac{1}{2}(z_1^*,z_2^*,z_1,z_2) P_\eps (z_1,z_2,z_1^*,z_2^*)^T   }  }{\pi^2 \epsilon_1^{2(1+q)}\sqrt{ (\det V)^2 +O\left(\frac{\eps}{a^{m-2-(m+2)s)}} \right) }} dz_1 dz_2.
\end{eqnarray}
Since 
\begin{equation}
\label{Robust:eqn:1d3er1}
\frac{\det V}{\sqrt{ (\det V)^2 +O\left(\frac{\eps}{a^{m-2-(m+2)s)}} \right) }}= 1+ O\left(\frac{\eps}{a^{(m-2)(1-s))}}\right),
\end{equation}
we can drop out the term $O\left(\frac{\eps}{a^{m-2-(m+2)s)}} \right) $ in \eqref{Robust:eqn:1d3g1g3}, which would generate an absolute error no more than $O\left(\frac{\eps}{a^{(m-2)(1-s))}}\right)$ in the estimate of $P(G_1\cap G_3)$. Let 
\[
g(z)=-\frac{1}{2}(z_1^*,z_2^*,z_1,z_2) P_\eps (z_1,z_2,z_1^*,z_2^*)^T,
\]
 then by the change of variables we have
\begin{eqnarray}
& & P\left(G_1\cap G_3\right) \nonumber \\
&\approx &   \int_{\{|z_1|<\delta_1,|z_2|<\delta_3\}}  \frac{e^{-\epsilon_1^{-\left(1+q\right)} z^{*} V^{-1}z  }e^{g(z)   }  }{\pi^2 \epsilon_1^{2(1+q)}\det V } dz_1 dz_2\nonumber \\
& = & \int_{\{|z_1|<\delta_1,|z_2|<\delta_3\}}\frac{e^{-\epsilon_1^{-\left(1+q\right)}\left(D_{11}|z_1'|^2+D_{22}|z_2'|^2\right)}  e^{g(U^*z')   }   }{\pi^2\epsilon_1^{2\left(1+q\right)}\det V} dz'_1dz'_2\nonumber \\
  & \geq & \int_{\{|z'_1|<d_1,|z'_2|<d_1\}}\frac{e^{-\epsilon_1^{-\left(1+q\right)}\left(D_{11}|z_1'|^2+D_{22}|z_2'|^2\right)} e^{g(U^*z')   } }{\pi^2\epsilon_1^{2\left(1+q\right)}\det V} dz'_1dz'_2\nonumber \\
  &= &\frac{1}{\pi^2\epsilon_1^{2\left(1+q\right)}\det V}\int_{0}^{d_1}   \int_{0}^{d_1}  \int_{0}^{2\pi}  \int_{0}^{2\pi} r_1 r_2 e^{-\frac{D_{11}r_1^2}{\epsilon_1^{1+q}}}  e^{-\frac{D_{22}r_2^2}{\epsilon_1^{1+q}}} e^{\tilde{g}(r_1,\theta_1,r_2,\theta_2)}d\theta_1 d\theta_2  dr_1dr_2\nonumber \\
  &=&   \frac{1}{\pi^2\epsilon_1^{2\left(1+q\right)}\det V}\int_{0}^{d_1}   \int_{0}^{d_1}  \int_{0}^{2\pi}  \int_{0}^{2\pi} r_1 r_2 e^{-\frac{D_{11}r_1^2}{\epsilon_1^{1+q}}}  e^{-\frac{D_{22}r_2^2}{\epsilon_1^{1+q}}} \left( e^{\tilde{g}(r_1,\theta_1,r_2,\theta_2)} -1 \right)d\theta_1 d\theta_2  dr_1dr_2\nonumber \\
  & & + \left(1-e^{-\frac{D_{11}d_1^2}{\epsilon_1^{1+q}}}\right)\left(1-e^{-\frac{D_{22}d_1^2}{\epsilon_1^{1+q}}}\right),
  \label{Robust:eqn:1d3g1g32}
  \end{eqnarray}
  where $\tilde{g}(r_1,\theta_1,r_2,\theta_2)=g(U^*z')$. Recall that the $2$-norm of $P_\eps$ is bounded by $O\left(\epsilon_1^{-\left(1+q\right)} \eps a^{(m-4)(s-1)} \right)$. Hence, 
  \[
  |\tilde{g}(r_1,\theta_1,r_2,\theta_2)|\leq O\left(\epsilon_1^{-\left(1+q\right)} \eps a^{(m-4)(s-1)} \right)\left( |z_1|^2+|z_2|^2\right)=O\left(\epsilon_1^{-\left(1+q\right)} \eps a^{(m-4)(s-1)} \right)\left( r_1^2+r_2^2\right).
  \]
  Therefore,
  the first term in \eqref{Robust:eqn:1d3g1g32} is bounded by
  \begin{eqnarray}
& &\frac{ O\left( \eps a^{(m-4)(s-1)} \right)}{\epsilon_1^{3\left(1+q\right)}\det V}\int_{0}^{d_1}   \int_{0}^{d_1} r_1 r_2 e^{-\frac{D_{11}r_1^2}{\epsilon_1^{1+q}}}  e^{-\frac{D_{22}r_2^2}{\epsilon_1^{1+q}}} \left( r_1^2+r_2^2\right)   dr_1dr_2 \nonumber  \\
& \leq & O\left( \frac{\eps}{D_{22} a^{(m-4)(1-s)}}\right) \int_{0}^{\infty} \int_{0}^{\infty} r_1 r_2 \left( r_1^2+r_2^2\right) e^{-r_1^2} e^{-r_2^2} dr_1 dr_2\nonumber \\
& = &  O\left( \frac{\eps}{a^{(m-4)(1-s)-2}}\right). 
\label{Robust:eqn:1d3er2}
  \end{eqnarray}
  The analysis in \eqref{Robust:eqn:1d3er1} and \eqref{Robust:eqn:1d3er2} implies that
  \[
  P(G_1\cup G_3) \geq \left(1-e^{-\frac{D_{11}d_1^2}{\epsilon_1^{1+q}}}\right)\left(1-e^{-\frac{D_{22}d_1^2 }{\epsilon_1^{1+q}}}\right) + O\left( \frac{\eps}{a^{(m-4)(1-s)-2}}\right).
  \]
  and similarly
    \begin{eqnarray*}
  P\left(G_2\cap G_3\right) 
  &\geq & \left(1-e^{-\frac{D_{11}d_2^2}{\epsilon_1^{1+q}}}\right)\left(1-e^{-\frac{D_{22}d_2^2}{\epsilon_1^{1+q}}}\right)  + O\left( \frac{\eps}{a^{(m-4)(1-s)-2}}\right).
  \end{eqnarray*}
The rest of the proof is exactly the same as the one in the first step and consequently we know this theorem is also true for a mother wave packets of type $(\eps,m)$ with $m$ satisfying $m\geq \frac{2}{1-s}+4$.
\end{proof}

Thus far, we have considered the analysis for small perturbation and white Gaussian process. Next, Theorem \ref{thm:1d3} is extended to a broader class of colored random processes. 

\begin{theorem}
  \label{thm:1d4}
  Suppose the mother wave packet is of type $(\eps,m)$, for any fixed $\eps\in(0,1)$ and any fixed integer $m\geq \frac{2}{1-s}+4$.  
  Suppose $g(x) = f(x) + e$, where $e$ is a zero mean stationary Gaussian process. Let $\widehat{e}(\xi)$ denote the spectrum of $e$, $\max_{\xi} |\widehat{e}(\xi)|\leq \eps^{-1}$ and $M_a=\max_{|\xi|<1}\widehat{e}(a^s\xi+a)$.  For any $p\in(0,\frac{1}{2}]$ and $q>0$, let $\delta_a=M_a^{(\frac{1}{2}-p)/(1+q)}+\sqrt\eps$,
  \[
  R_{\delta_a} = \{(a,b): |W_g(a,b)|\geq a^{-s/2} \delta_a\},
  \]
    \[
  S_{\delta_a} = \{(a,b): |W_g(a,b)|\geq  \delta_a\},
  \]
  and 
  \[
  Z_{k} = \{(a,b): |a-N_k\phi_k'(b)|\leq a^s \}
  \]
  for $1\le k\le K$.   
  For fixed $M$, $m$, $K$, $s$, and $\eps$, there
  exists a constant $N_0\left(M,m,K,s,\eps\right)\simeq \max \left\{  \epsilon^{\frac{-1}{2s-1}},\epsilon^{\frac{-1}{1-s}}\right\}$ such that for any
  $N>N_0\left(M,m,K,s,\eps\right)$ and $f(x)\in F\left(M,N,K,s\right)$ the following statements
  hold.
  \begin{enumerate}[(i)]
  \item $\{Z_k: 1\le k \le K\}$ are disjoint. 
  \item If $(a,b)\in R_{\delta_a}$, then $(a,b)\in \bigcup_{1\le k \le K} Z_k$ with a probability at least \[1-e^{-O\left(N_k^{-s}M_a^{-q/(1+q)}\right)}+ O\left( \frac{\eps}{N_k^{m(1-s)}}\right) .\]
  \item If $(a,b)\in S_{\delta_a}$, then $(a,b)\in \bigcup_{1\le k \le K} Z_k$ with a probability at least \[1-e^{-O(M_a^{-q/(1+q)})}+ O\left( \frac{\eps}{N_k^{m(1-s)}}\right) .\]
  \item If $(a,b) \in R_{\delta_a} \cap Z_k$ for some $k$, then 
    \[
    \frac{|v_g(a,b)-N_k\phi_k'(b)|}{ |N_k \phi_k'(b)|}\lesssim\sqrt \eps + M_a^{p/(1+q)}
    \]
    is true with a probability at least \[\left(1-e^{-O\left(N_k^{2-3s}M_a^{-q/(1+q)}\right)}\right)\left(1-e^{-O\left(N_k^{-s-2}M_a^{-q/(1+q)}\right)}\right)+ O\left( \frac{\eps}{N_k^{(m-4)(1-s)-2}}\right).\]
    \item If $(a,b) \in S_{\delta_a} \cap Z_k$, then
    \[
    \frac{|v_g(a,b)-N_k\phi_k'(b)|}{ |N_k \phi_k'(b)|}\lesssim N_k^{-s/2}\left(\sqrt \eps + M_a^{p/(1+q)} \right)
    \]
    is true with a probability at least \[\left(1-e^{-O\left(N_k^{2-2s}M_a^{-q/(1+q)} \right)}\right)\left(1-e^{-O\left(N_k^{-2}M_a^{-q/(1+q)}\right)}\right)+ O\left( \frac{\eps}{N_k^{(m-4)(1-s)-2}}\right).\]
  \end{enumerate}
\end{theorem}

\begin{proof}
The proof of this theorem is nearly identical to Theorem \ref{thm:1d3}. Although calculations are more cumbersome, the derivation proceeds along the same line (see the proof of Theorem 3.2.3 in \cite{Thesis} for details).
\end{proof}

Theorem \ref{thm:1d3} and \ref{thm:1d4} provide a new insight that a smaller $s$ yields a synchrosqueezed transform that can provide a good estimation with higher probability. The parameter $m$ in the mother wave packet is also important, e.g., satisfying $m\geq \frac{2}{1-s}+4$. In a special case, if a compactly supported synchrosqueezed wavelet transform (corresponding to $s=1$) is preferable, then we require that $m\geq  \frac{2}{1-s}+4=\infty$. Hence, the mother wavelet is better to be $C^\infty$. 

We have not optimized the requirement of the variance of the Gaussian process $e$ and the probability bound in Theorem \ref{thm:1d3} and Theorem \ref{thm:1d4}. According to the numerical performance of the SSTs, the requirement of the variance could be weakened and the probability estimation could be improved. A key step is to improve the estimate of $P(G_1\cap G_3)$ and $P(G_2\cap G_3)$ in the above proofs. This is left as future work.

The above statistical property of the $1$D SSWPT can be extended to higher dimensional cases, e.g. the $2$D SSWPT and SSCT. However, the notations are much heavier and the calculations are more tedious. We close this section with the theorem for the $2$D SSWPT. See the proof of Theorem 3.4.4 and a similar theorem for the SSCT in \cite{Thesis}.

\begin{theorem}
  \label{thm:2d3}
  Suppose the $2$D mother wave packet is of type $(\epsilon,m)$, for any fixed $\eps\in(0,1)$ and any fixed integer $m \geq \max\left\{ \frac{2(1+s)}{1-s},\frac{2}{1-s}+4 \right\}$.  
  Suppose $g(x)=f(x)+e$, where $e$ is a zero mean stationary Gaussian process with a spectrum denoted by  $\widehat{e}(\xi)$ and $\max_{\xi} |\widehat{e}(\xi)|\leq \eps^{-1}$. Define $M_a=\max_{|\xi|<1}\widehat{e}(|a|^s\xi+a)$. For any $p\in(0,\frac{1}{2}]$ and $q>0$, let $\delta_a=M_a^{(\frac{1}{2}-p)/(1+q)}+\sqrt\eps$,
  \[
  R_{\delta_a} = \{(a,b): |W_g(a,b)|\geq |a|^{-s} \delta_a\},
  \]
    \[
  S_{\delta_a} = \{(a,b): |W_g(a,b)|\geq  \delta_a\},
  \]
  and 
  \[
  Z_{k} = \{(a,b): |a-N_k\grad_b\phi_k(b)|\leq |a|^s \}
  \]
 for $1\le k\le K$. For fixed $M$, $m$, $s$, $\eps$ and $K$, there
  exists a constant $N_0\left(M,m,K,s,\eps\right)\simeq \max\left\{ \epsilon^{\frac{-2}{2s-1}},\epsilon^{\frac{-1}{1-s}} \right\}$ such that for any
  $N>N_0$ and $f(x)\in F\left(M,N,K,s\right)$ the following statements
  hold.
  \begin{enumerate}[(i)]
  \item $\{Z_k: 1\le k \le K\}$ are disjoint. 
  \item If $(a,b)\in R_{\delta_a}$, then $(a,b)\in \bigcup_{1\le k \le K} Z_k$ with a probability at least \[1-e^{-O\left(N_k^{-2s}M_a^{-q/(1+q)}\right)}+ O\left( \frac{\eps}{N_k^{m(1-s)}}\right) .\]
  \item If $(a,b)\in S_{\delta_a}$, then $(a,b)\in \bigcup_{1\le k \le K} Z_k$ with a probability at least \[1-e^{-O(M_a^{-q/(1+q)})}+ O\left( \frac{\eps}{N_k^{m(1-s)}}\right) .\]
    \item If $(a,b) \in R_{\delta_a} \cap Z_{k}$ for some $k$, then
    \[
    \frac{|v_g(a,b)-N_k\grad \phi_k(b)|}{ |N_k \grad \phi_k(b)|}\lesssim\sqrt \eps +M_a^{p/(1+q)} 
    \]
    is true with a probability at least 
    \begin{eqnarray*}
    &\left(1-e^{-O\left(N_k^{2-4s}M_a^{-q/(1+q)}\right)}\right)\left(1-e^{-O\left(N_k^{-4s}M_a^{-q/(1+q)}\right)}\right)\left(1-e^{-O\left(N_k^{-2-2s}M_a^{-q/(1+q)}\right)}\right)\\
    &+ O\left( \frac{\eps}{N_k^{(m-4)(1-s)-2}}\right)+O\left(\frac{\eps}{N_k^{m-2-(m+2)s}}\right).
    \end{eqnarray*}
     \item If $(a,b) \in S_{\delta_a} \cap Z_{k}$ for some $k$, then
    \[
    \frac{|v_g(a,b)-N_k\grad \phi_k(b)|}{ |N_k \grad \phi_k(b)|}\lesssim N_k^{-s}\left(\sqrt \eps+M_a^{p/(1+q)} \right)
    \]
    is true with a probability at least 
    \begin{eqnarray*}
   & \left(1-e^{-O\left(N_k^{2-2s}M_a^{-q/(1+q)}\right)}\right)\left(1-e^{-O\left(N_k^{-2s}M_a^{-q/(1+q)}\right)}\right)\left(1-e^{-O\left(N_k^{-2}M_a^{-q/(1+q)}\right)}\right)\\
    &+ O\left( \frac{\eps}{N_k^{(m-4)(1-s)-2}}\right)+O\left(\frac{\eps}{N_k^{m-2-(m+2)s}}\right).
    \end{eqnarray*}
  \end{enumerate}
\end{theorem}

\section{Implementation and numerical results}
\label{sec:numeric}
In this section, we provide numerical examples to demonstrate some statistical properties discussed in Section \ref{sec:thm}. Guided by these properties, we explore several new ideas to improve the statistical stability of discrete SSTs in the presence of heavy noise. We have developed SynLab, a collection of MATLAB implementation for various SSTs that has been publicly available at: \url{https://github.com/HaizhaoYang/SynLab}.
Most numerical examples presented in this paper can be found in this toolbox. 

\begin{figure}[ht!]
  \begin{center}
    \begin{tabular}{cc}
   \includegraphics[height=2in]{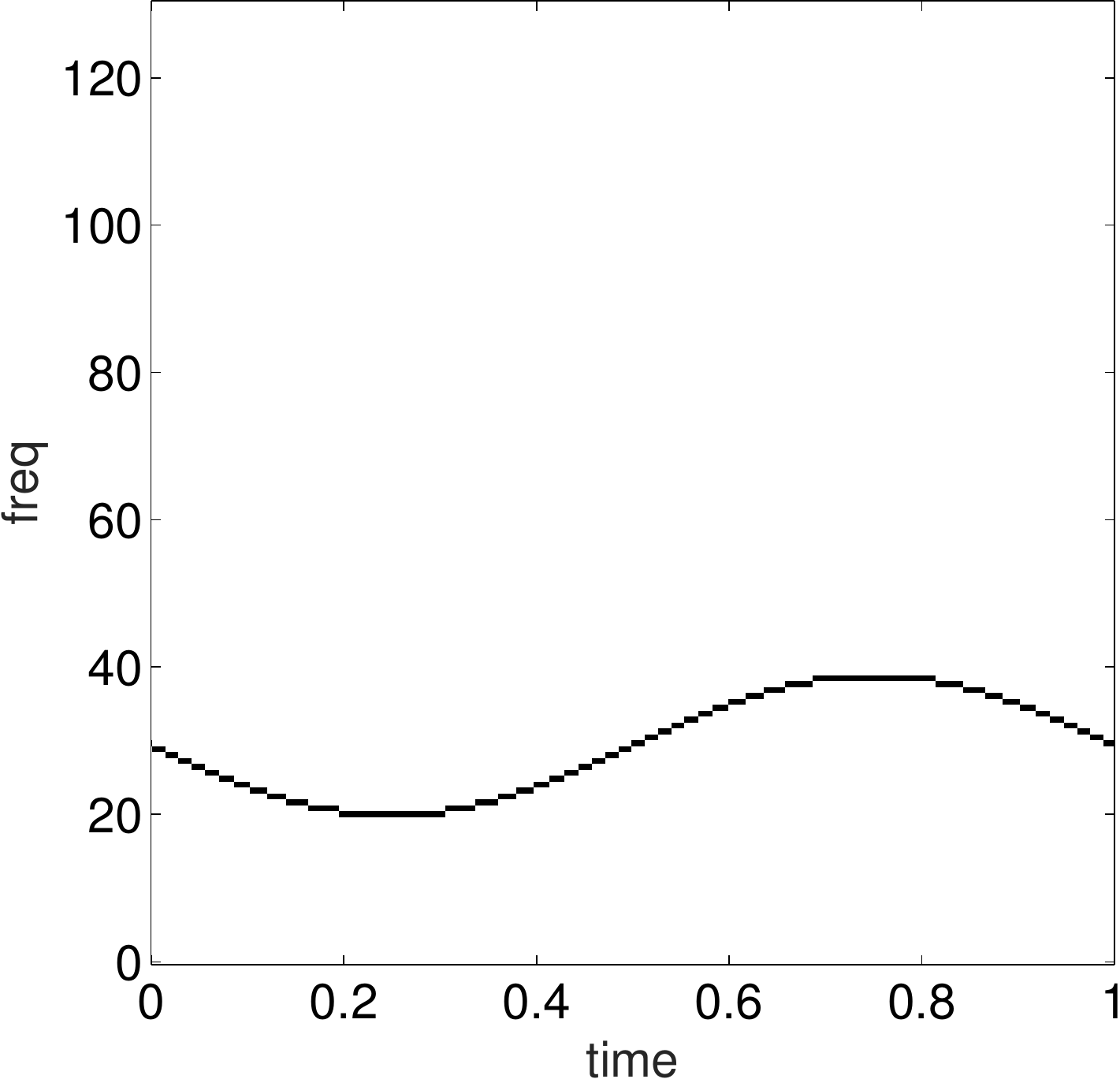} &  \includegraphics[height=2in]{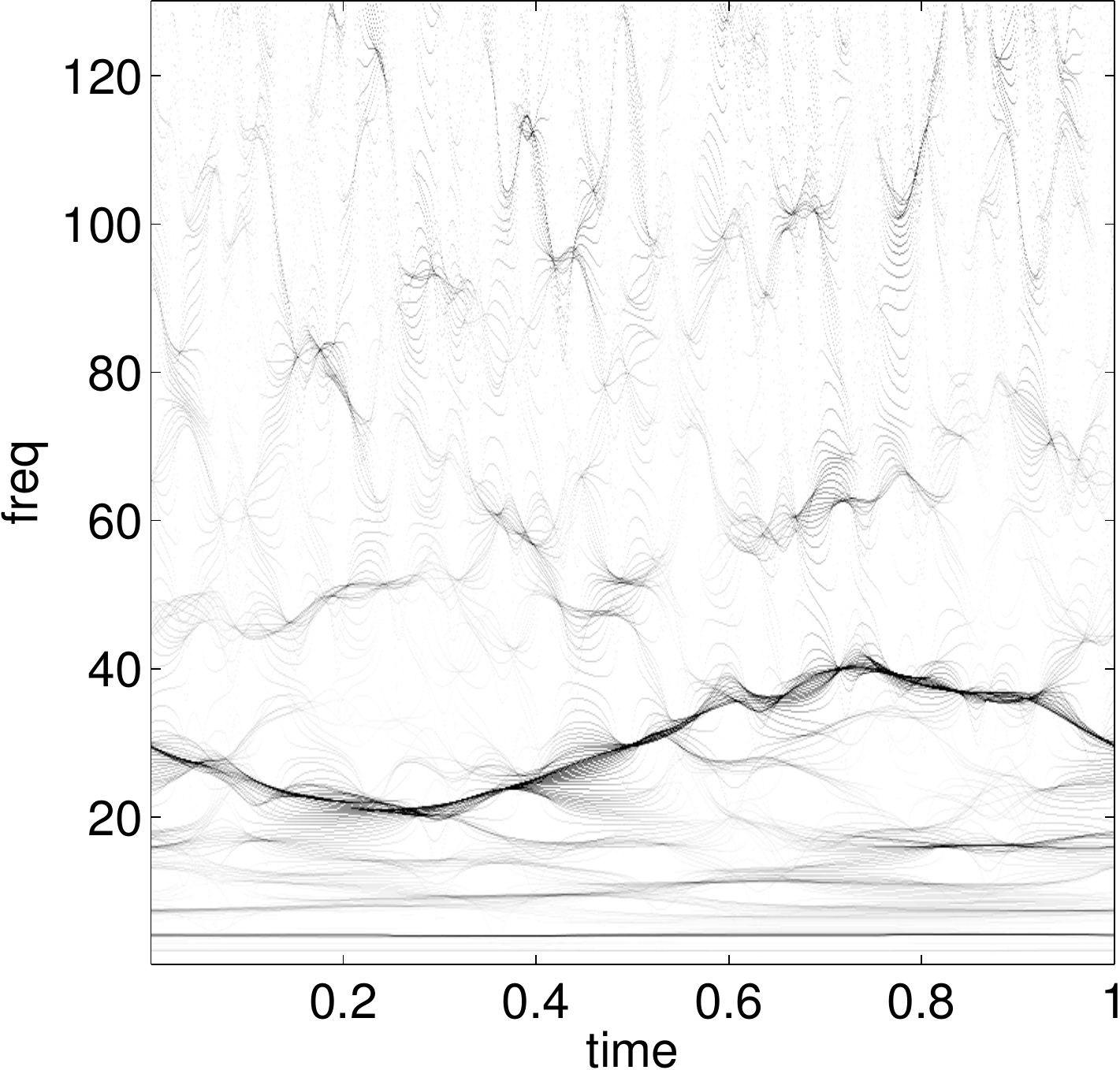}\\
   \includegraphics[height=2in]{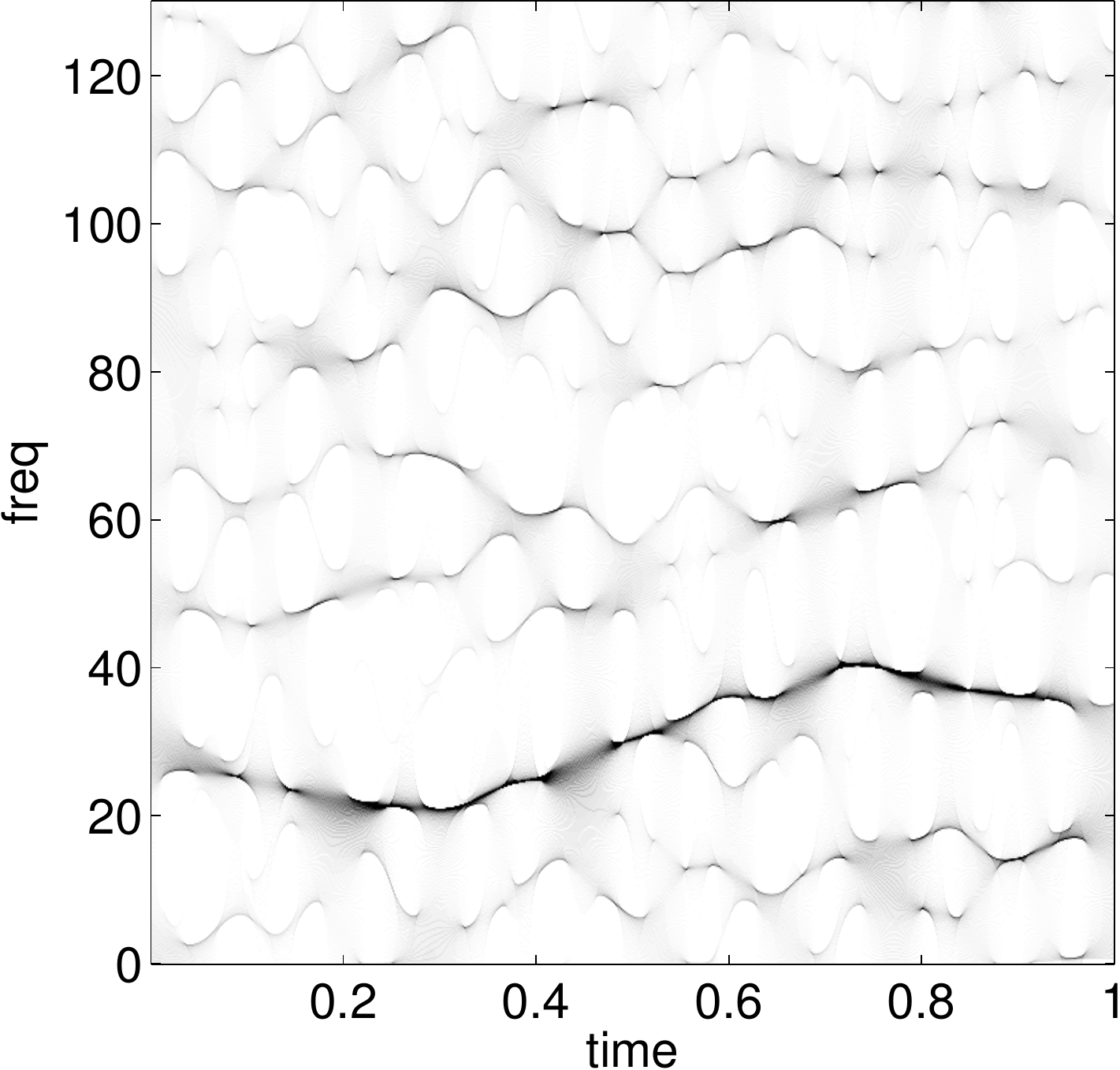} & \includegraphics[height=2in]{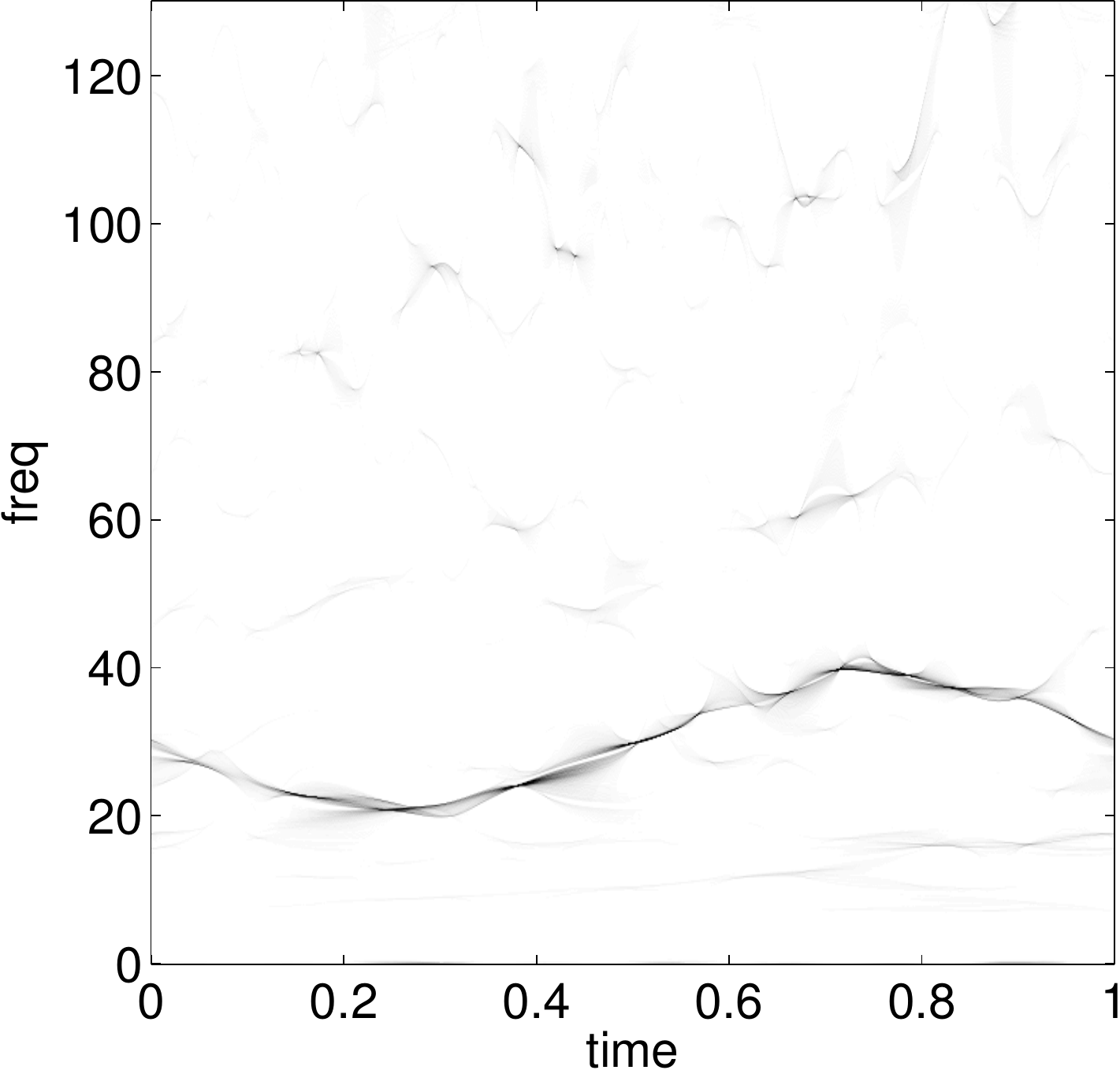} 
    \end{tabular}
  \end{center}
  \caption{Top-left: The real instantaneous frequency. Top-right: The standard SSWT available in \cite{code}. Bottom-left: The standard SSSTFT available in \cite{code}. Bottom-right: The SSWPT with only one frame, i.e., the redundancy parameter $red=1$ in SynLab, which only one wave packet frame is used to compute the synchrosqueezed transform. Parameters in these transforms have been tuned for reasonably good  visualization.}
\label{fig:orgSST}
\end{figure}

\subsection{Implementation for better statistical stability}
\label{subsec:idea}
From the discussion in Section \ref{sec:thm}, we can summarize several observations for designing better implementation of SSTs to reduce the fluctuation of noise. 

\begin{enumerate}
\item \textbf{Smaller geometric parameter $s$}\\
As we can see in the theorems for noisy data, a smaller geometric parameter $s$ results in a higher probability of a good estimation of the instantaneous frequency or the local wave vector via the information function $v_f(a,b)$. Hence, in the case when noise is heavy, it is better to apply an SST with a smaller geometric parameter $s$.
\item  \textbf{Highly redundant frames}\\
Previously in the synchrosqueezed wavelet transform (SSWT) \cite{Daubechies2011}, the synchrosqueezed short-time Fourier transform (SSSTFT) \cite{SSSTFT}, and the synchrosqueezed wave packet transform (SSWPT) \cite{1DSSWPT}, various SSTs were generated from time-frequency transforms with low redundancy to obtain efficient forward and inverse transforms. However, the resultant transforms are not reliable when noise is heavy as we can see in the example in Figure \ref{fig:orgSST}. In this example, one single IMT $f(x)=e^{60\pi i(x+0.05\cos(2\pi x))}$ is embedded in white Gaussian noise with a distribution $2\mathcal{N}(0,1)$. The signal is sampled in a time interval $[0,1]$ with a sampling rate $1024$ Hz. We apply the standard SSWT, SSSTFT in \cite{code} and the SSWPT with $s=0.75$ in SynLab and visualize their results using the same discrete grid in the time-frequency domain. Although we could identify a rough curve in these results to estimate the instantaneous frequency of $f(x)$, the accuracy is poor in some areas and there are many misleading curves coming from the noise.

As we can see in the theory for SSTs, the accuracy of the estimation provided by the information function $v_f(a,b)$ is essentially independent of $a$ and the mother wave packet $w(x)$. This motivates us to synchrosqueeze a highly redundant time-frequency transform with over-complete samples in $a$ and different mother wave packets $w(x)$. This generates many samples of estimations from a single realization of the noisy data. Averaging these estimation samples leads to a better result. It can be understood in the point of view that the contribution of IMTs to the synchrosqueezed energy distribution $T_f(a,b)$ will remain the same due to the coherent averaging, but the contribution of the noise will be smoothed out because of the  incoherent averaging. Applying different mother wave packets is essentially the same as applying multitapers in the multitaper time-frequency reassignment in \cite{Taper4}. In this paper, we only focus on oversampling the variable $a$ and leave the design of different multidimensional mother wave packets as future work. In the numerical examples later, for a fixed geometric parameter $s$, we synchrosqueeze a union of wave packet frames generated by time-frequency shifting, dilation (and rotation in multidimensional spaces). The number of frames is denoted as $red$ in SynLab.

\item \textbf{Selective time-frequency coefficients}\\
As we can see in the proofs of previous theorems, a larger time-frequency coefficient results in a higher probability for a good estimation provided by the information function $v_f(a,b)$. This inspires two ideas: 1) applying an adaptive time-frequency transform before synchrosqueezing; 2) only reassigning the largest coefficient in the domain where there is at most one IMT. The first idea aims at generating large coefficients while the second idea avoids incorrect reassignment as much as possible. Selecting the best coefficients to reassign, in some sense, is similar to the idea of sparse matching pursuit in a highly redundant frame, but we avoid its expensive optimization.
\end{enumerate}

\subsection{Numerical examples}
Here we provide numerical examples to demonstrate the efficiency of the proposed implementation in Section \ref{subsec:idea}. 
In all examples, we assume the given data $g(x) = f(x) + e(x)$ is defined in $[0,1]^n$, where $f(x)$ is the target signal, $e(x)$ is white Gaussian noise with a distribution $\sigma^2\mathcal{N}(0,1)$, and $n$ is the dimension. 

As we have seen in the theorems in Section \ref{sec:thm}, a proper threshold adaptive to the noise level after the wave packet transform is important to obtain an accurate instantaneous frequency/local wave vector estimate. We refer to \cite{noise1,noise2} for estimating noise level and \cite{SSRobust} for designing thresholds for the SSWT. The generalization of these techniques for the transforms here is straightforward. In this paper, we use a small uniform threshold $\delta=10^{-2}$ (rather than a threshold adaptive to noise level) and set $\sigma^2$ such that the noise is overwhelming the original signal. 
We refer the reader to \cite{SSWPT,SSCT} for detailed implementation using these parameters.

\subsubsection{Visual illustrations for statistical properties}

To support the theoretical analysis in Section \ref{sec:thm} and the proposals in Section\ref{subsec:idea}, we compare the performance of the SSWPT with different redundancy parameter $red$ and $s=1/2+k/8$, where $k=1$, $2$ and $3$, in both noiseless cases and highly noisy cases.

\begin{figure}[ht!]
  \begin{center}
    \begin{tabular}{ccc}
   \includegraphics[height=1.6in]{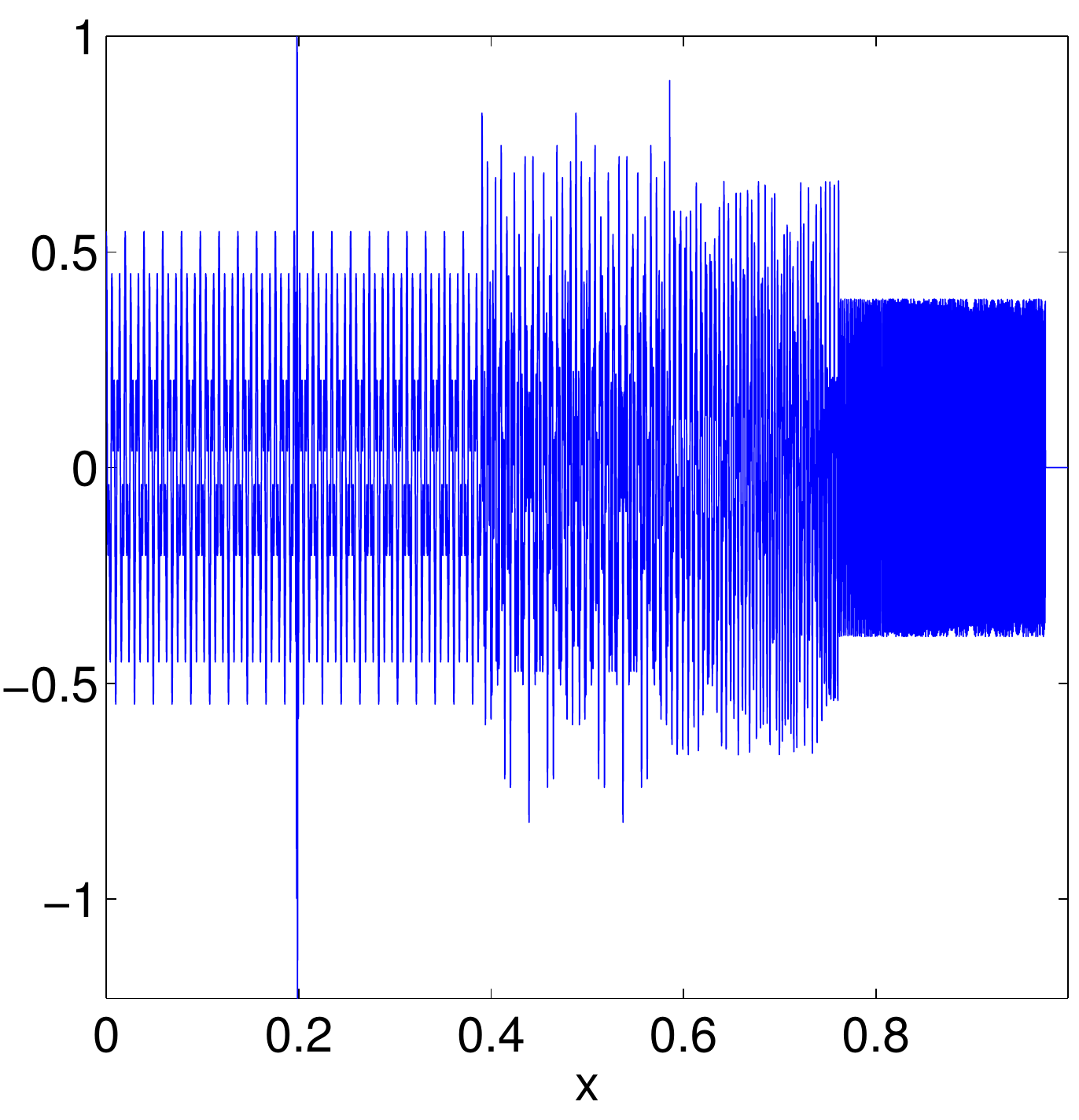}& \includegraphics[height=1.6in]{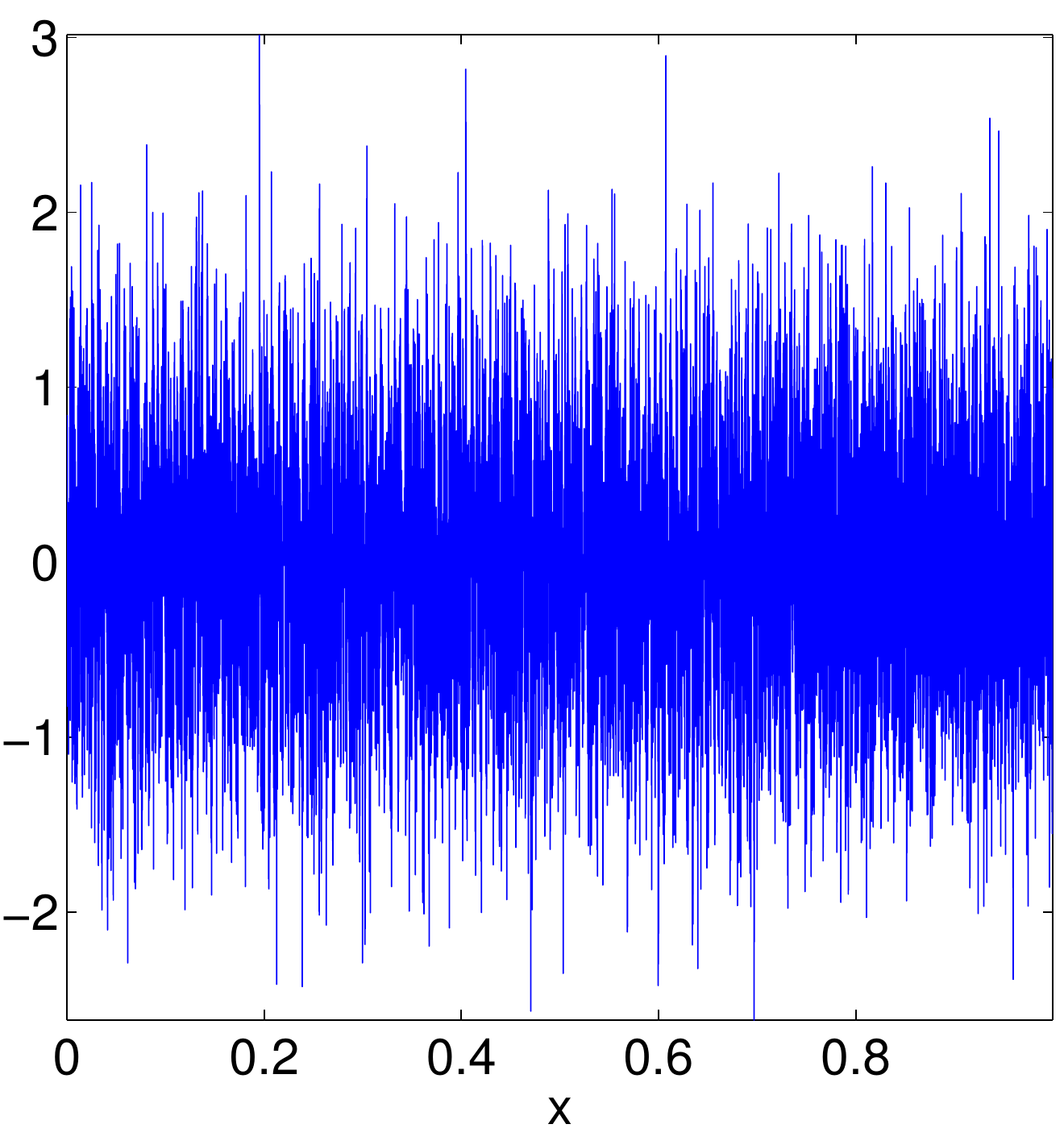} & \includegraphics[height=1.6in]{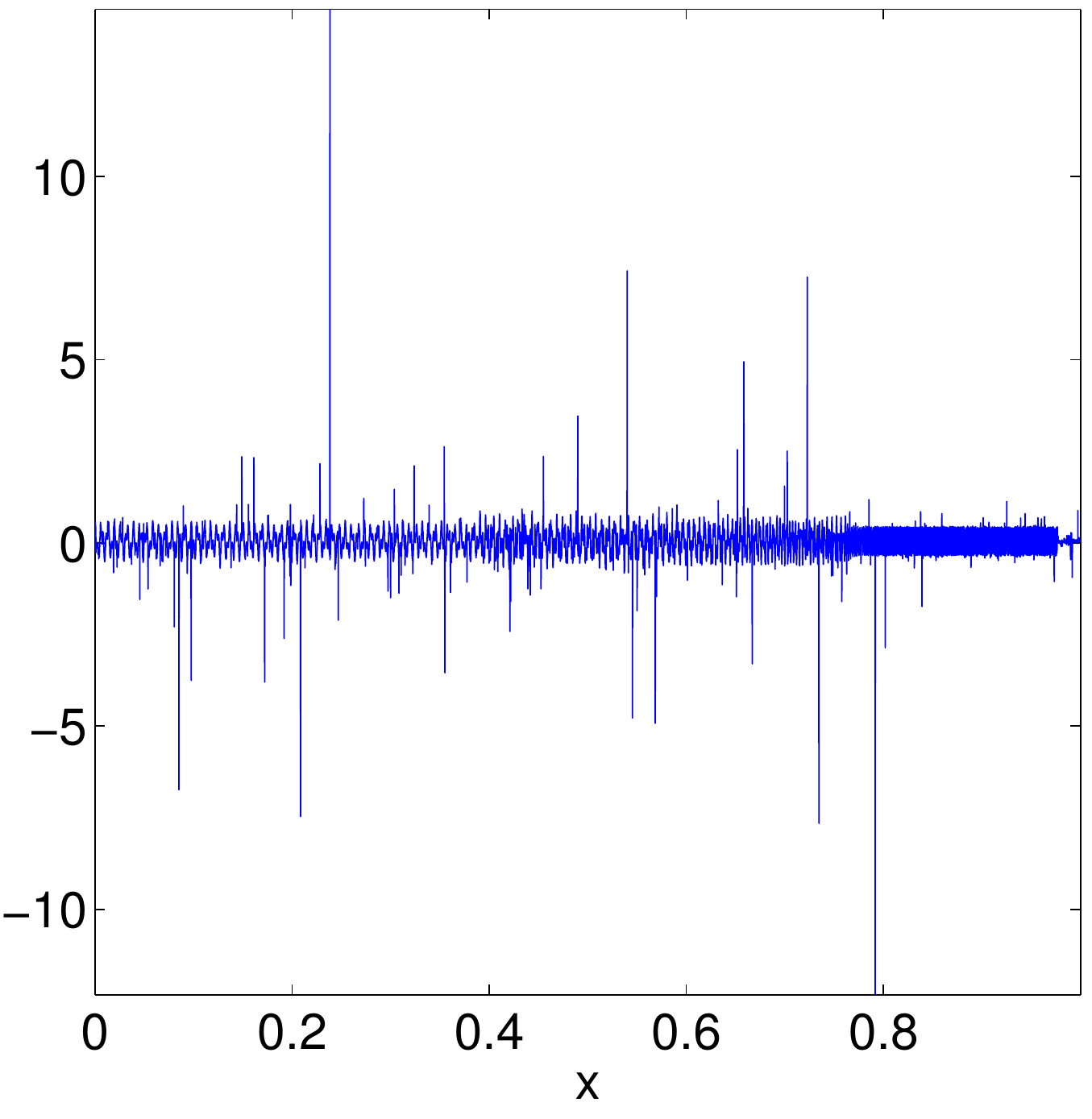} 
    \end{tabular}
  \end{center}
  \caption{Left: A $1$D synthetic benchmark signal. It is normalized using $L^\infty$ norm. Middle: A noisy version generated with white Gaussian noise $0.75\mathcal{N}(0,1)$. Right: A noisy version contaminated by an $\alpha$ stable random noise \cite{impulse} with parameters $\alpha=1$, dispersion$=0.9$, $\delta=1$, $N=8192$. The noise is re-scaled to have a $L^\infty$-norm equal to $15$ by dividing a constant factor.}
\label{fig:1Dsignal}
\end{figure}

\textbf{One-dimensional examples:}
\vspace{0.5cm}

We start with the $1$D SSWPT. In some real applications, e.g., seismic data analysis \cite{GeoReview,SSCT}, wave-like components are only supported in a bounded domain or they have sharp changes in instantaneous frequencies. Hence, we would like to test a benchmark signal $f(x)$ in which there is a component with a bounded support and an oscillatory instantaneous frequency, and a component with an exponential instantaneous frequency (see Figure \ref{fig:1Dsignal}). Of a special interest to test the performance of synchrosqueezed transforms for impulsive waves, a wavelet component is added in this signal at $x=0.2$. The synthetic benchmark signal\footnote{Prepared by Mirko van der Baan and available in \cite{GeoReview,BM}.} $f(x)$ is defined as 
\[
f(x) = \chi_{[0,0.6]}(x)f_1(x) + \chi_{[0,0.6]}(x)f_2(x) + \chi_{[0.6,1]}(x)f_4(x) + \chi_{[0.4,0.8]}(x)f_3(x) +f_5(x),
\]
where
\begin{eqnarray*}
f_1(x)&=&0.6\cos(700\pi x),\\
f_2(x)&=&0.8\cos(300\pi x),\\
f_3(x)&=&0.7\cos(1300\pi x+5\sin(20\pi x)),\\
f_4(x)&=&\sin\left( \frac{80\pi 100^{5x/4}}{\ln (100)} \right),\\
f_5(x)&=&3e^{-50(x-0.2)^2}\cos(50(x-0.2)),
\end{eqnarray*}
and $\chi(x)$ is the indicator function. 
$f(x)$ is sampled in $[0,1]$ with a sampling rate $8192$ Hz and the range of instantaneous frequencies is $150-1600$ Hz. The white Gaussian noise in this example is $0.75\mathcal{N}(0,1)$. 

Although we are not aware of the optimal value of the scaling parameter $s$, it is clear from Theorem \ref{thm:1d3} and \ref{thm:1d4} that the synchrosqueezed transform with a smaller $s$ is more suitable for noisy signals. As shown in the second and the third rows in Figure \ref{fig:1D}, in the noisy cases, the synchrosqueezed energy distribution with $s=0.625$ (in the first column) is better than the one with $s=0.75$ (in the second column), which is better than the one with $s=0.875$ (in the last column). This agrees with the conclusion in Theorem \ref{thm:1d3} and \ref{thm:1d4} that a smaller $s$ results in a higher probability to obtain a good instantaneous frequency estimate. 

Another key point is that a wave packet coefficient with a larger magnitude has a higher probability to give a good instantaneous frequency estimate. A highly redundant wave packet transform would have wave packets better fitting the local oscillation of wave-like components. In another word, there would be more coefficients with large magnitudes. The resulting synchrosqueezed energy distribution has higher non-zero energy concentrating around instantaneous frequencies. This is also validated in Figure \ref{fig:1D}. The results in the third row obtained with $red=16$ is better than those in the second row obtained with $red=1$. 

It is also interesting to observe that the SST with a smaller $s$ is better at capturing the component boundaries, e.g. at $x=0.39$, $0.59$ and $0.77$ and is more robust to an impulsive perturbation (see Figure \ref{fig:1Dsignal} and \ref{fig:1D} at $x=0.2$ and an example of $\alpha$ stable noise in Figure \ref{fig:1Dsignal} and \ref{fig:1D2}). Boundaries and impulse perturbations would produce frequency aliasing. The SSWPT with a smaller $s$ has wave packets with a smaller support in frequency and a larger support in space. Hence, it is more robust to frequency aliasing in the sense that the influence of impulsive perturbations is smoothed out and the synchrosqueezed energy of the target components might not get dispersed when it meets the frequency aliasing, as shown in Figure \ref{fig:1D2}.

However, if $s$ is small, the instantaneous frequency estimate might be smoothed out and it is difficult to observe detailed information of instantaneous frequencies. As shown in the first row of Figure \ref{fig:1D}, when the input signal is noiseless, the synchrosqueezed transforms with $s=0.75$ and $0.875$ have better accuracy than the one with $s=0.625$.  In short, it is important to tune scaling parameters for data-dependent synchrosqueezed transforms, which has been implemented in the SynLab toolbox.

\begin{figure}[ht!]
  \begin{center}
    \begin{tabular}{ccc}
       \includegraphics[height=1.8in]{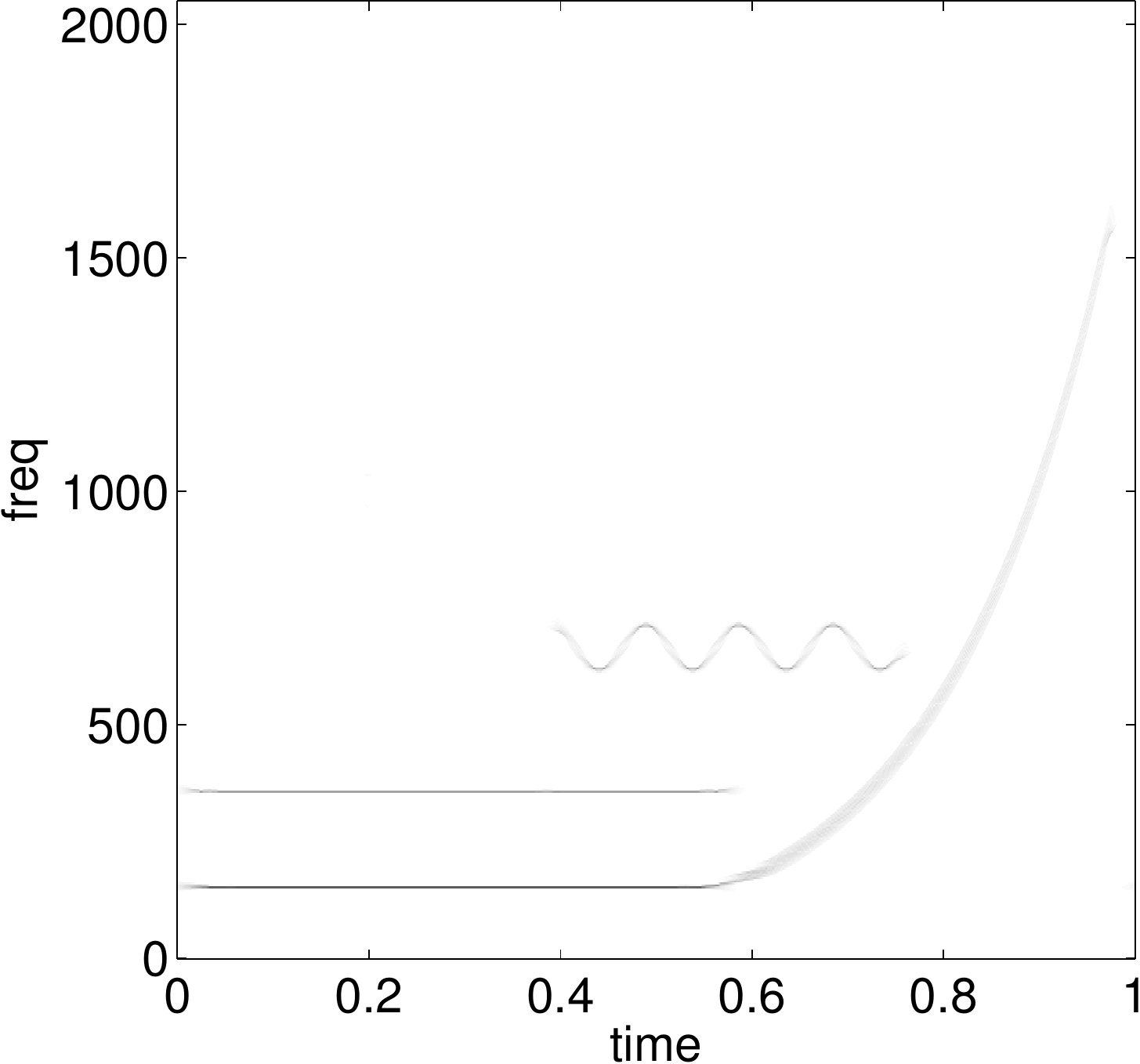}& \includegraphics[height=1.8in]{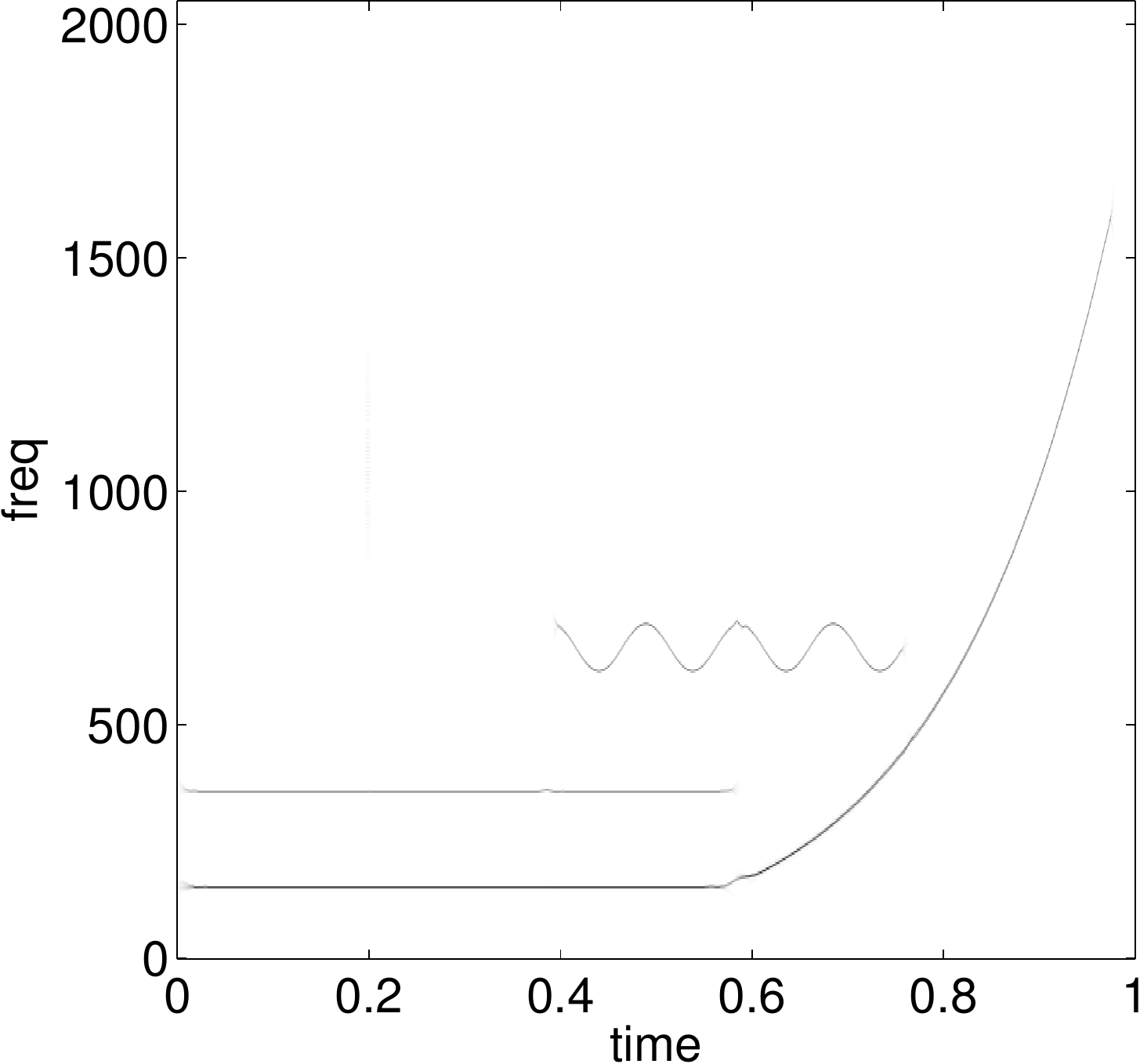}  & \includegraphics[height=1.8in]{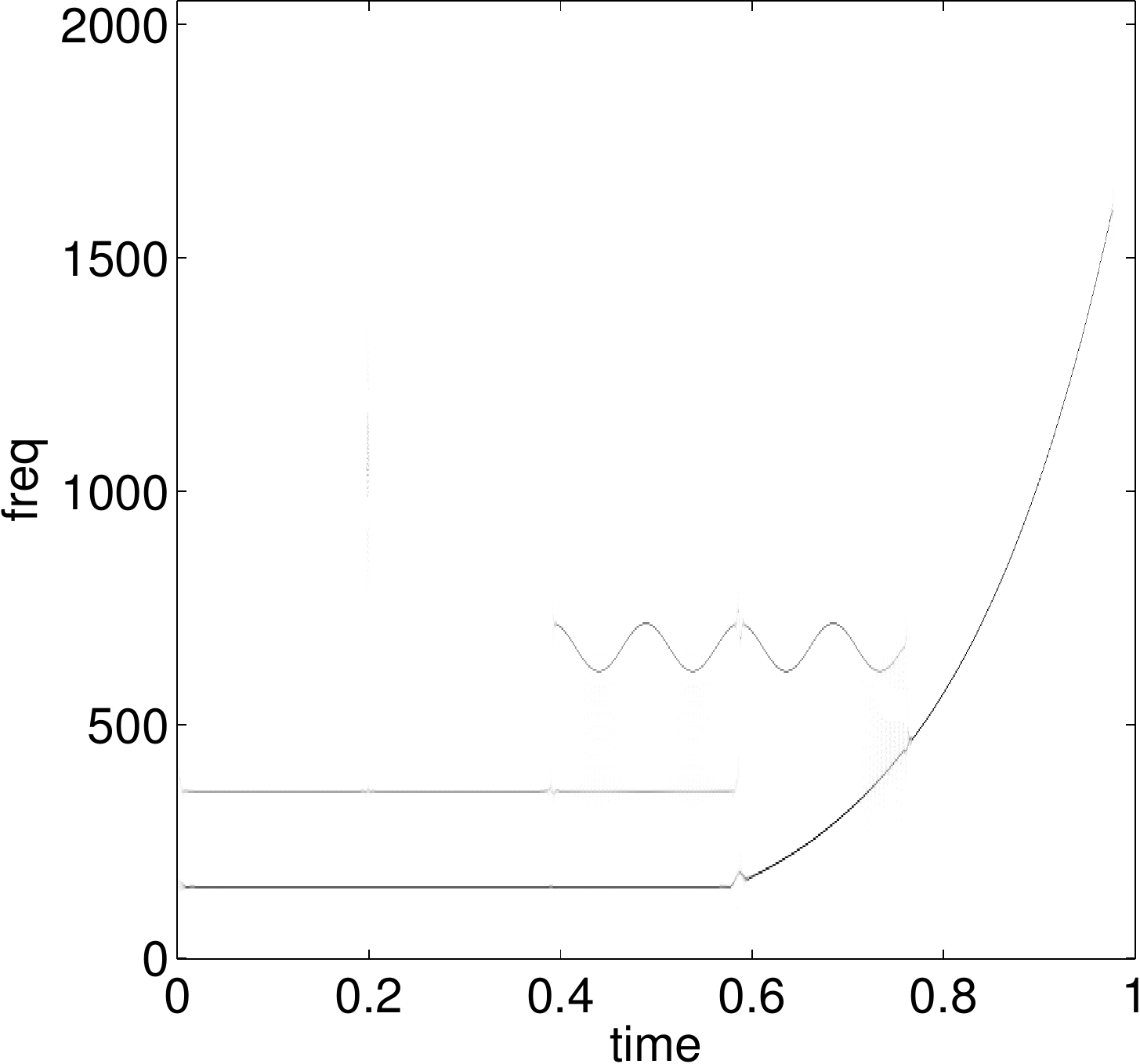}\\
   \includegraphics[height=1.8in]{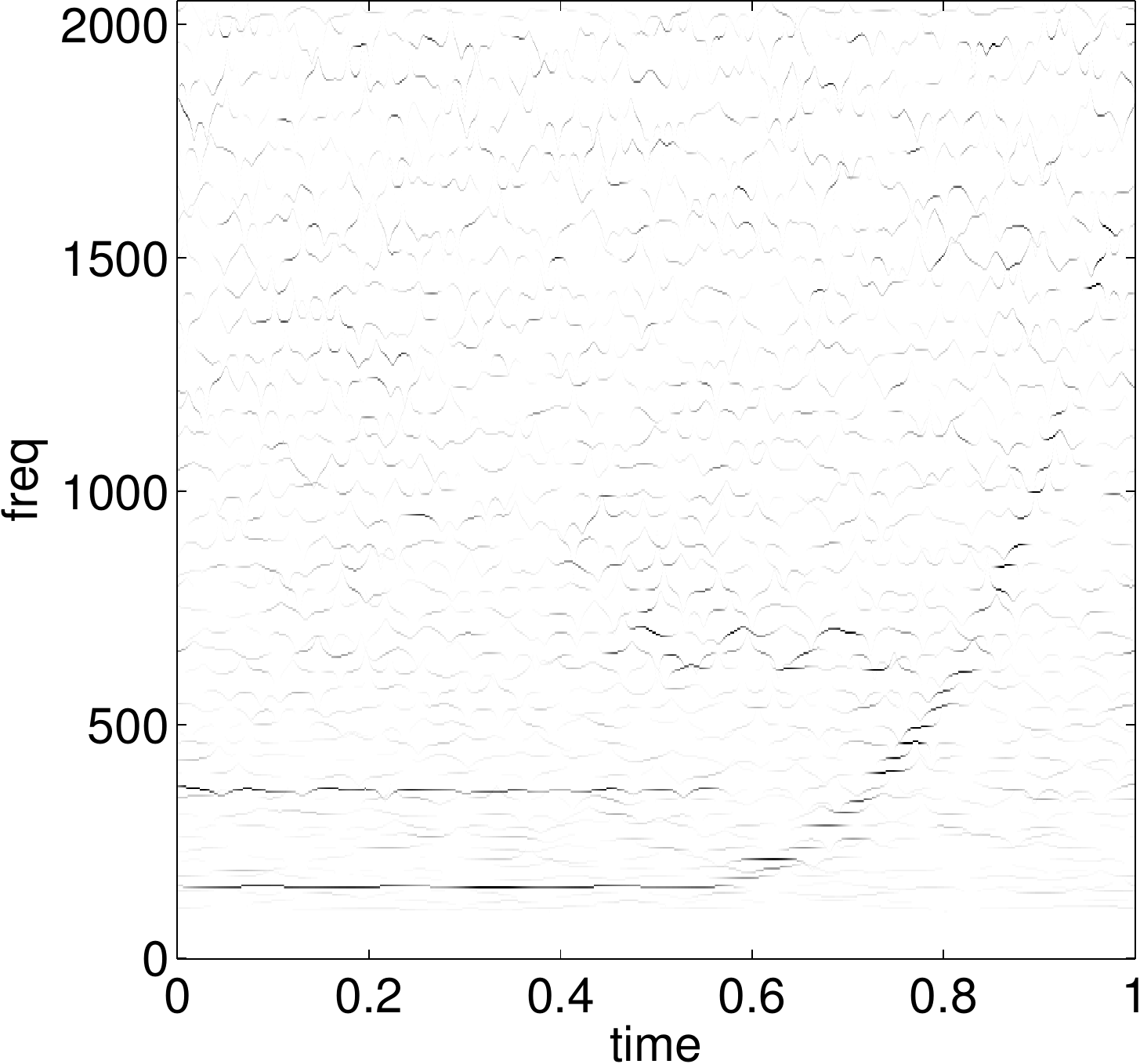}& \includegraphics[height=1.8in]{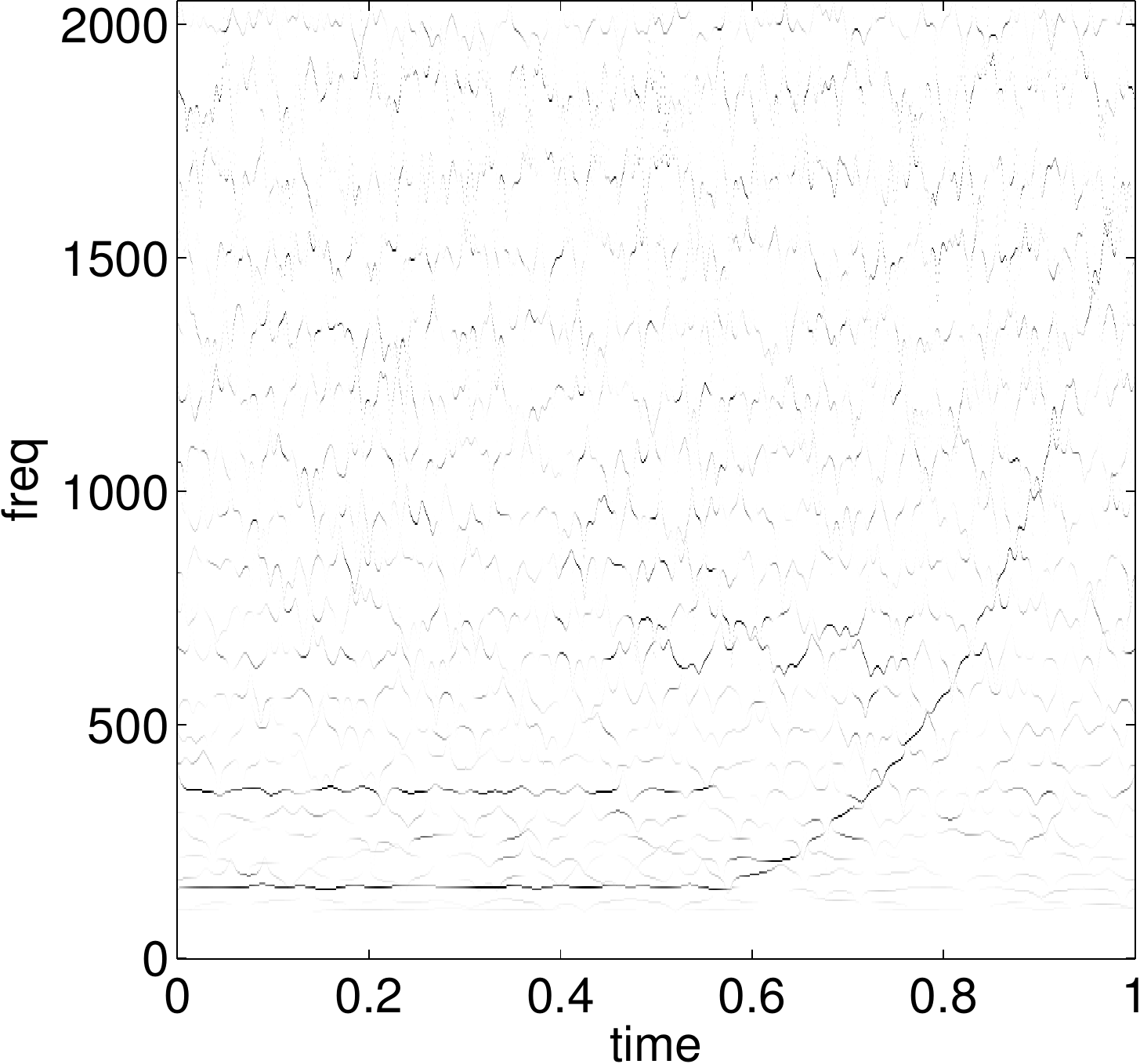}  & \includegraphics[height=1.8in]{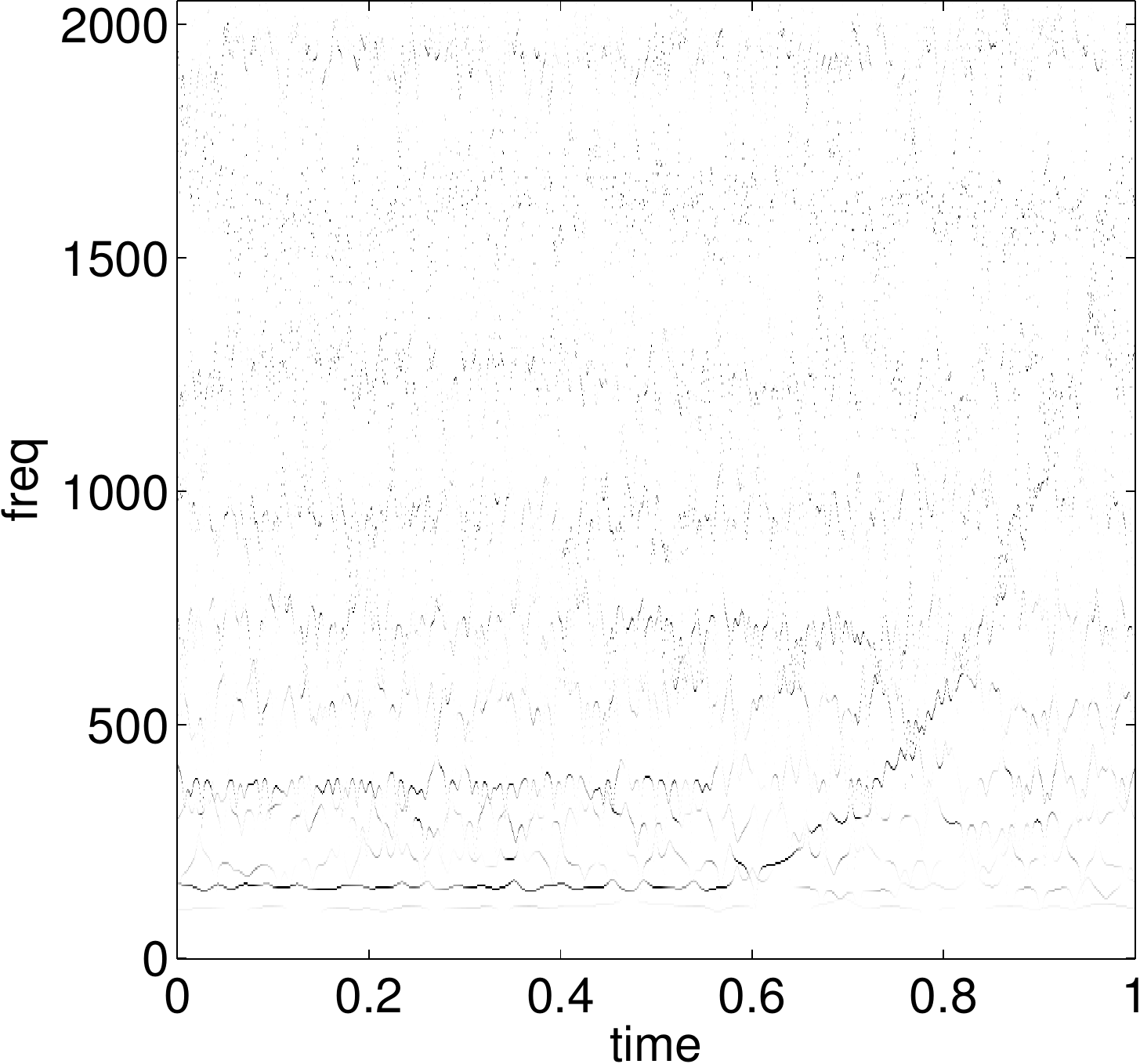}\\
  \includegraphics[height=1.8in]{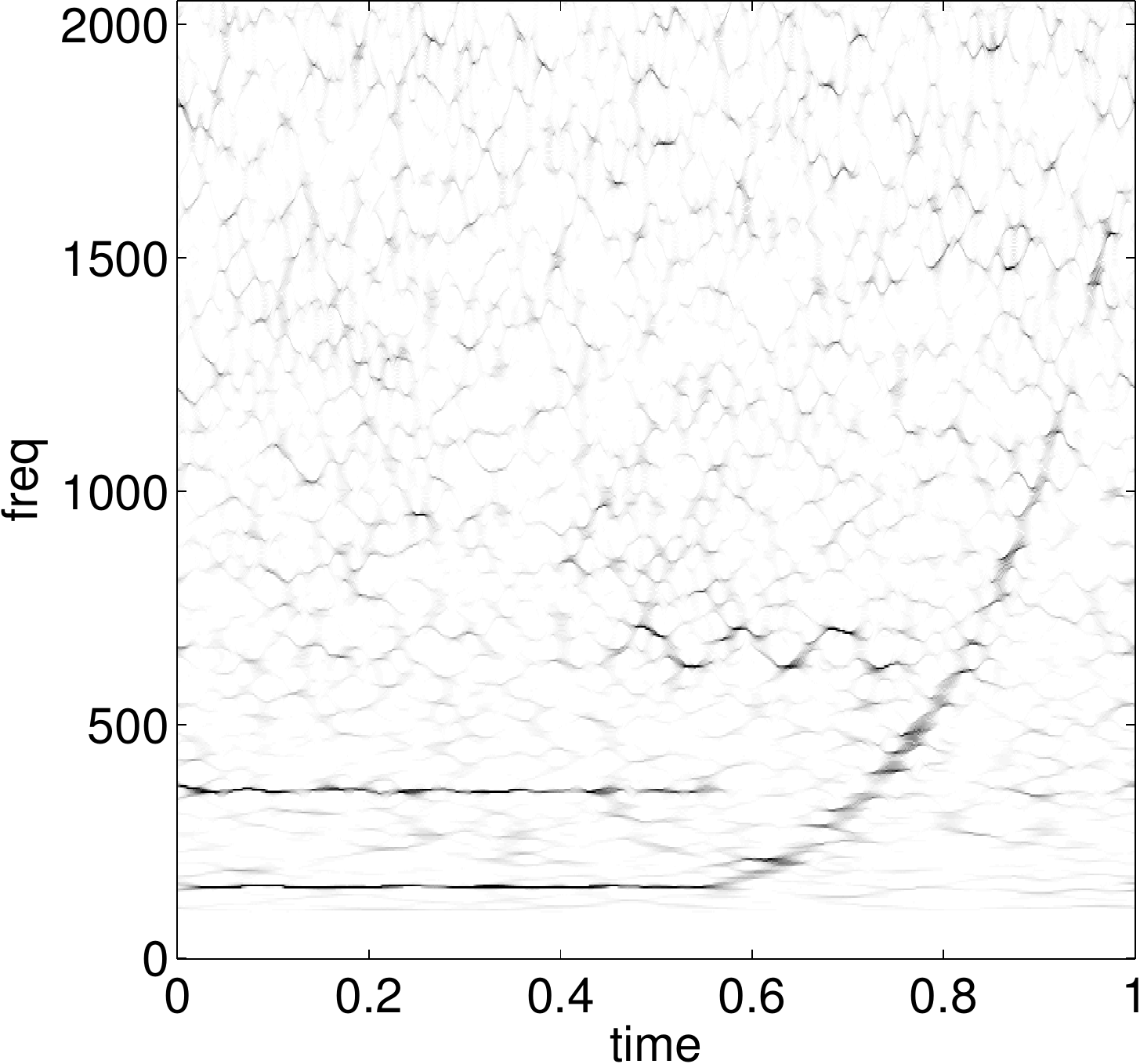}& \includegraphics[height=1.8in]{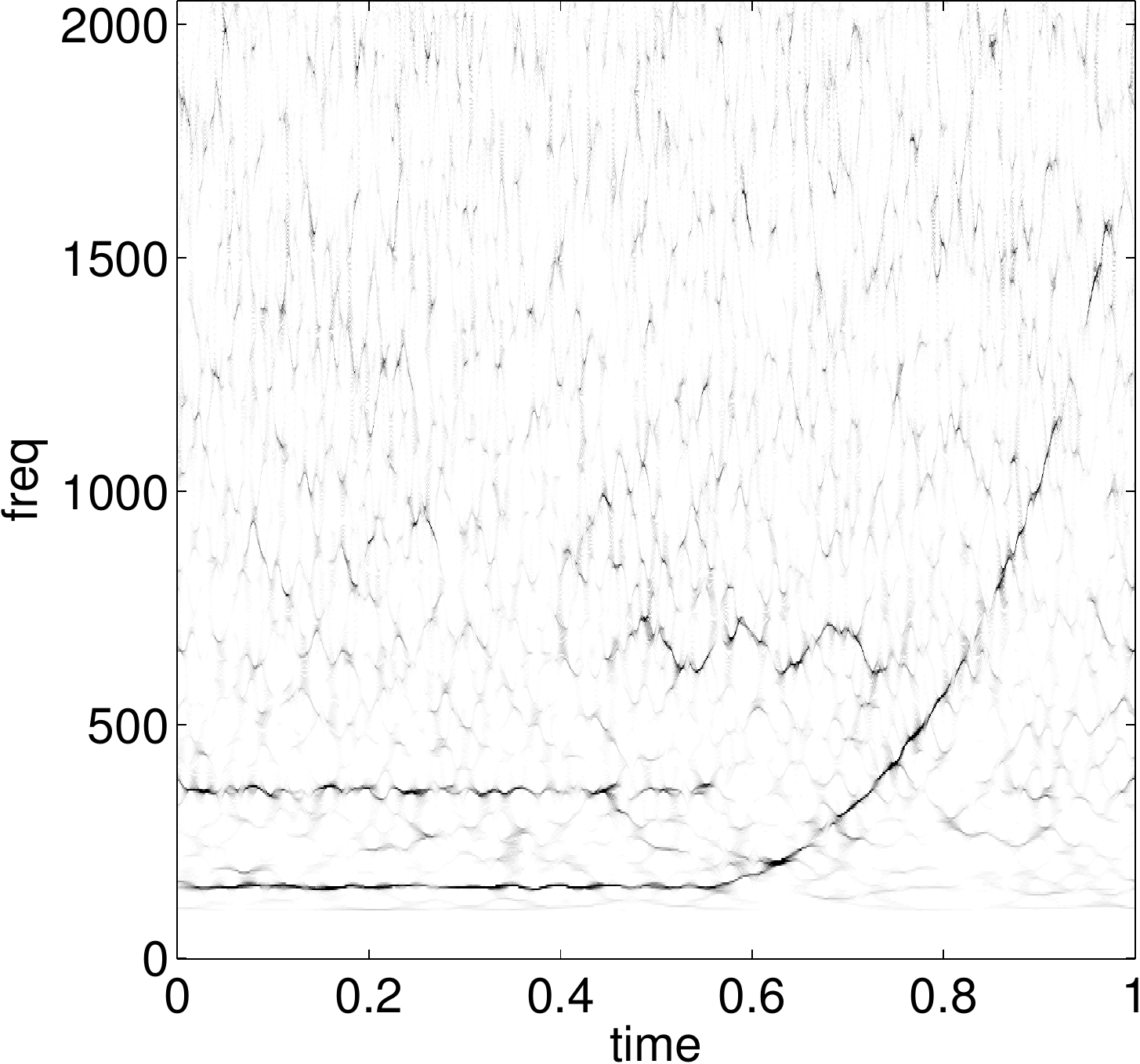}  & \includegraphics[height=1.8in]{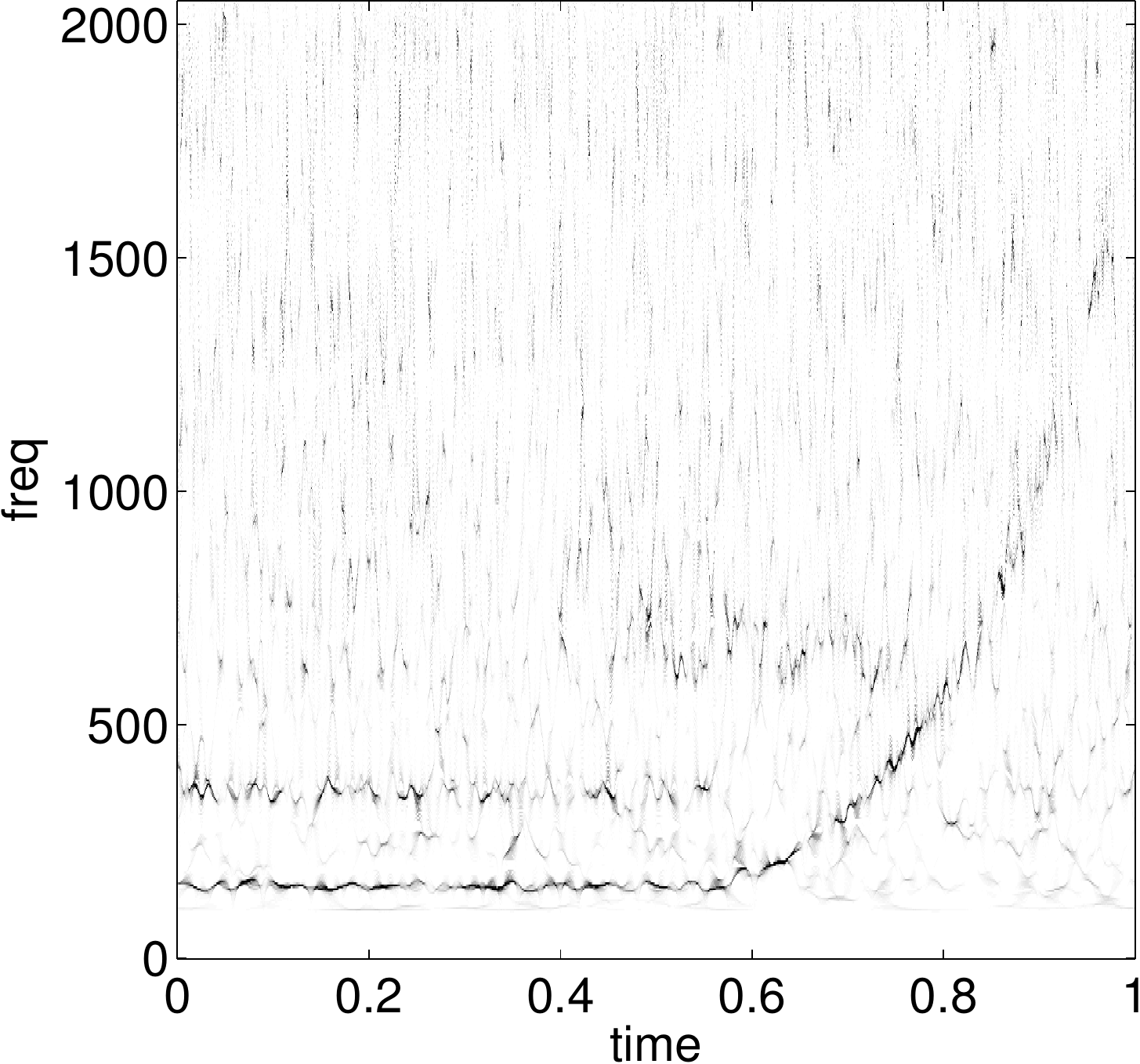}
    \end{tabular}
  \end{center}
  \caption{Synchrosqueezed energy distributions with $s=0.625$ (left column), $s=0.75$ (middle column) and $s=0.875$ (right column). In the first row, we apply the SSWPT to clean data. In the second row, the SSWPT with a smaller redundancy is applied to the noisy data with $0.75\mathcal{N}(0,1)$ noise in Figure \ref{fig:1Dsignal}. In the last row, a highly redundant SSWPT is applied to the same noisy data.}
\label{fig:1D}
\end{figure}

\begin{figure}[ht!]
  \begin{center}
    \begin{tabular}{ccc}
  \includegraphics[height=1.6in]{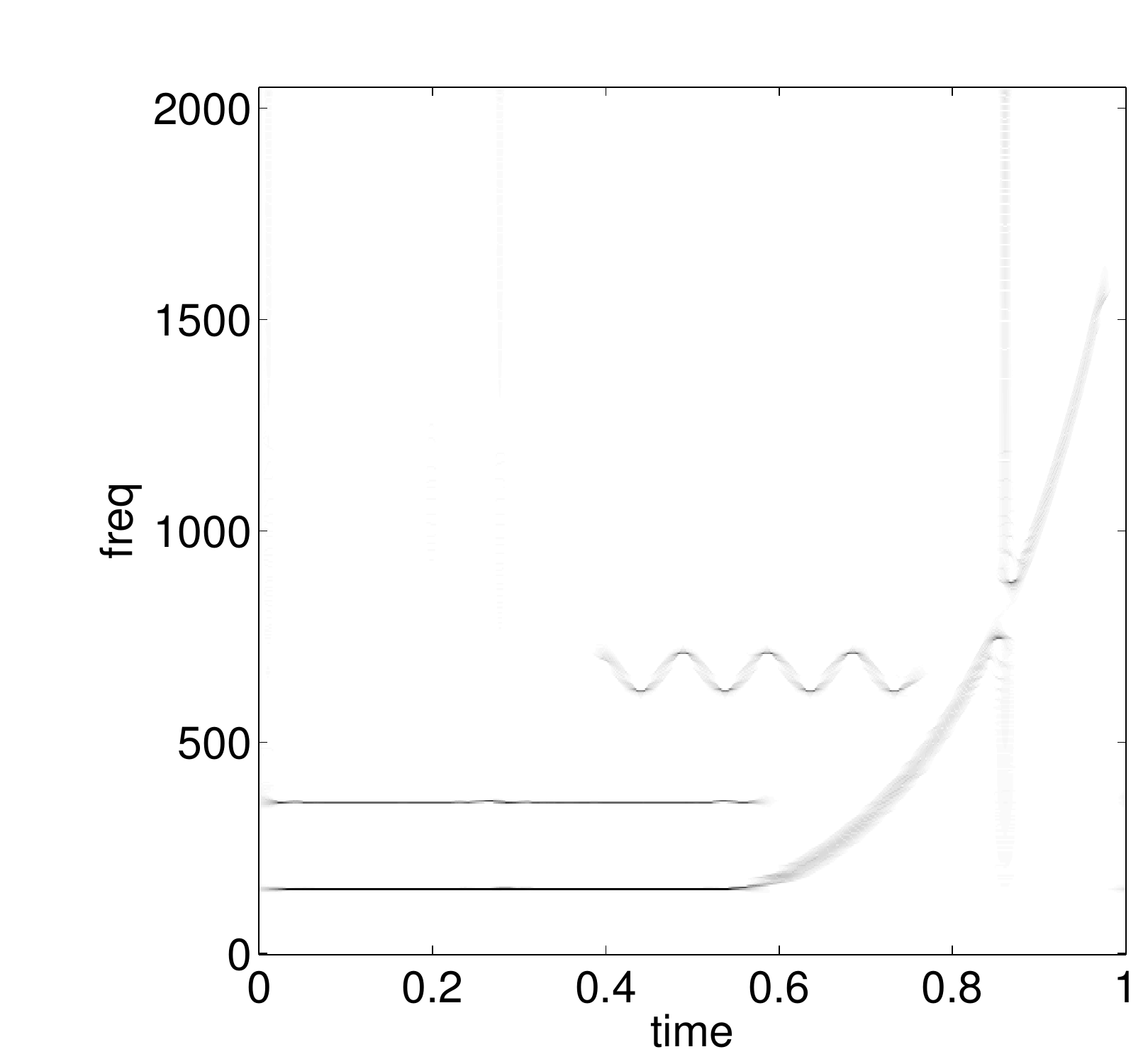}& \includegraphics[height=1.6in]{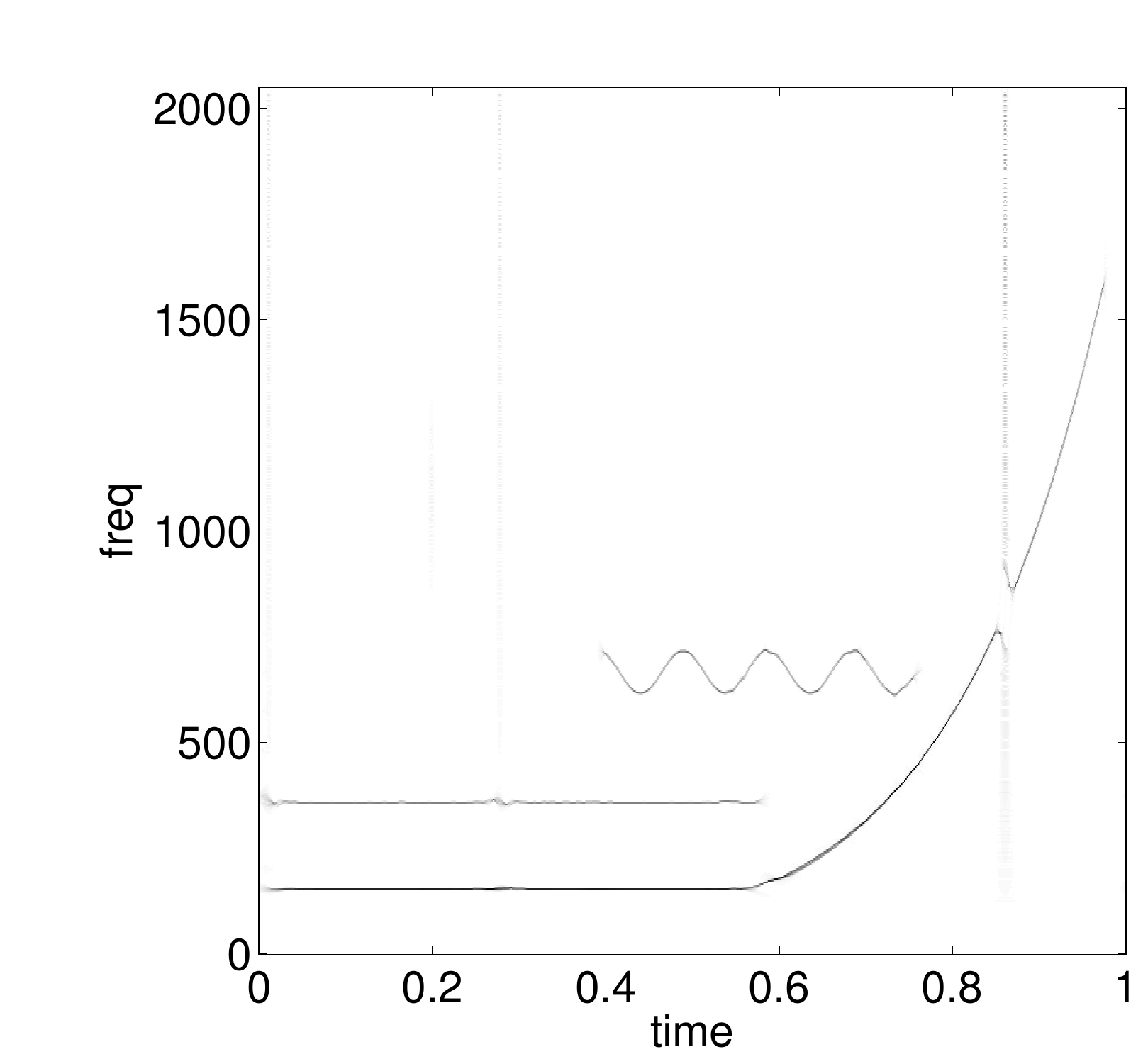}  & \includegraphics[height=1.6in]{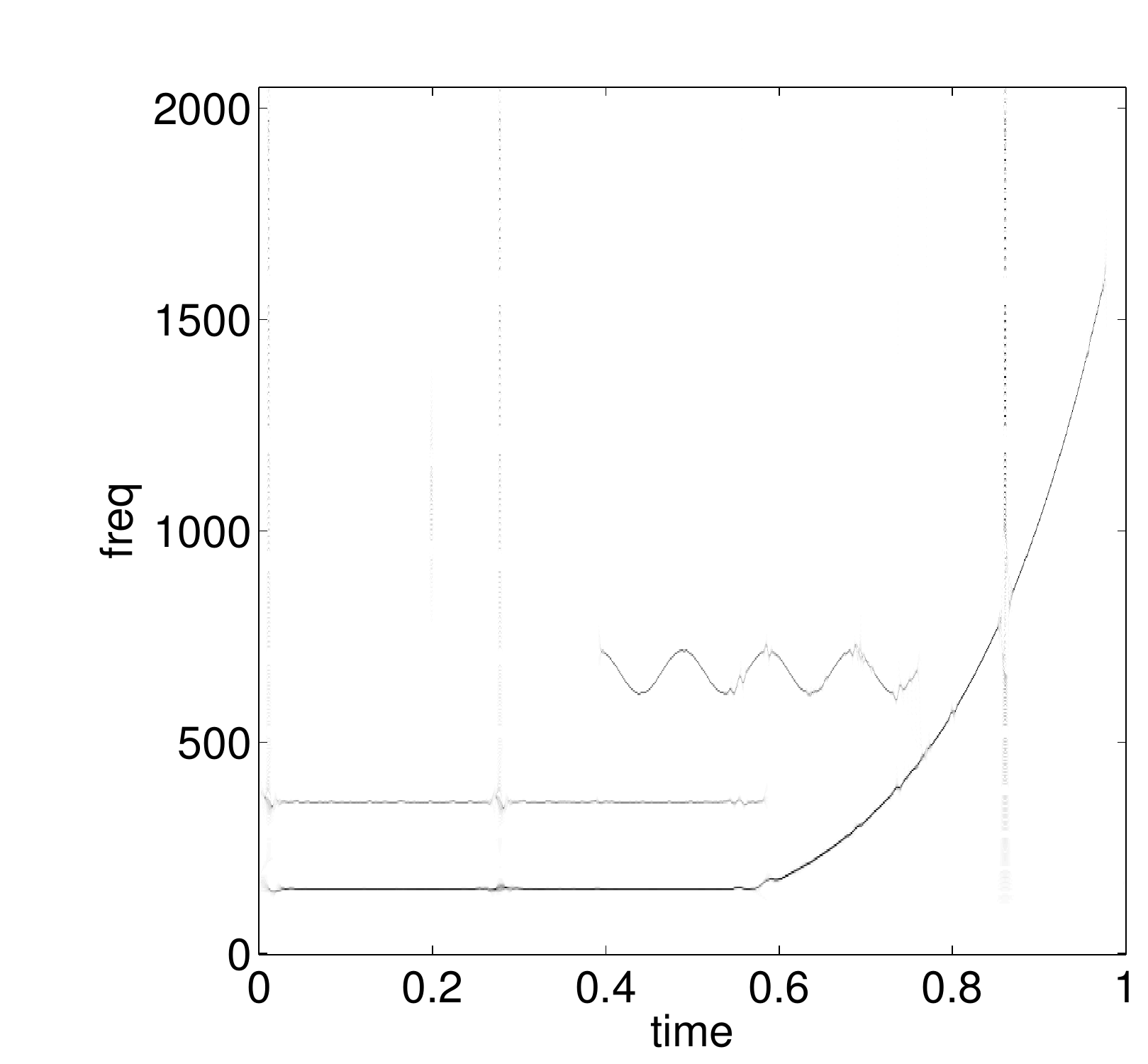}
    \end{tabular}
  \end{center}
  \caption{Synchrosqueezed energy distributions with $s=0.625$ (left), $s=0.75$ (middle) and $s=0.875$ (right) using highly redundant SSWPTs. The synchrosqueezed energy with a smaller $s$ is smoother and the influence of impulsive noise is weaker.}
\label{fig:1D2}
\end{figure}

\begin{figure}[ht!]
  \begin{center}
    \begin{tabular}{cc}
   \includegraphics[height=2in]{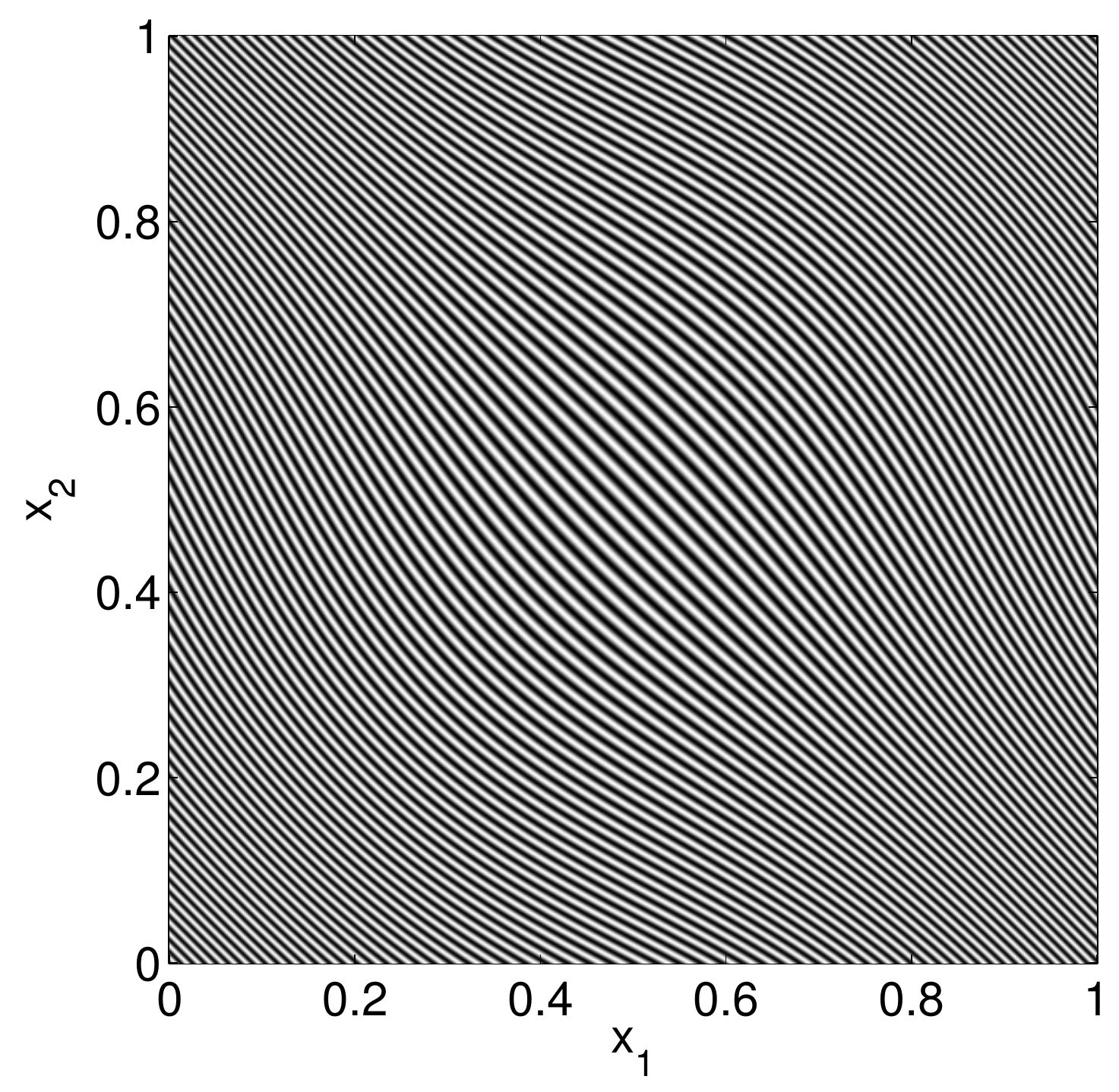}& \includegraphics[height=2in]{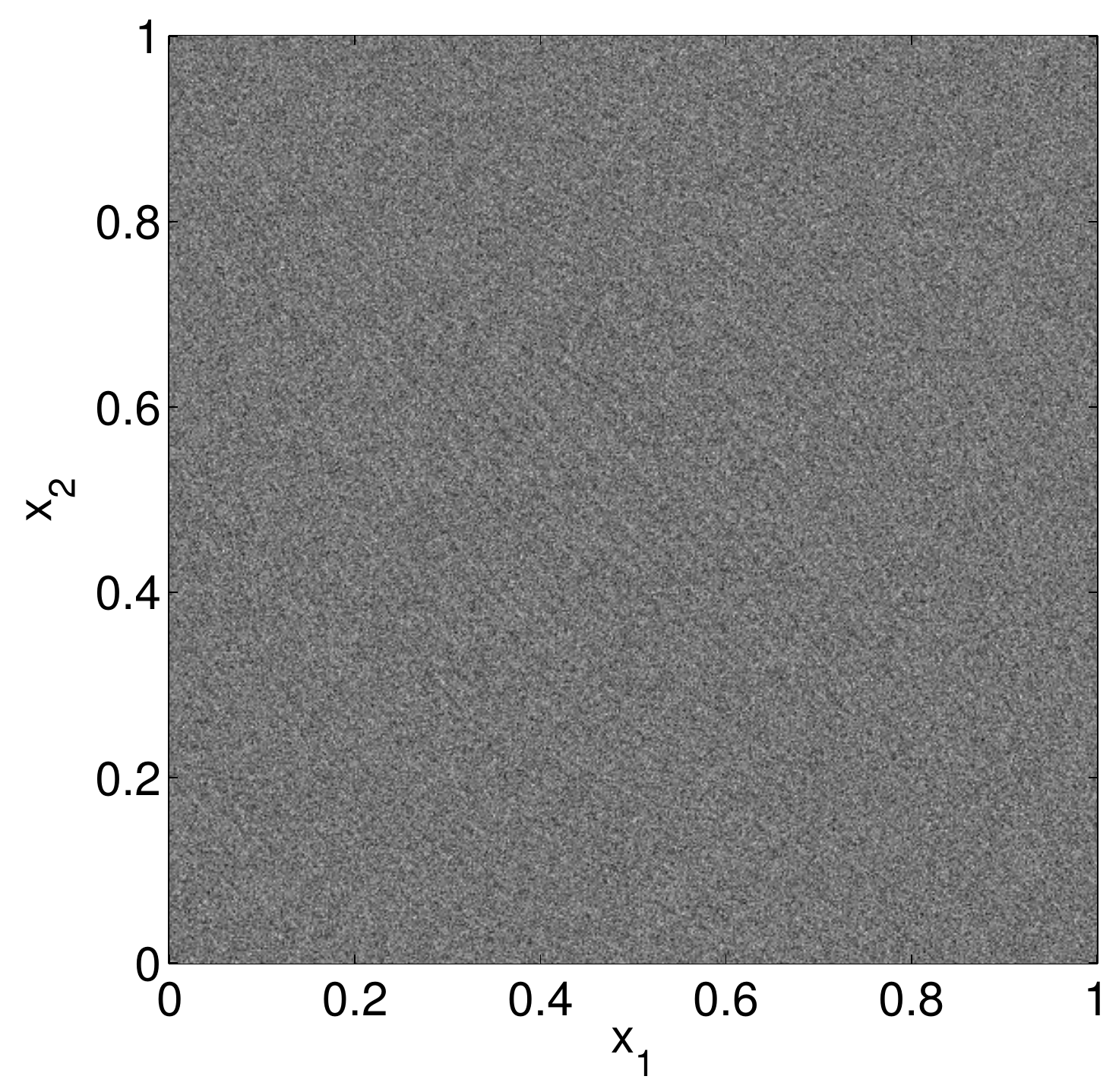} 
    \end{tabular}
  \end{center}
  \caption{Left: A $2$D noiseless wave-like component. Right: A noisy wave-like component generated with white Gaussian noise $5\mathcal{N}(0,1)$.}
\label{fig:2Dsignal}
\end{figure}

\vspace{1cm}
\textbf{Two-dimensional examples:}
\vspace{0.5cm}

We now explore the performance of the $2$D SSWPT 
using a single wave-like component in Figure \ref{fig:2Dsignal}. The function 
\begin{equation}
\label{eqn:2deximg}
f(x)=e^{2\pi i\left( 60\left(x_1+0.05\sin(2\pi x_1) \right) +60\left(x_2+0.05\sin(2\pi x_2)\right)\right)}
\end{equation}
is uniformly sampled in $[0,1]^2$ with a sampling rate $512$ Hz and is disturbed by additive white Gaussian noise $5\mathcal{N}(0,1)$. The $2$D SSWPTs with $s=0.625$, $0.75$ and $0.875$ are applied to this noisy example and their results are shown in Figure \ref{fig:2D}. Since the synchrosqueezed energy distribution $T_f(v_1,v_2,x_1,x_2)$ of an image is a function in $\R^4$, we fix $x_2=0$, stack the results in $v_2$, and visualize $\int_{\R}T_f(v_1,v_2,x_1,0)dv_2$.
The results in Figure \ref{fig:2D} again validate the theoretical conclusion in Theorem \ref{thm:2d3} that a smaller scaling parameter $s$ and a higher redundancy yield to a better SST for noisy data. 

It is interesting that a band-limited SST can also provide better statistical stability if the range of instantaneous frequencies/local wave vectors is known a priori. Let us justify this idea with the follow example. We apply the band-limited SSWPT to the $2$D noisy image in Figure \ref{fig:2D} and present the results in the last row of Figure \ref{fig:2D}. Comparing to the results in the second row of Figure \ref{fig:2D}, the band-limited SSWPT clearly outperforms the original SSWPT. 

\begin{figure}[ht!]
  \begin{center}
    \begin{tabular}{ccc}
   \includegraphics[height=1.6in]{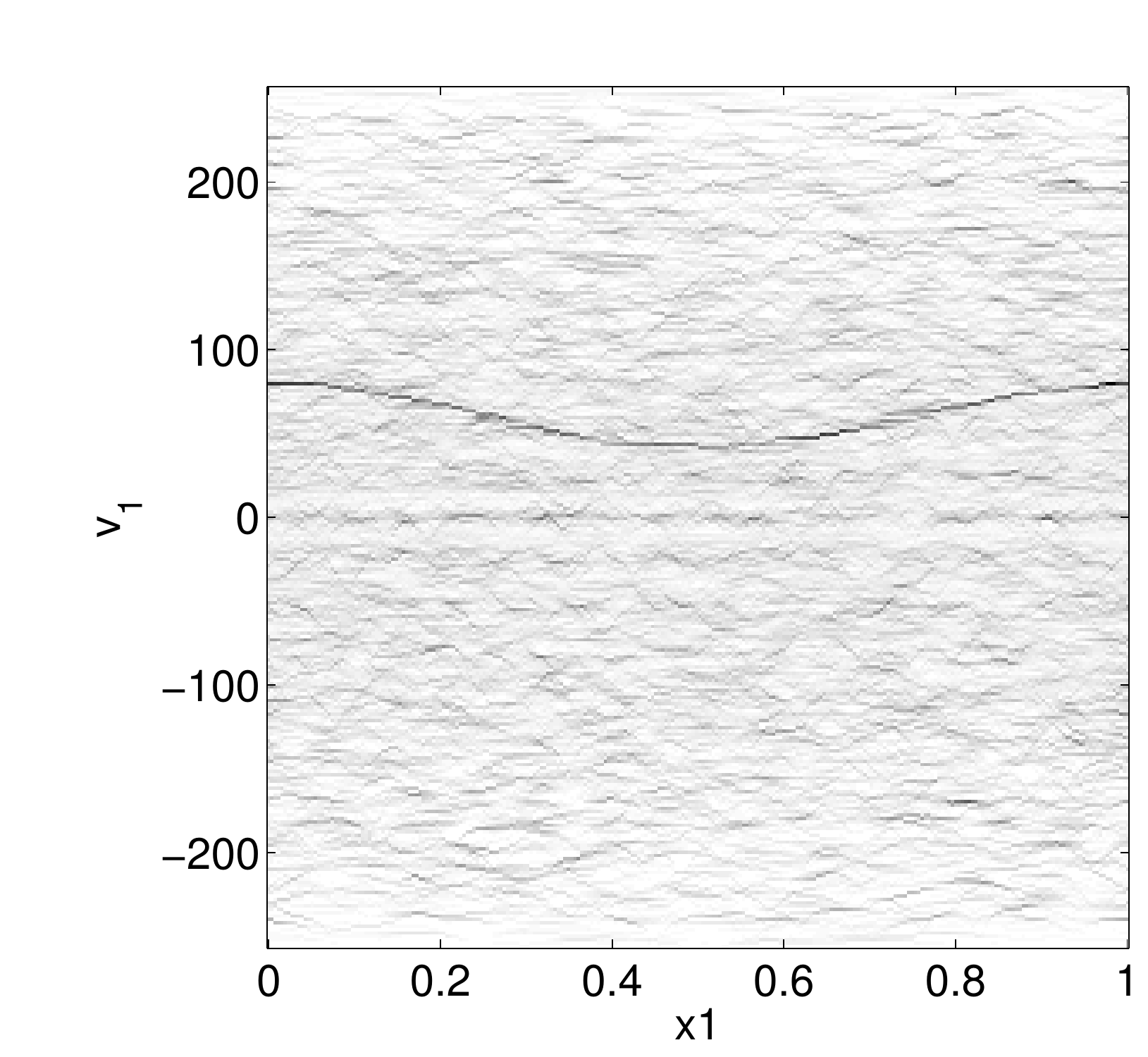}& \includegraphics[height=1.6in]{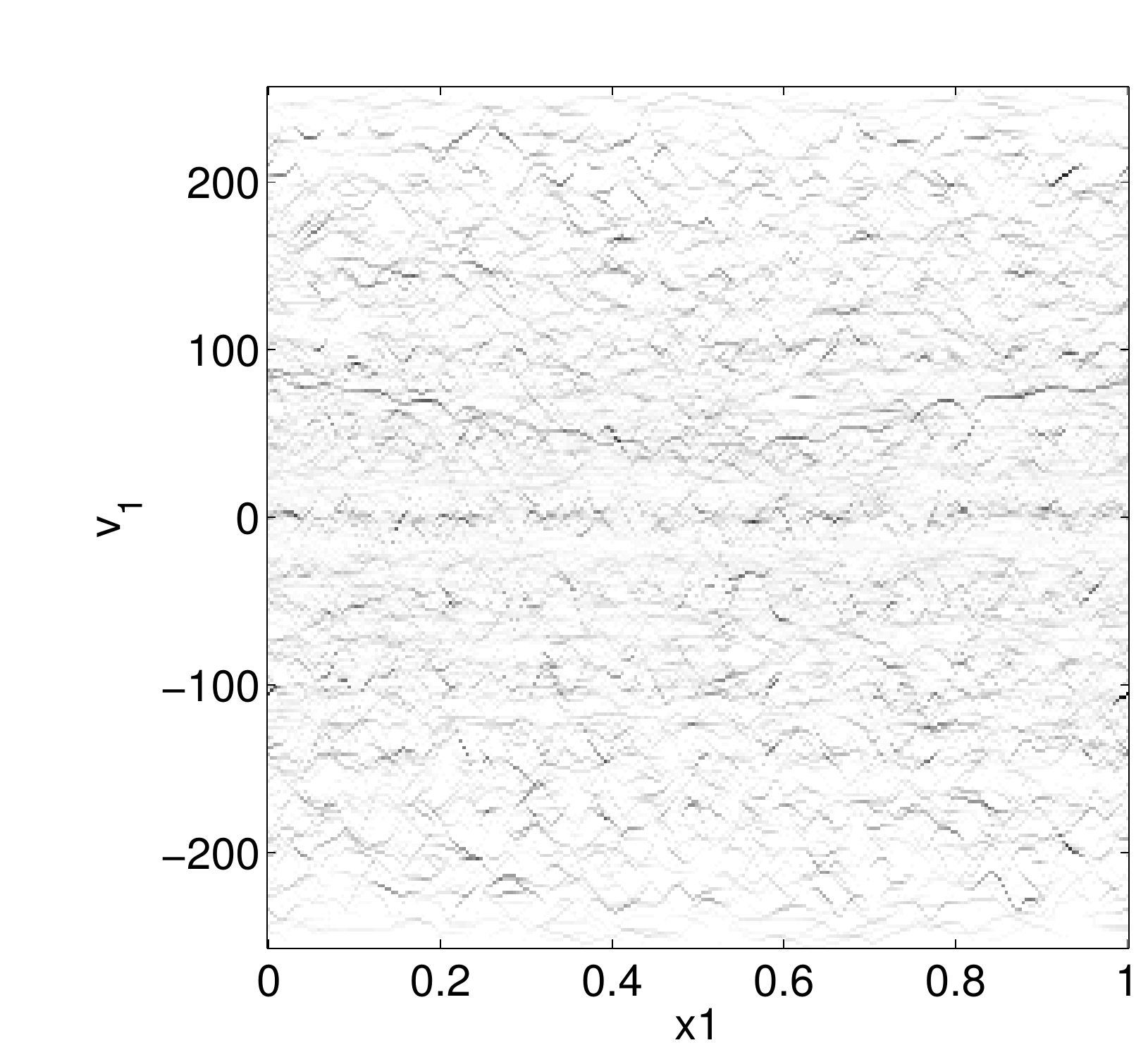}  & \includegraphics[height=1.6in]{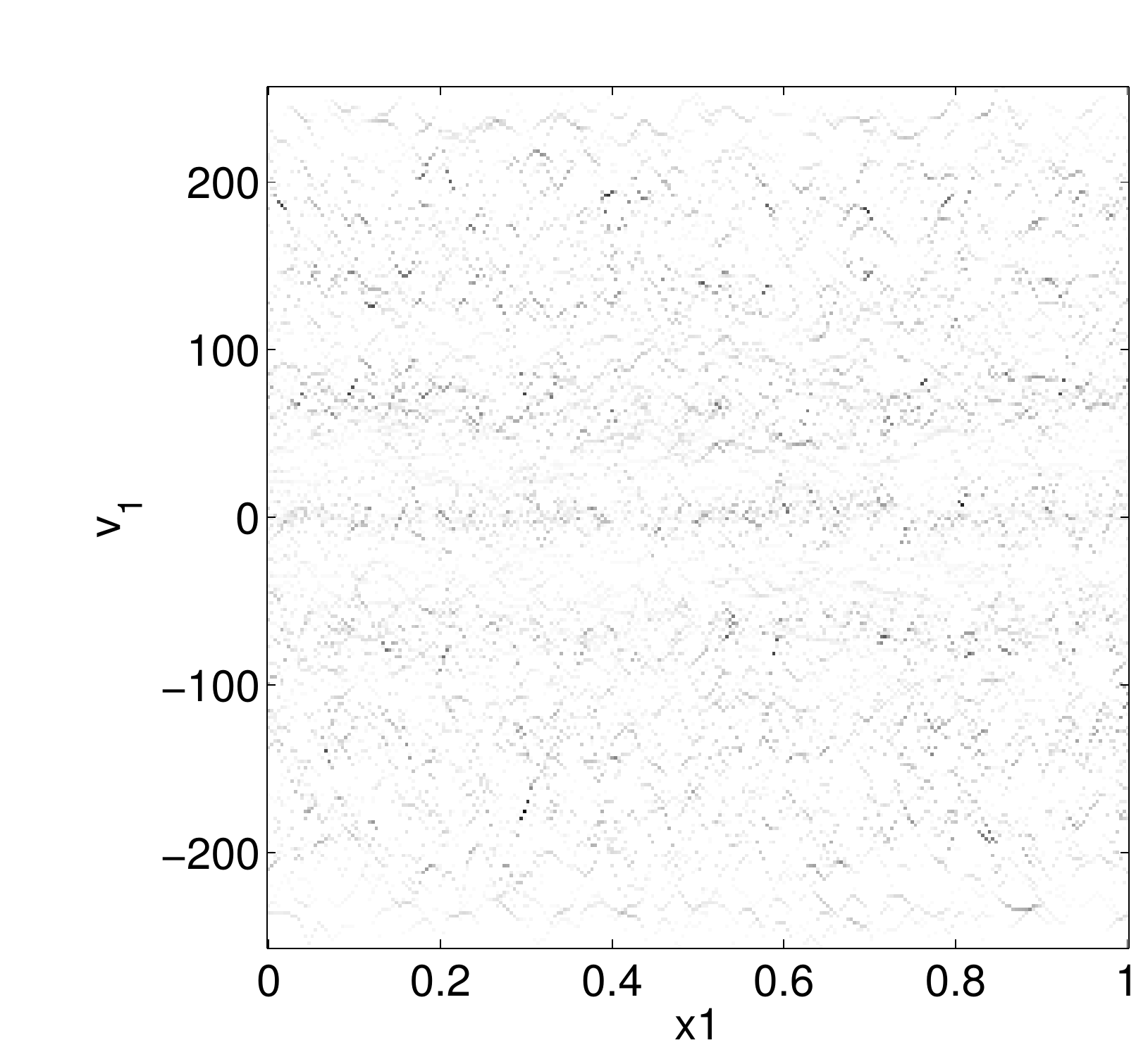}\\
   \includegraphics[height=1.6in]{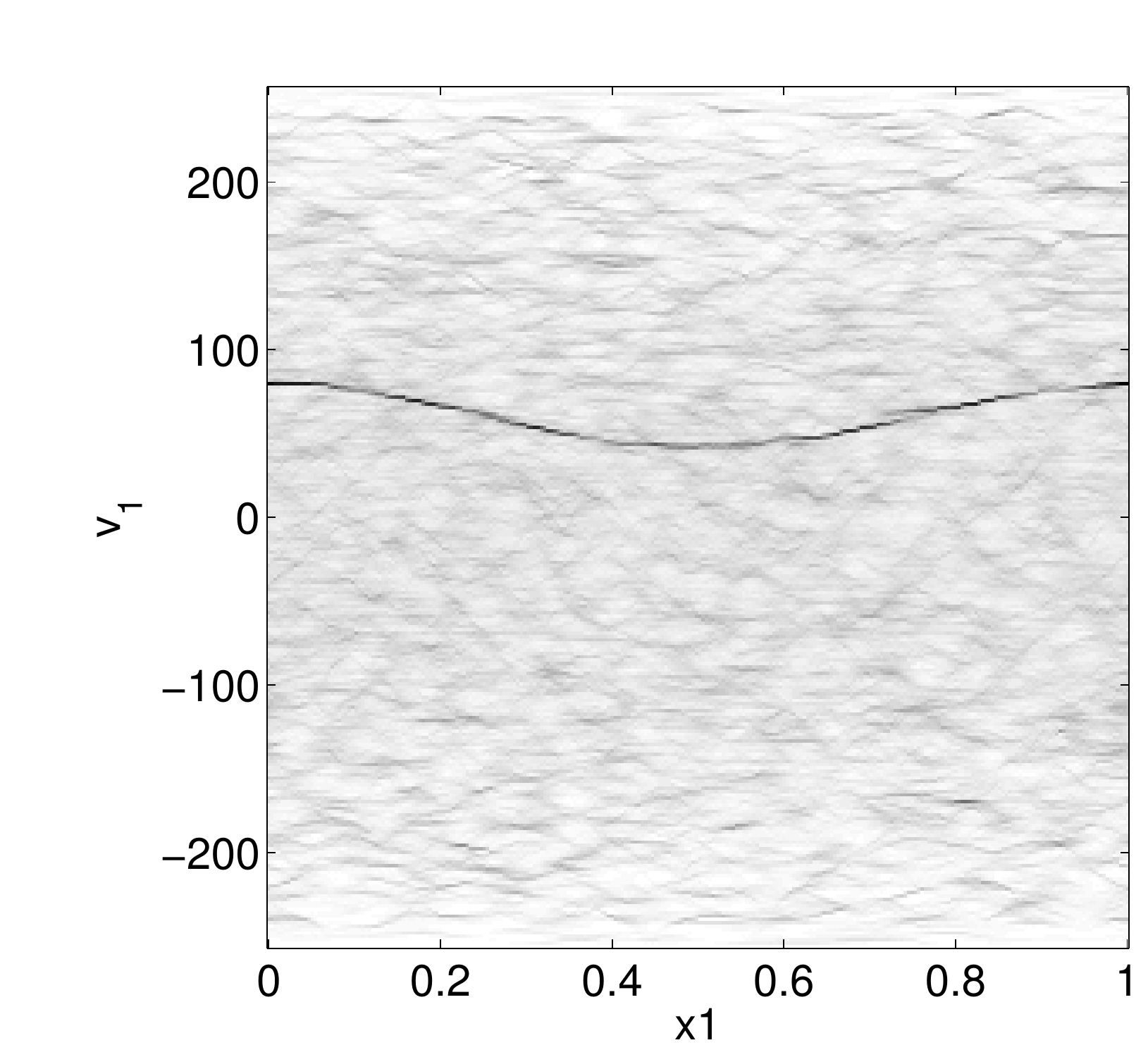}& \includegraphics[height=1.6in]{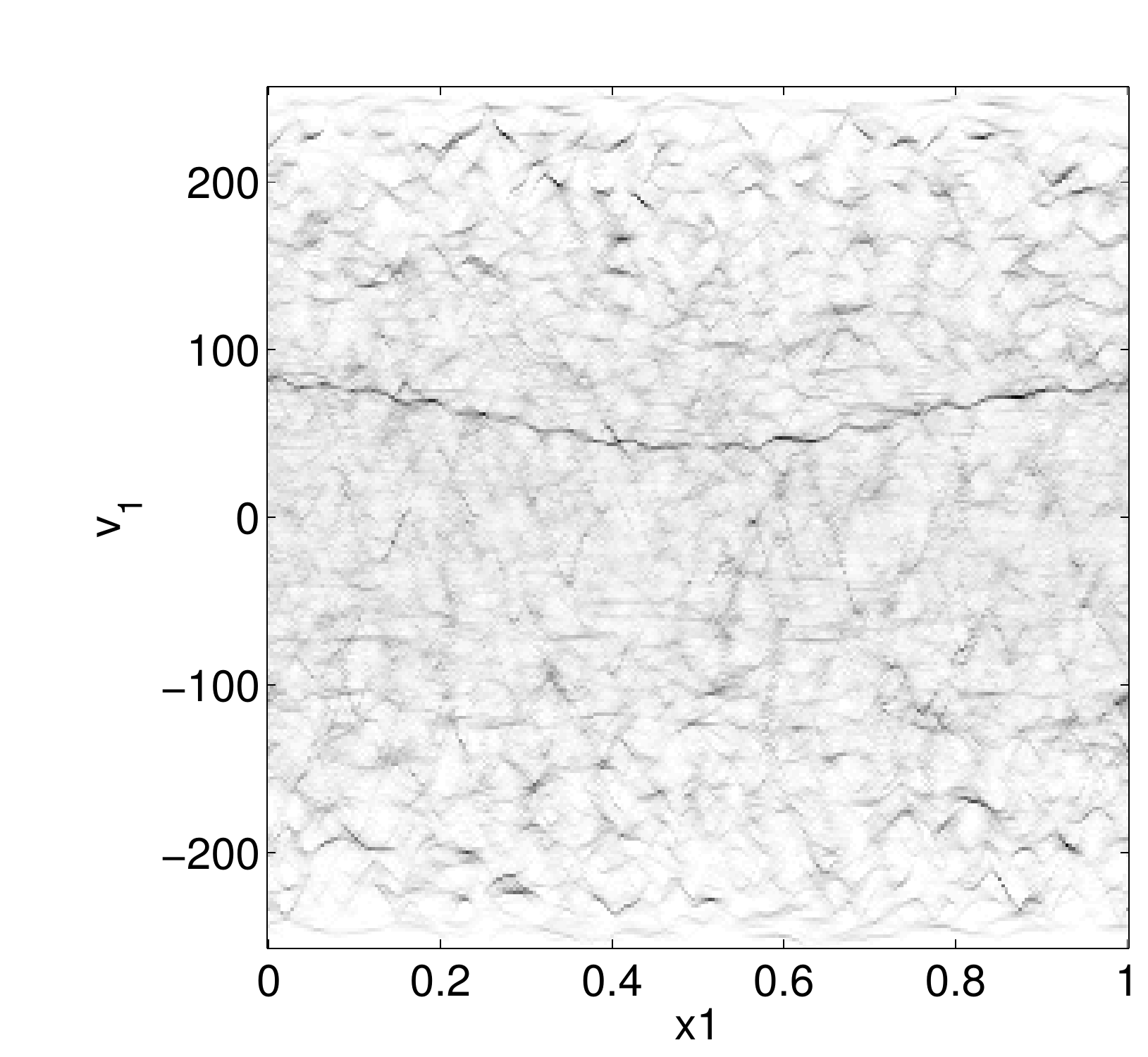}  & \includegraphics[height=1.6in]{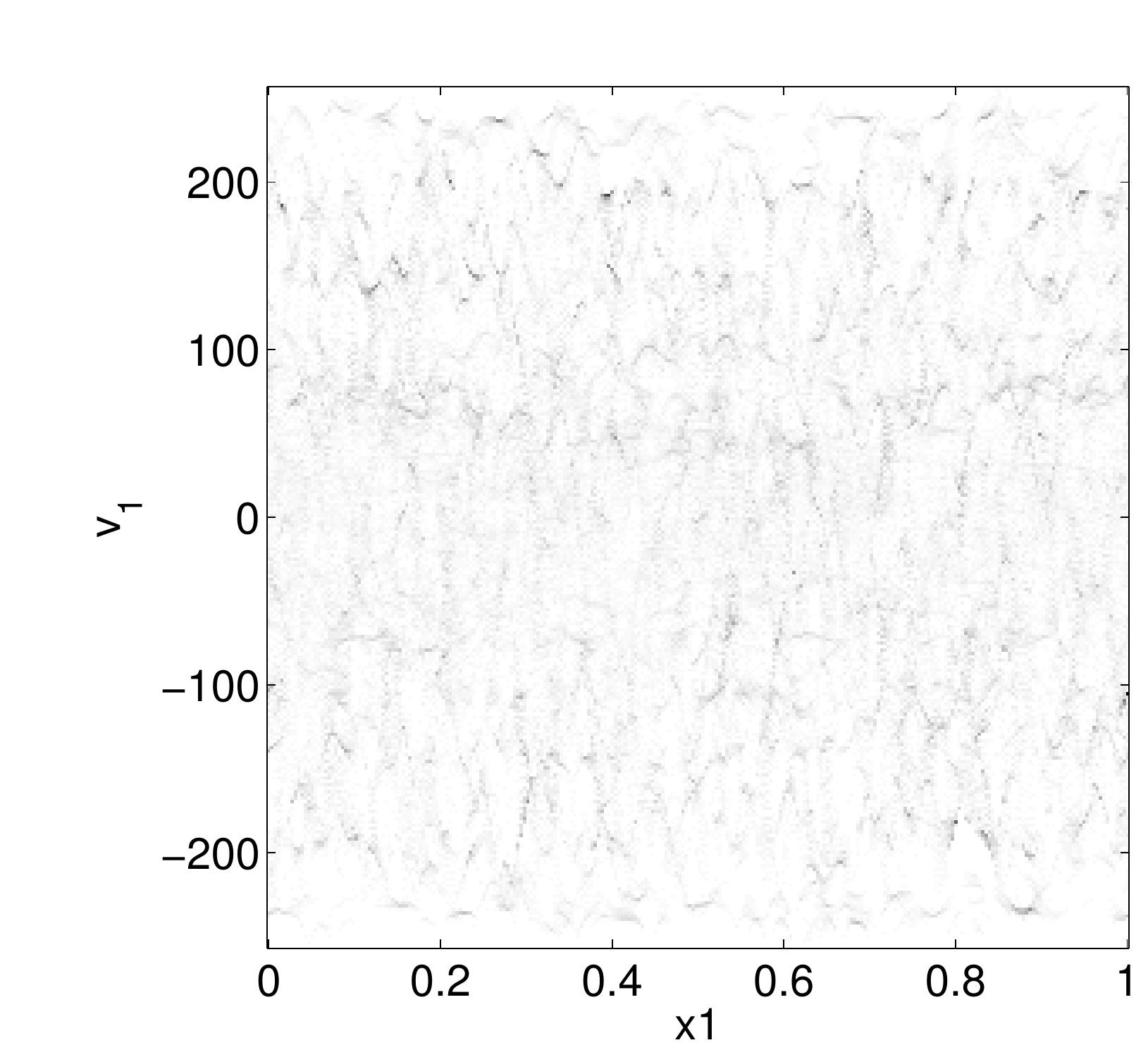}\\
      \includegraphics[height=1.6in]{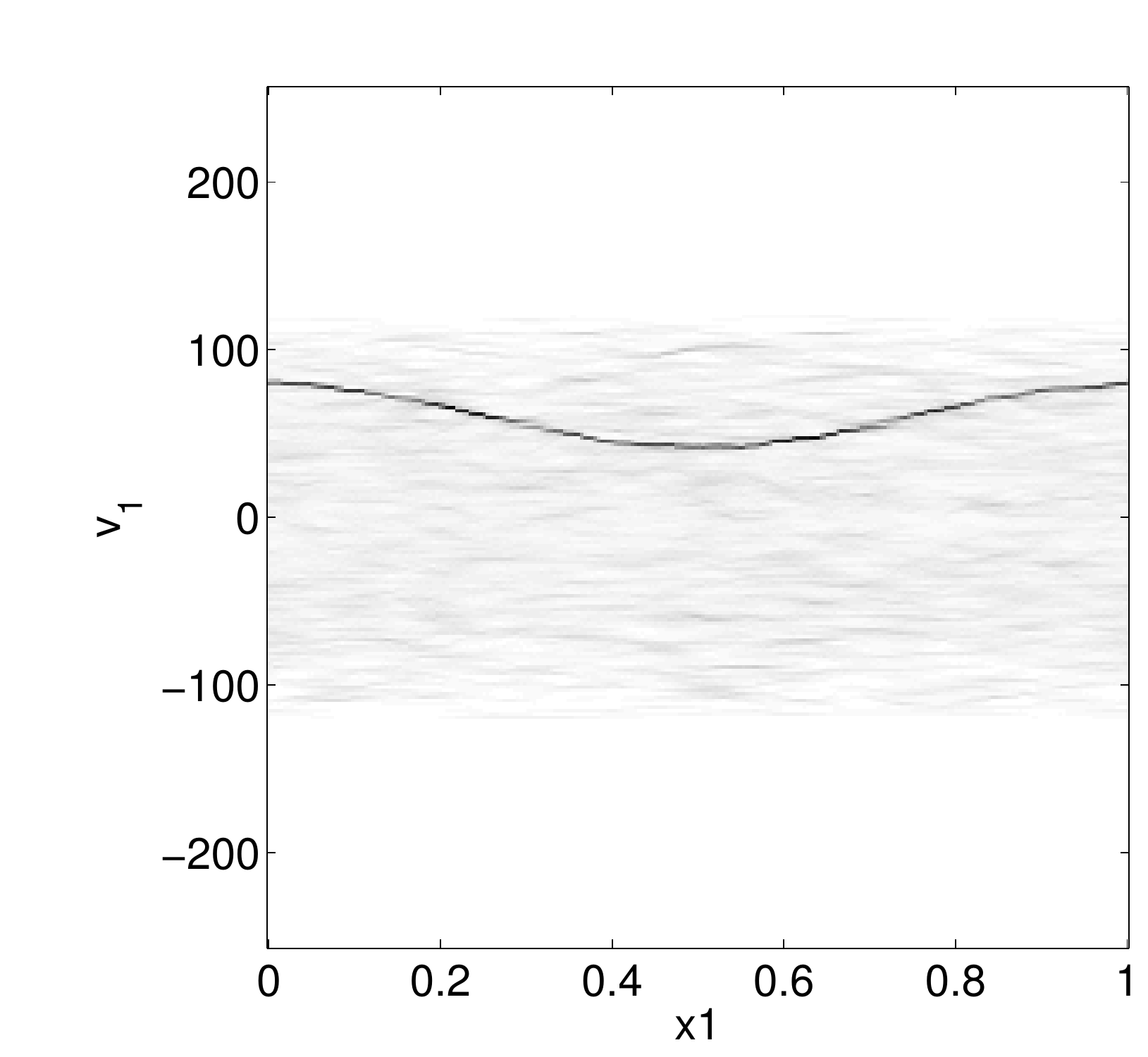}& \includegraphics[height=1.6in]{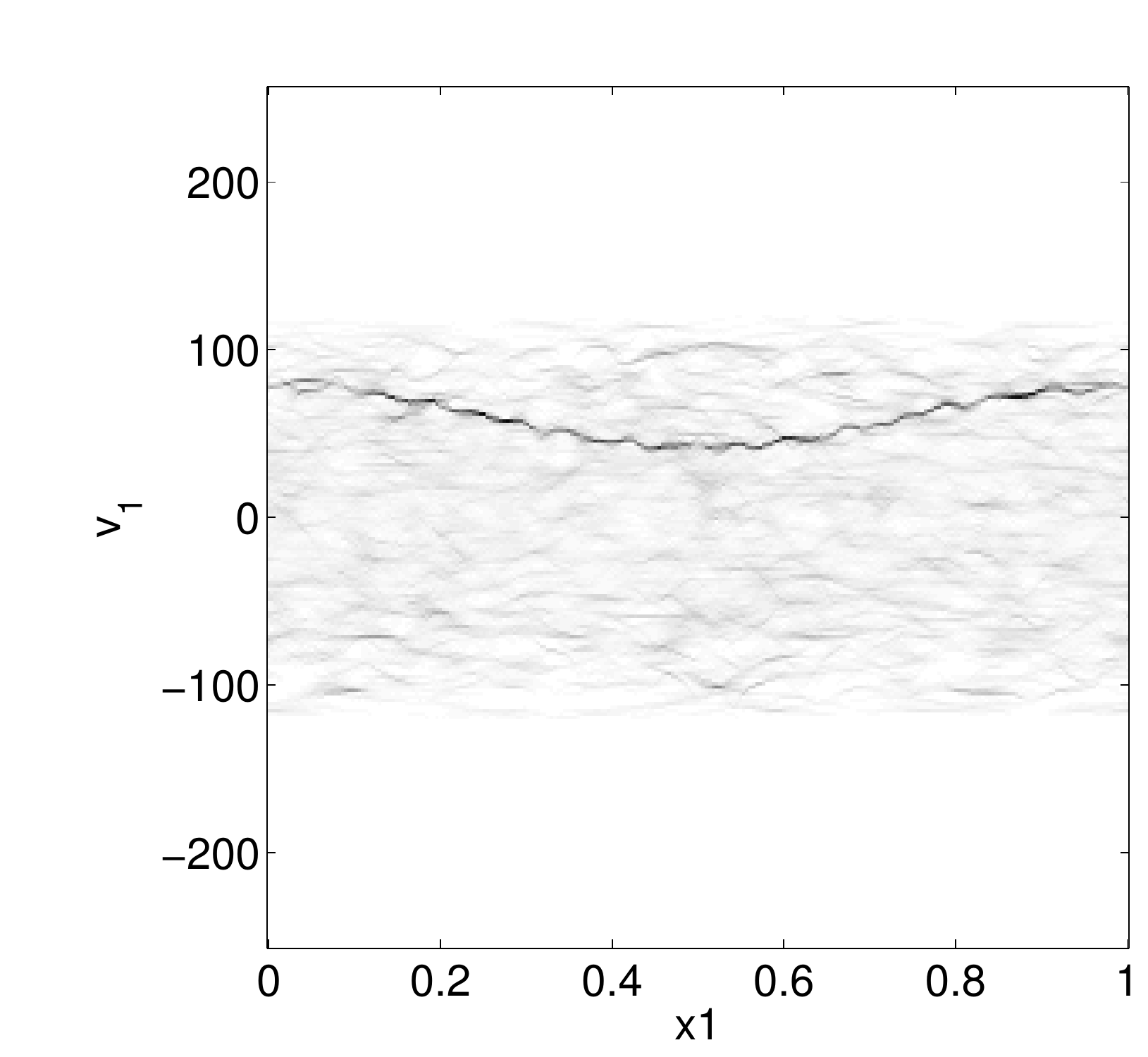}  & \includegraphics[height=1.6in]{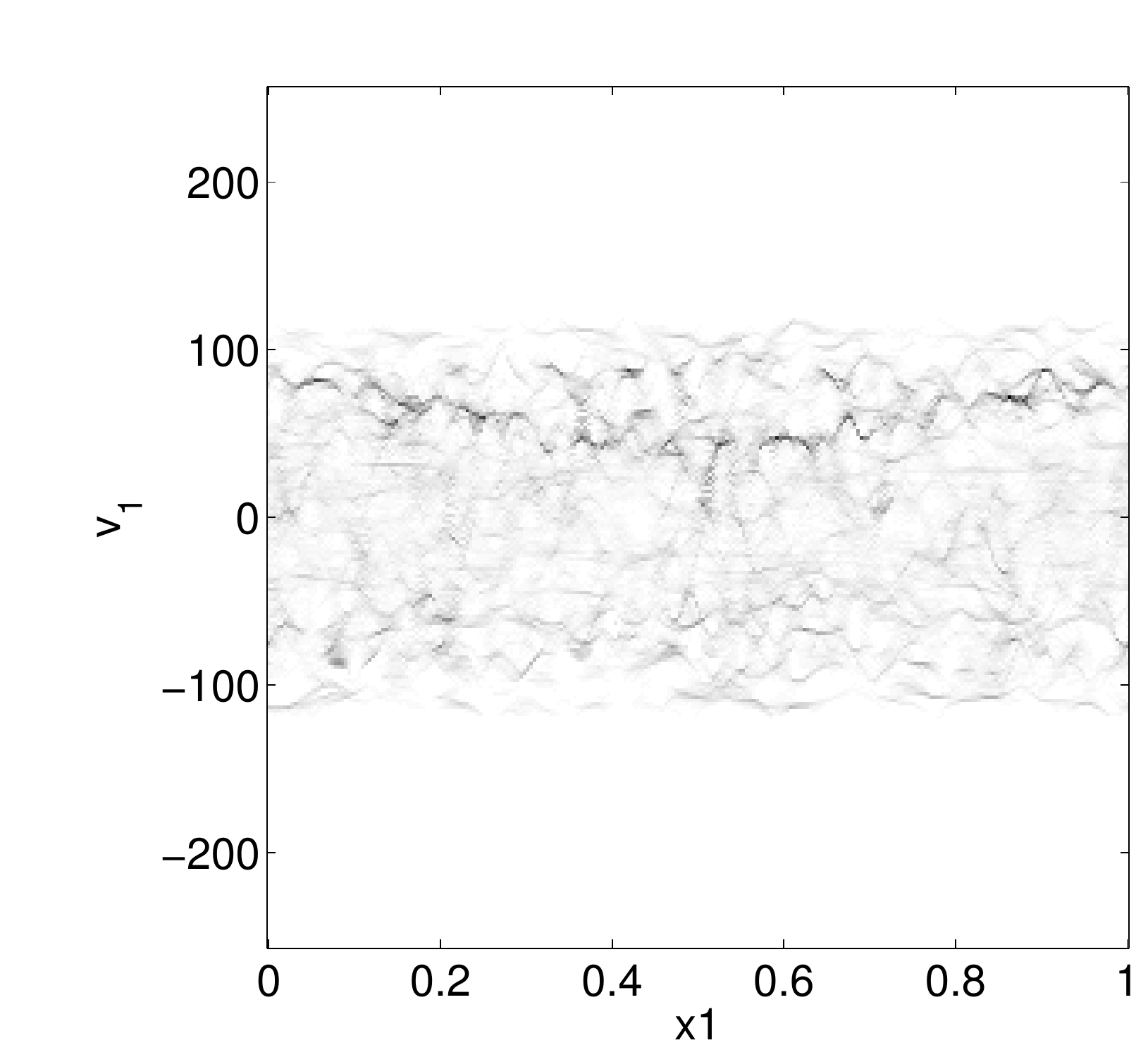}
    \end{tabular}
  \end{center}
  \caption{Stacked synchrosqueezed energy distribution $\int_{\R}T_f(v_1,v_2,x_1,0)dv_2$ of the noisy $2$D signal in Figure \ref{fig:2Dsignal}. From left to right, $s=0.625$, $0.75$ and $0.875$. From top to bottom: $red=1$, $red=10$, and $red=10$ with a restricted frequency band from $20$ to $120$ Hz.}
\label{fig:2D}
\end{figure}

\vspace{1cm}
\textbf{Component test}
\vspace{0.5cm}

Here we present an example to validate the last observation in Section \ref{subsec:idea}. Suppose we look at a region in the time-frequency or phase space domain and we know there might be only one IMT in this region. This assumption is reasonable because after the SST people might be interested in the synchrosqueezed energy in a particular region: is this corresponding to a component or just heavy noise? A straightforward solution is that, at each time or space grid point, we only reassign those coefficients with the largest magnitude. By Theorem \ref{thm:2d3}, if there is an IMT, we can obtain a sketch of its instantaneous frequency or  local wave vector with a high probability. If there was only noise, we would obtain random reassigned energy with a high probability. Using this idea, we apply the band-limited SSWPT with $s=0.625$ and $red=10$ to a noisy version of the image in Figure \ref{fig:2Dsignal} left. From left to right, Figure \ref{fig:max} shows the results of a noisy image \eqref{eqn:2deximg} with $5\mathcal{N}(0,1)$ noise, a noisy image \eqref{eqn:2deximg} with $10\mathcal{N}(0,1)$ noise, and an image with only noise, respectively. A reliable sketch of the  local wave vector is still visible even if the input image is highly noisy.

\begin{figure}[ht!]
  \begin{center}
    \begin{tabular}{ccc}
     \includegraphics[height=1.6in]{2Dm17red10IsBand1T1S1.pdf} &  \includegraphics[height=1.6in]{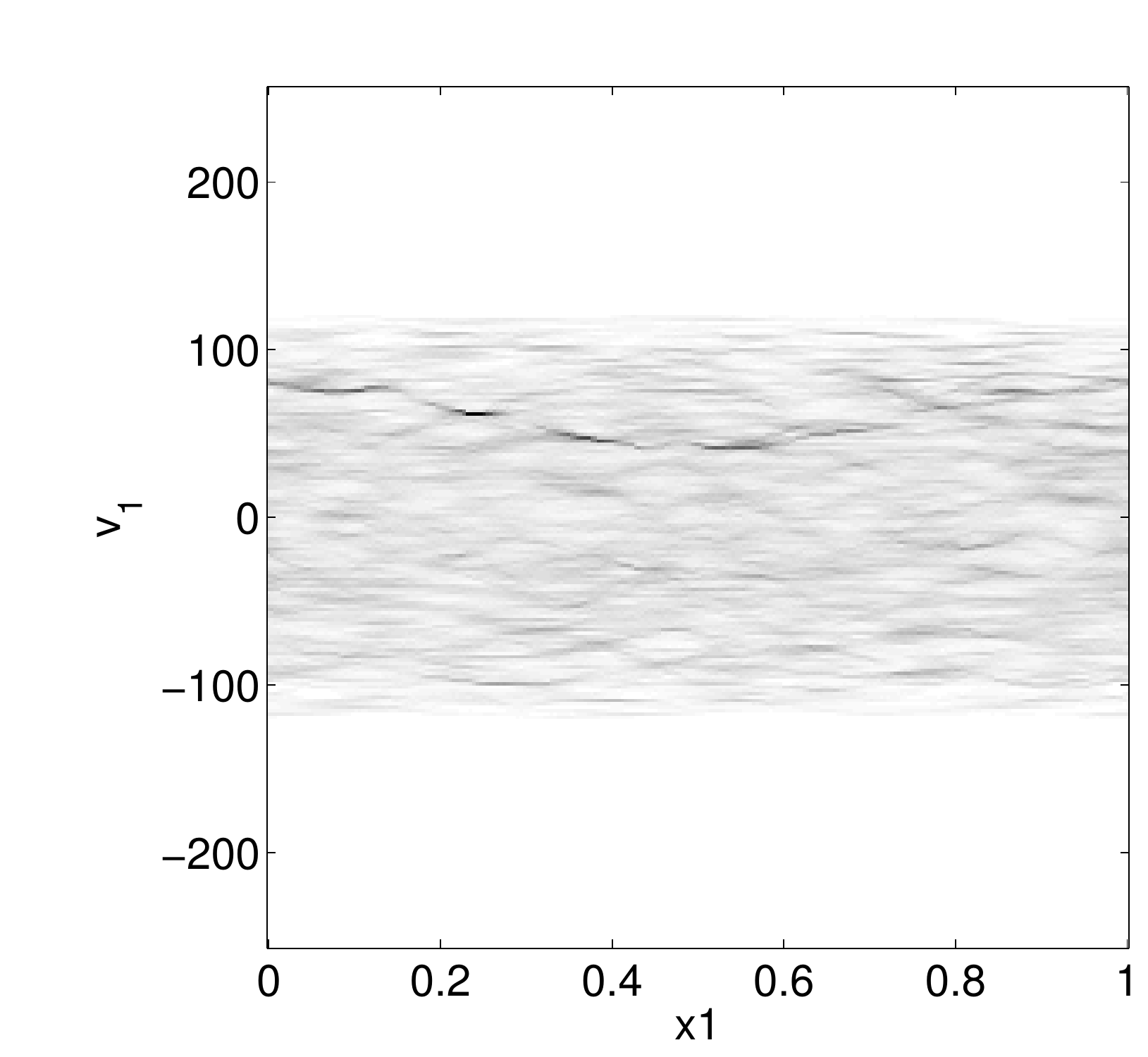} &  \includegraphics[height=1.6in]{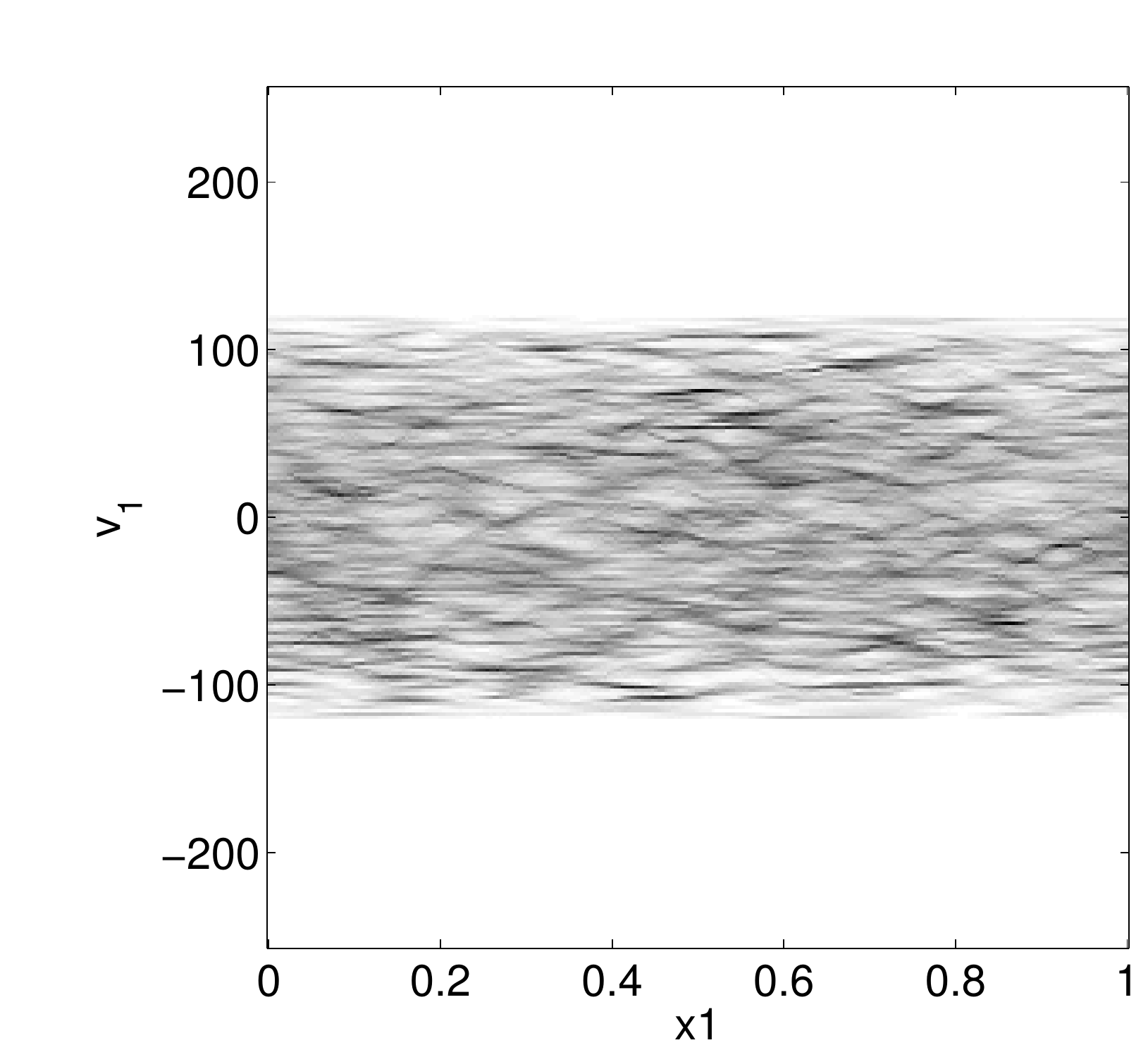} \\
 \includegraphics[height=1.6in]{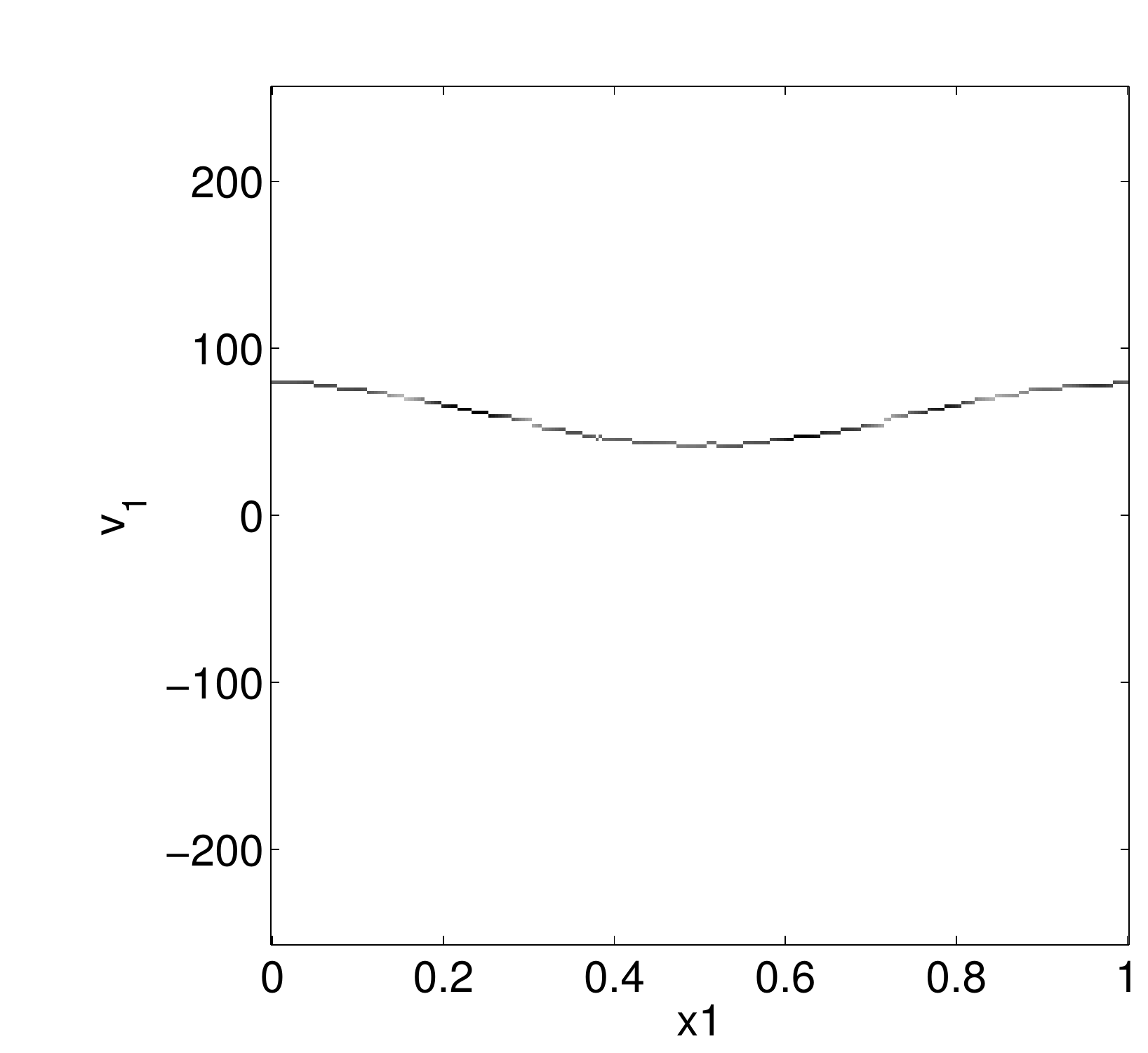} &  \includegraphics[height=1.6in]{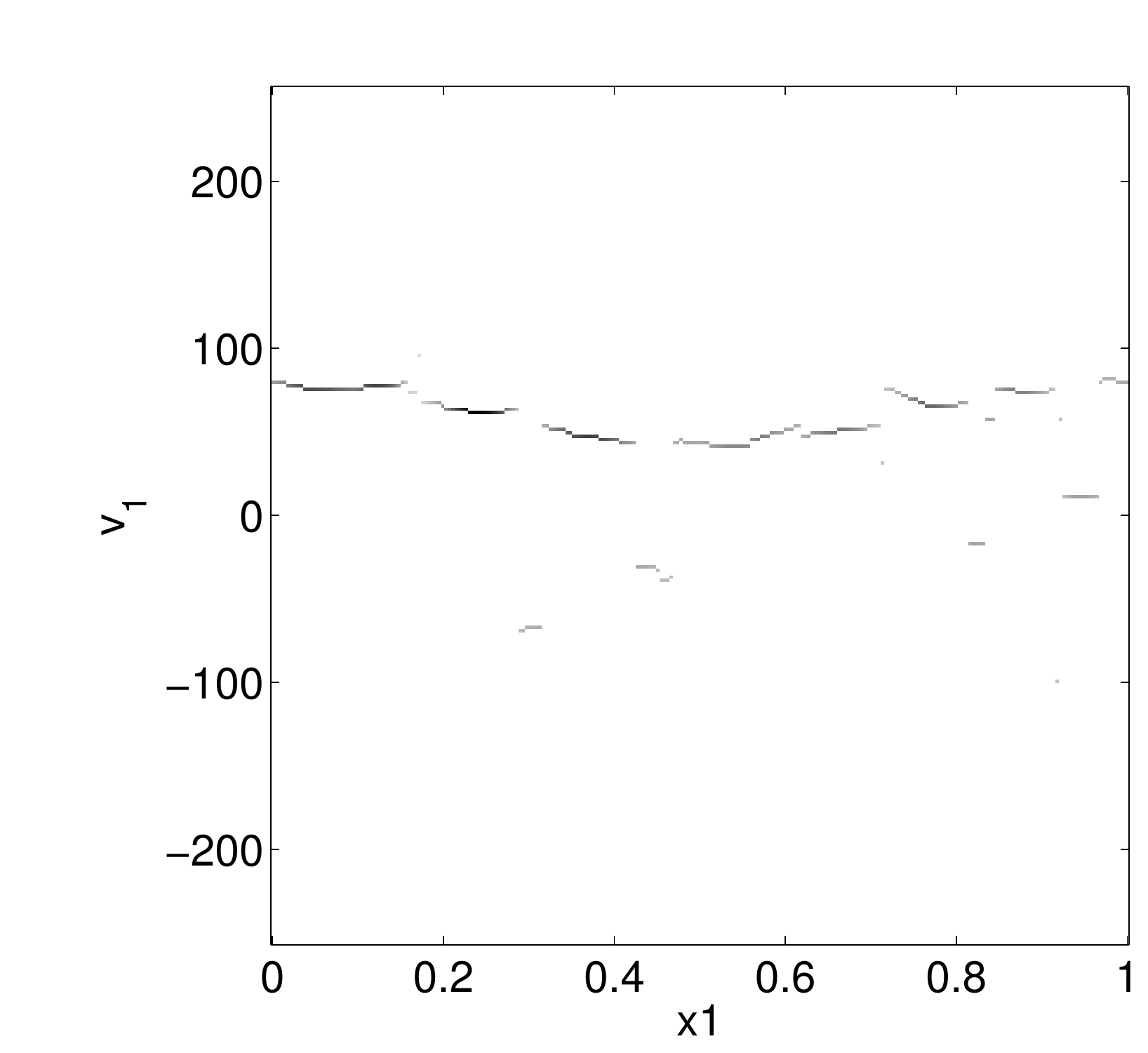} &  \includegraphics[height=1.6in]{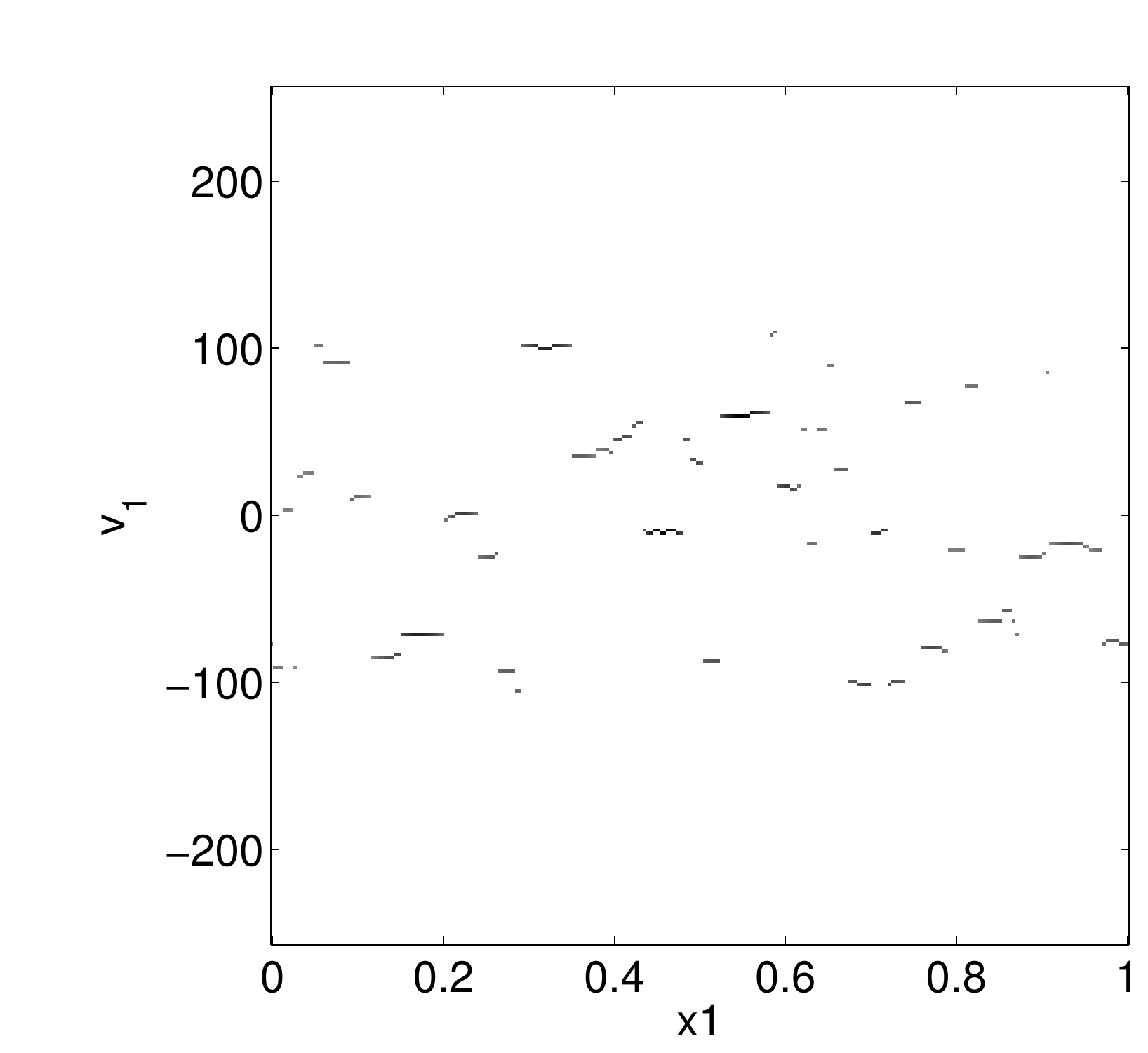} 
    \end{tabular}
  \end{center}
  \caption{Top row: The synchrosqueezed energy distribution of the highly redundant band-limited SSWPT with a frequency band $20$ to $120$ Hz. Bottom row: reassigned wave packet coefficients with the largest magnitude at a space location. Left column: $5\mathcal{N}(0,1)$ noise. Middle column: $10\mathcal{N}(0,1)$ noise. Right column: noise only.}
\label{fig:max}
\end{figure}
\subsubsection{Quantitative performance analysis}
We quantitatively analyze the performance of the highly redundant SSWPT and compare it with other methods in this subsection in terms of statistical stability and computational efficiency. Let us revisit the wave-like component in Figure \ref{fig:orgSST}, where 
\[
f(x)=e^{60\pi i(x+0.05\cos(2\pi x))}
\]
is sampled in $[0,1]$ with a sampling rate $1024$ Hz. Its instantaneous frequency is $q(x)=30(1-0.1\pi\sin(2\pi x))$. Let $g=f+e$ be the noisy data, where $e$ is white Gaussian noise with a distribution $\sigma^2\mathcal{N}(0,1)$. To measure noise, we introduce the following signal-to-noise ratio (\SNR) of the input data $g=f+e$:
\[
\SNR[dB](g)=10\log_{10}\left( \frac{\VAR (f)}{\VAR(e)}\right).
\]
An ideal time-frequency distribution of $f(x)$ is a function $D(v,x)$ such that $D(v,x)=\delta(v-q(x))$, where $\delta(v)$ is a Dirac delta function of $v$. 
To quantify the numerical performance of various methods, we introduce the Earth mover's distance (EMD) \cite{EMDD1,EMDD2,EMDD0} to measure the distance between a resultant time-frequency distribution $T(v,x)$ and the ideal distribution $D(v,x)$. At each $x$, after the discretization and normalization of $T(v,x)$, we obtain a $1$D discrete distribution $\tilde{T}(v,x)$. Similarly, we compute the discrete version $\tilde{D}(v,x)$ of $D(v,x)$. At each $x$, we compute the $1$D EMD between $\tilde{T}(v,x)$ and $\tilde{D}(v,x)$. The average EMD at all $x$ is defined as the distance (also denoted as EMD) between $T(v,x)$ and $D(v,x)$ in this paper. A smaller EMD means a better time-frequency concentration to the ground true instantaneous frequency and fewer noise fluctuations.

\begin{figure}[ht!]
  \begin{center}
    \begin{tabular}{ccc}
   \includegraphics[height=1.6in]{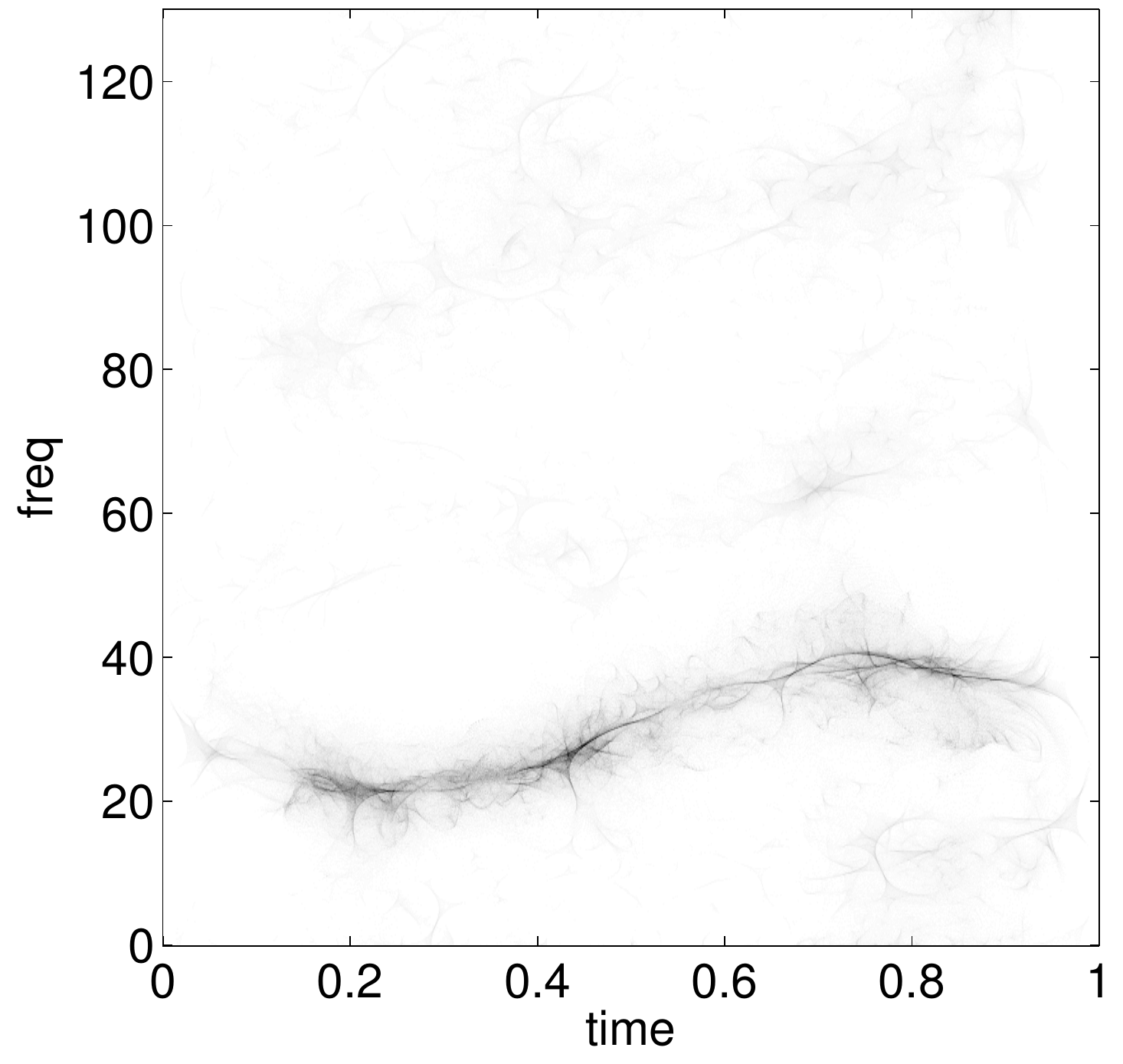}& \includegraphics[height=1.6in]{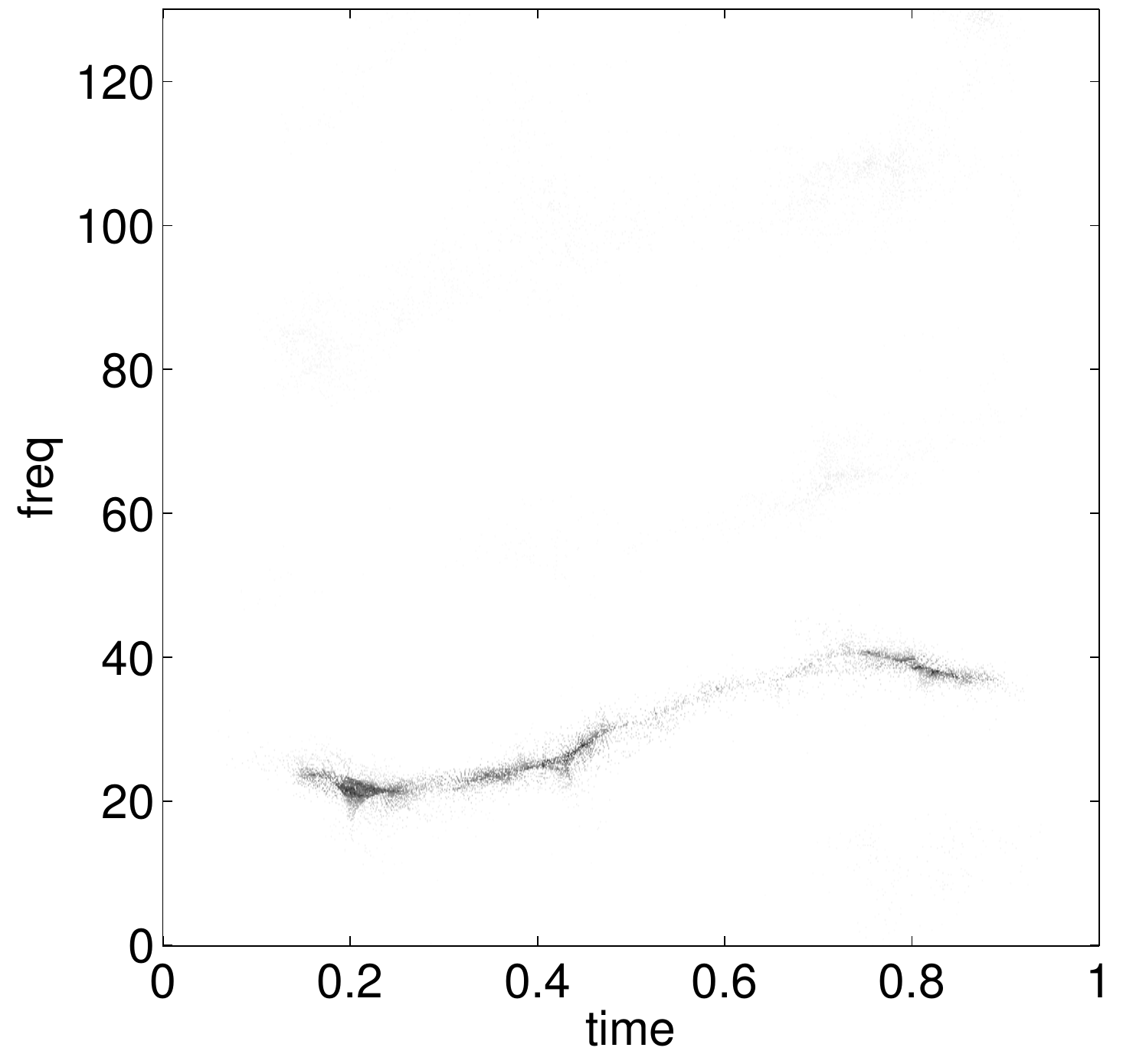} & \includegraphics[height=1.6in]{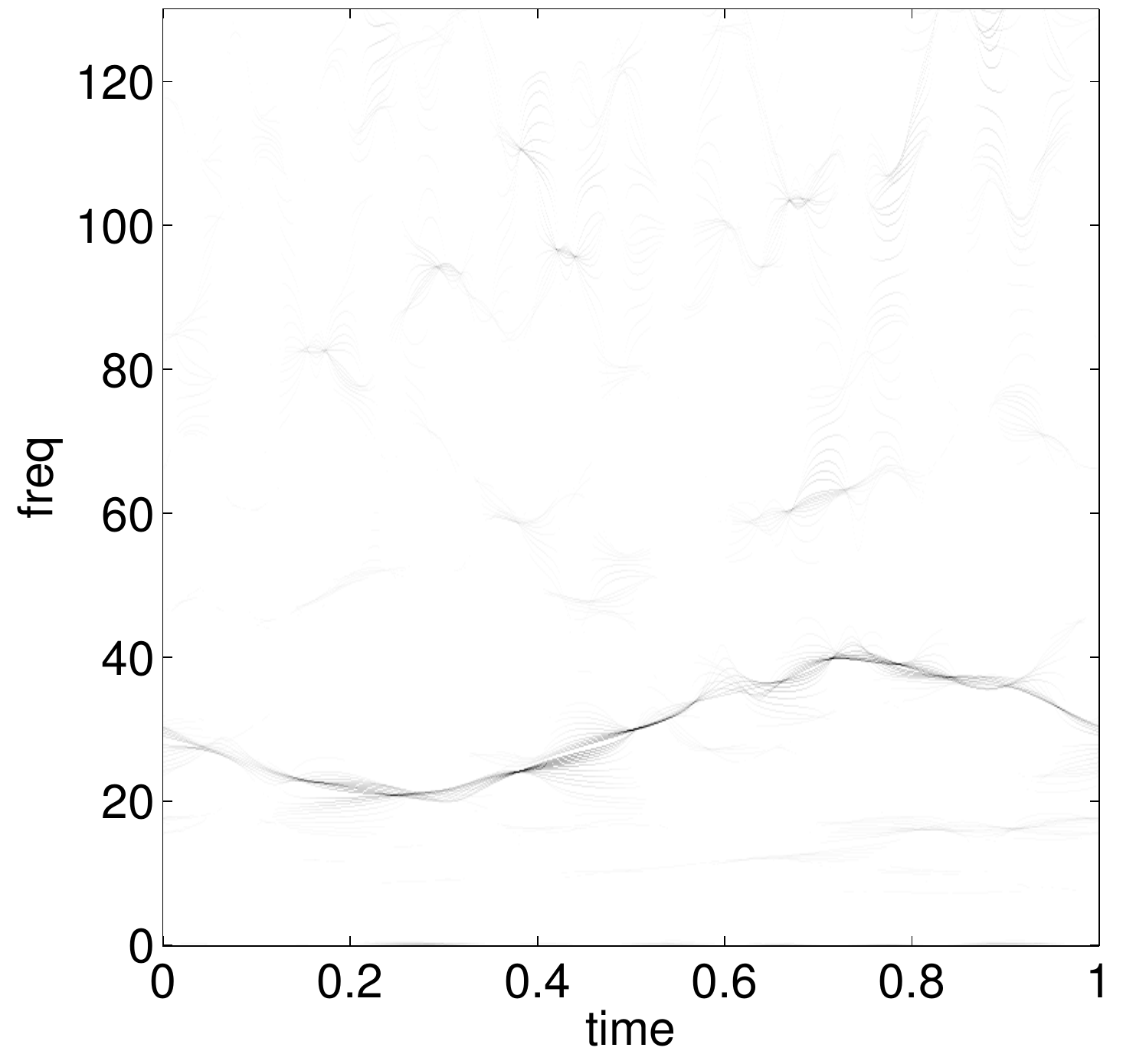} 
    \end{tabular}
  \end{center}
  \caption{Left and middle: The multitaper time-frequency reassignment using an arithmetic mean (MRSA) and a geometric mean (MRSG) with $10$ tapers. Right: The highly redundant SSWPT with $s=0.75$ and $red=10$.}
\label{fig:MTcomp}
\end{figure}

First, we compare the performance of the SSWPT with different redundancy parameters $red$ for noisy data with different $\SNR$s. We compare the EMD between their resultant time-frequency distributions and the ideal one. Figure \ref{fig:Qcomp} (top-left) shows the EMD as functions in the variable of $red$. The EMD functions decrease fast when the $red$ increases to $10$. All the SSWPTs with $red\geq 10$ have almost the same efficiency. 

Second, we compare the performance of the SSWPT with the same redundancy parameters $red=10$ but different geometric scaling parameter $s$ for noisy data with different noise level, e.g. $\sigma^2$ ranging from $0$ to $4$, i.e., $\SNR$ from $\infty$ to $-15$. We compare the EMD between their resultant time-frequency distributions and the ideal one. Figure \ref{fig:Qcomp} (top-right) shows the EMD as functions in the variable of $\sigma^2$ for different $s$. It shows that when the data is clean, the SSWPT with $s=1$, which is essentially the SSWT with high redundancy, has the best performance. However, once the data is noisy, the SSWPT with $s=0.75$ has the best overall performance. This observation agrees with our theorems in Section \ref{sec:thm}. If $s$ is larger, the instantaneous frequency estimation is more accurate if there is no noise. But the probability for a good estimation is smaller if there is noise. If $s$ is smaller , the estimation accuracy is worse but the probability for a good estimation is higher. In theory, it is interesting to find an optimal $s$ that makes a good balance. In practice, a heuristic choice for the optimal $s$ is $0.75$. 

Third, we compare the performance of the standard SSWT, SSSTFT in \cite{code} and the highly redundant SSWPT. Again, we choose the parameters in the example of Figure \ref{fig:orgSST} because they result in good visualization of time-frequency distribution. These transforms are applied to noisy data with different noise level. We compare the EMD between their resultant time-frequency distributions and the ideal one. The comparison is shown in Figure \ref{fig:Qcomp} (bottom-left). The highly redundant SSWPT has better time-frequency concentration than the standard SSWT and SSSTFT in all examples, even if in the noiseless case.

Finally, we compare the multitaper time-frequency reassignment using an arithmetic mean (MRSA) and a geometric mean (MRSG) in\cite{Taper4} with the highly redundant SSWPT. A MATLAB package of the MRSA and the MRSG is available on the authors' homepage. The time-frequency distribution of these three methods are visualized in Figure \ref{fig:MTcomp}. We visualize their results using the same discrete grid in the time-frequency domain. To make a fair comparison, the number of tapers are chosen to be $10$ for the MRSA and the MRSG, and the redundancy parameter $red$ is $10$ for the SSWPT with $s=0.75$. For one realization of this experiment, the running time for the MRSA or the MRSG is $155.13$ seconds, while the running time for the SSWPT is only $0.43$ seconds. A visual inspection of Figure \ref{fig:MTcomp} shows that the SSWPT also outperforms the MRSA and the MRSG: clear spectral energy at the boundary and better time-frequency concentration. To quantify this comparison, we apply the EMD again and the results are shown in Figure \ref{fig:Qcomp} (bottom-right). In most cases, the SSWPT with $s=0.75$ gives a smaller EMD. When $\sigma^2$ is near $4$, i.e., the $\SNR$ is near $-15$,  the MRSG with $10$ tapers and the SSWPT have comparable performance.

\begin{figure}[ht!]
  \begin{center}
    \begin{tabular}{cc}
   \includegraphics[height=2in]{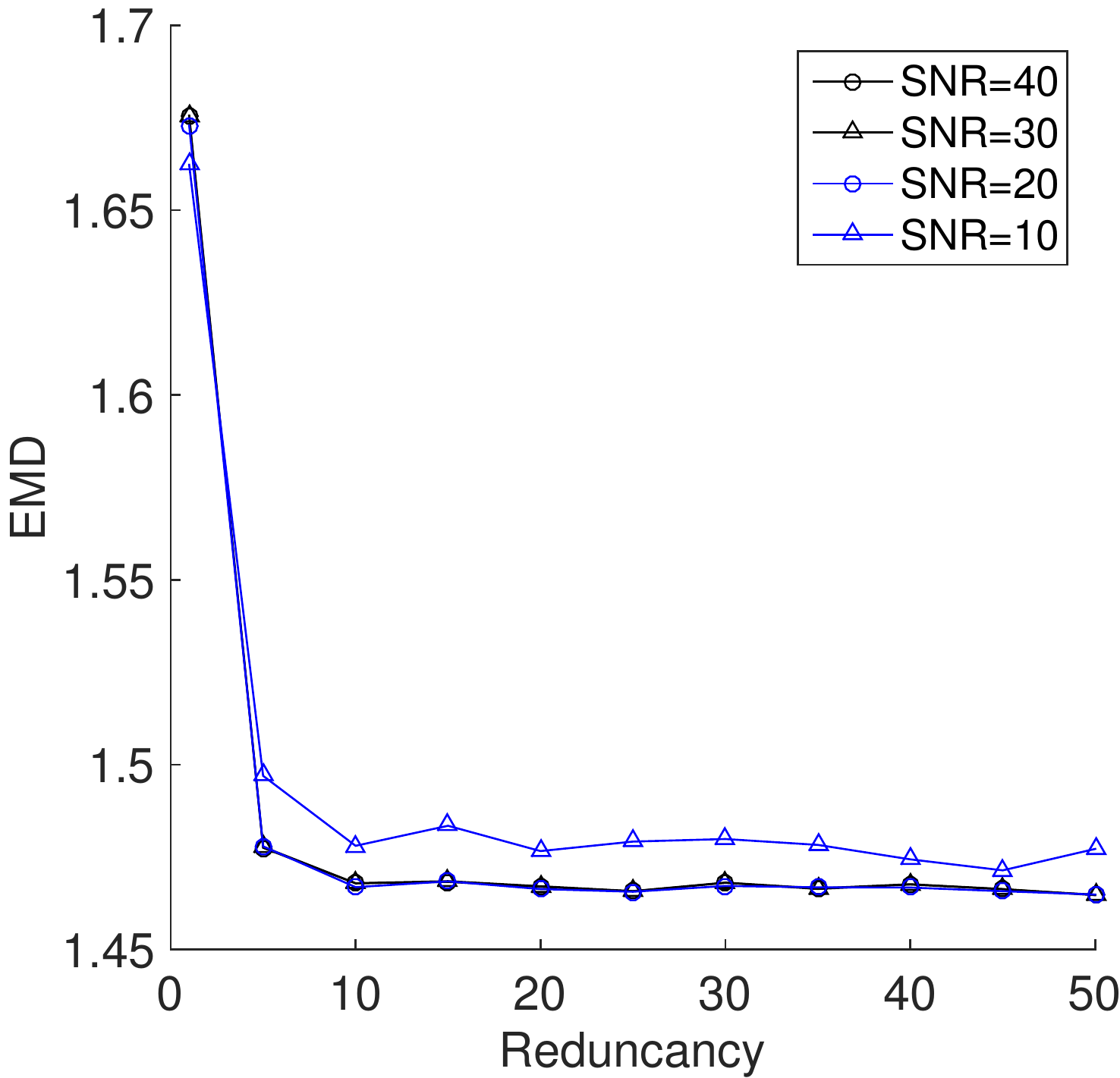}& \includegraphics[height=2in]{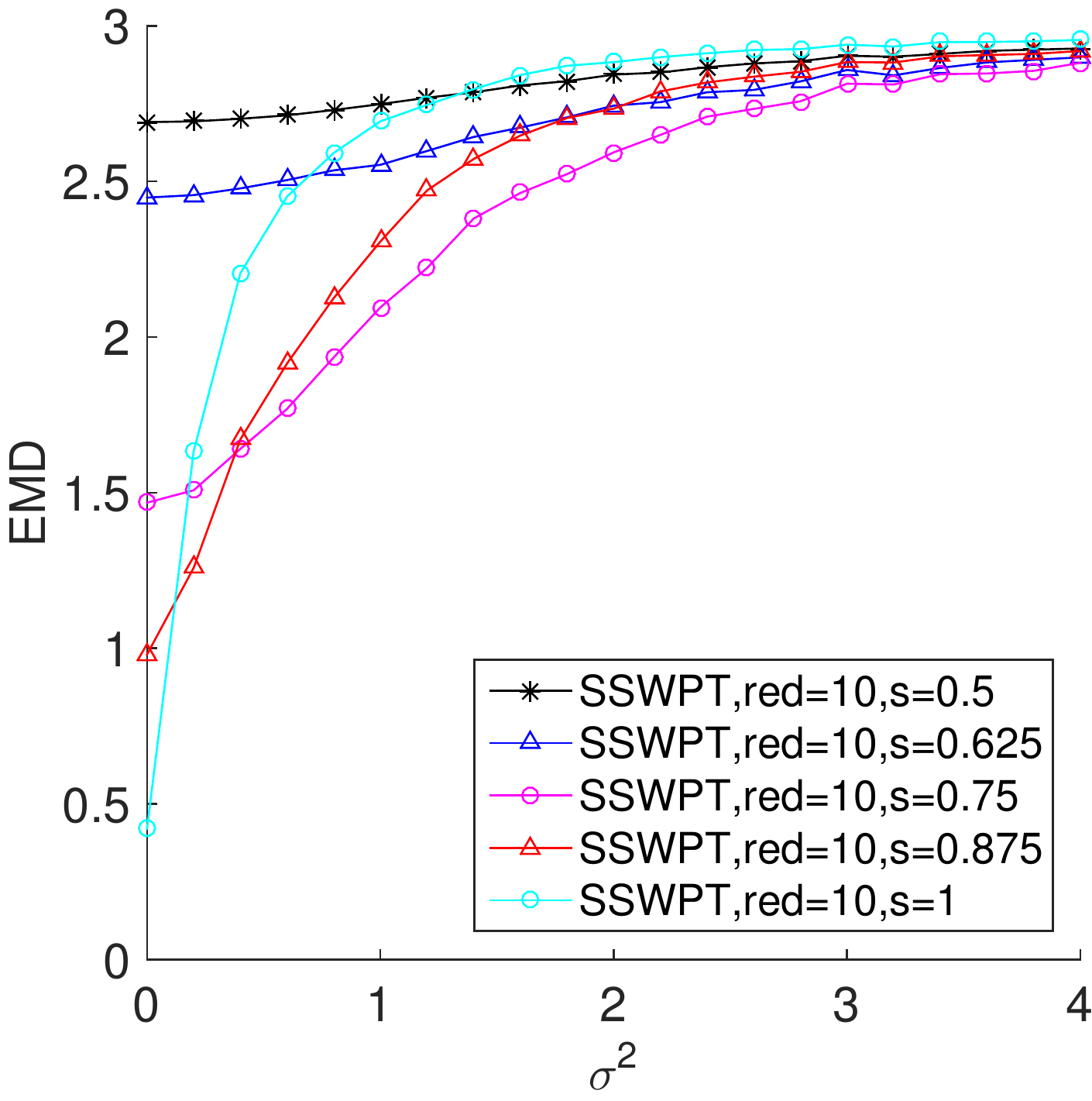}\\
     \includegraphics[height=2in]{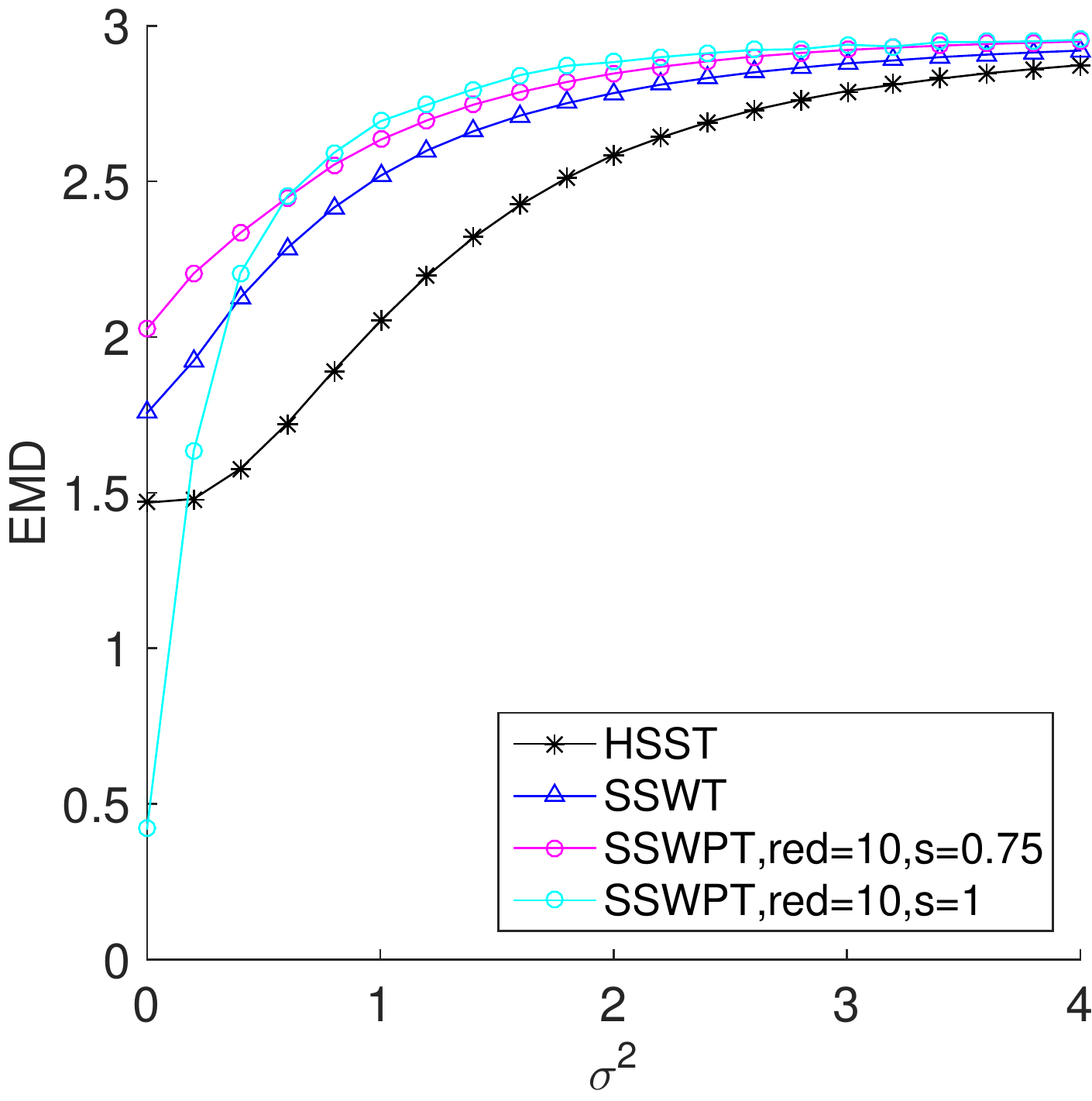} &\includegraphics[height=2in]{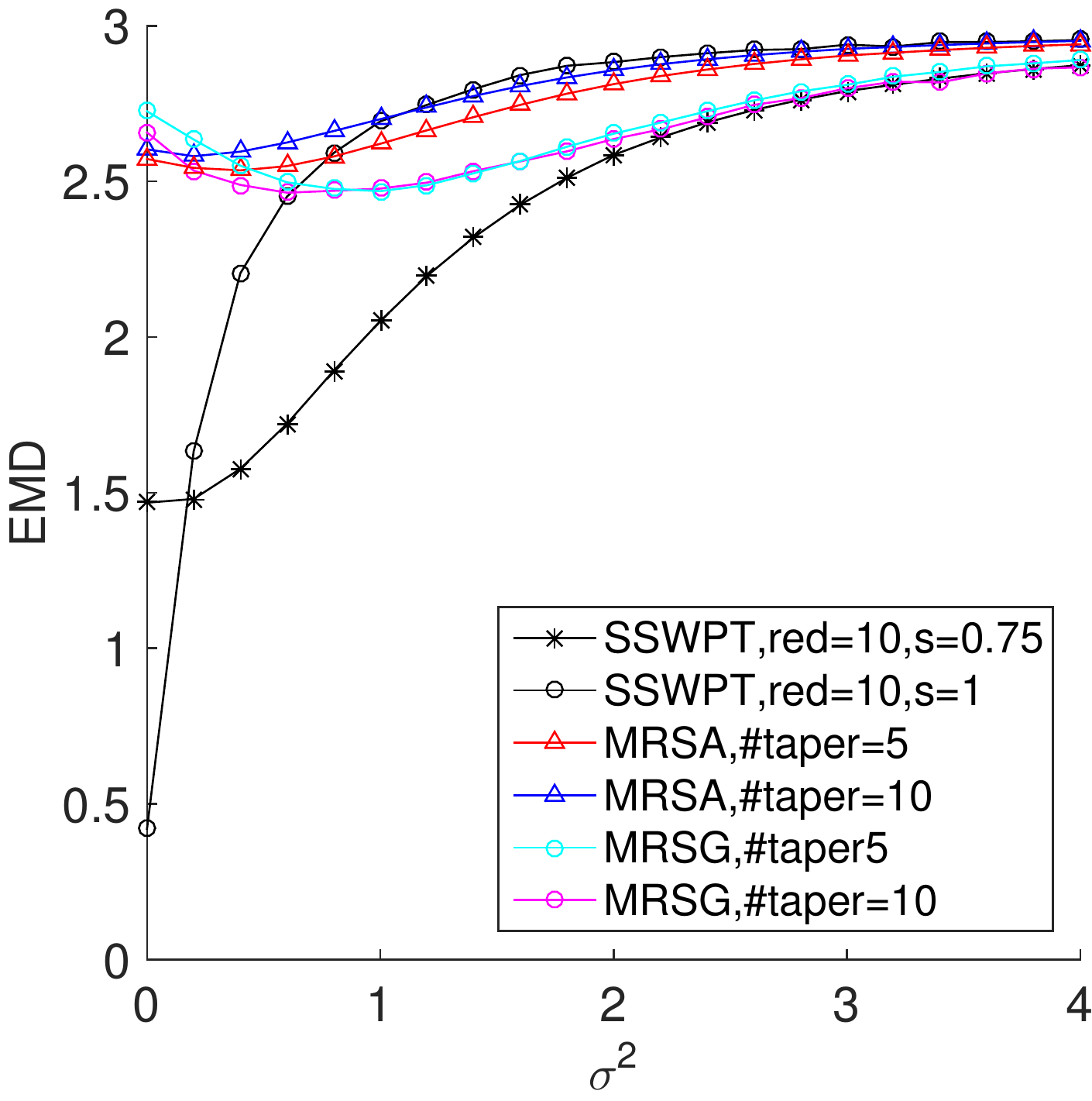} 
    \end{tabular}
  \end{center}
  \caption{Top-left: The EMD as functions in the variable of $red$ for different SNRs. Bottom-left: Comparison of the standard SSWT, SSSTFT and the SSWPT. Bottom-right: Comparison of the multitaper time-frequency reassignment (the numbers of tapers are $5$ and $10$) and the SSWPT. Data plotted above is the averaging of $20$ independent realizations.}
\label{fig:Qcomp}
\end{figure}

\section{Conclusion}
\label{sec:conclusion}

In theory, the statistical analysis in this paper has analyzed the statistical properties of a wide range of compactly supported synchrosqueezed transforms in multidimensional spaces, considering zero mean stationary Gaussian random process and small perturbation. Guided by these properties, this paper has presented several approaches to improve the performance of these synchrosqueezed transforms under heavy noise. A MATLAB package SynLab for these algorithms is available at \url{https://github.com/HaizhaoYang/SynLab}.

{\bf Acknowledgments.} 
This work was partially supported by the National Science
   Foundation under award DMS-0846501 and the U.S. Department of
   Energy’s Advanced Scientific Computing Research program under award
   DE-FC02-13ER26134/DE-SC0009409 for Lexing Ying. H. Yang is also supported by a AMS Simons travel grant. He thanks Jean-Baptiste Tary and Mirko van der Baan for the discussion of benchmark signals, thanks Lexing Ying, Charles K. Chui, Jianfeng Lu and Segey Fomel for the discussion of the implementation and the application of synchrosqueezed transforms.

\bibliographystyle{abbrv} \bibliography{ref}

\end{document}